\documentclass[dvips,ssy, preprint]{imsart}

\arxiv{math.PR/0000000}

\RequirePackage[OT1]{fontenc}
\RequirePackage{amsthm,amsmath,enumerate}
\RequirePackage[colorlinks]{hyperref}
\RequirePackage[numbers]{natbib}
\RequirePackage{hypernat}
\RequirePackage{amssymb,epsfig}

\startlocaldefs

\numberwithin{equation}{section}

\newtheorem{lemma}{Lemma}[section]
\newtheorem{theorem}{Theorem}[section]
\newtheorem{definition}{Definition}[section]
\newtheorem{corollary}{Corollary}[section]

\newtheorem{remark}{Remark}[section]

\newcommand{\RR}{{\mathbb R}}
\newcommand{\rS}{{\mathbb S}}
\newcommand{\rA}{{\mathbb A}}
\newcommand{\rB}{{\mathbb B}}
\newcommand{\ZZ}{{\mathbb Z}}

\newcommand{\eeq}{\end{equation}}

\newcommand{\ra}{\rightarrow}

\newcommand{\beql}[1]{\begin{equation}\label{#1}}
\newcommand{\eqn}[1]{(\ref{#1})}
\newcommand{\beq}{\begin{displaymath}}
\newcommand{\eeqno}{\end{displaymath}}

\newcommand{\af}{\alpha}

\newcommand{\lm}{\lambda}
\newcommand{\ep}{\epsilon}

\newcommand{\qandq}{\quad\mbox{and}\quad}

\newcommand{\qasq}{\quad\mbox{as}\quad}
\newcommand{\qinq}{\quad\mbox{in}\quad}
\newcommand{\qforallq}{\quad\mbox{for all}\quad}

\newcommand{\bes}{\begin{equation*}}
\newcommand{\ees}{\end{equation*}}
\newcommand{\bequ}{\begin{equation}}
\newcommand{\bi}{\begin{itemize}}
\newcommand{\ei}{\end{itemize}}
\newcommand{\bsplit}{\begin{split}}
\newcommand{\esplit}{\end{split}}
\newcommand{\bea}{\begin{eqnarray}}
\newcommand{\eea}{\end{eqnarray}}
\newcommand{\beas}{\begin{eqnarray*}}
\newcommand{\eeas}{\end{eqnarray*}}
\newcommand{\btab}{\begin{tabular}}
\newcommand{\etab}{\end{tabular}}

\newcommand{\barq}{\bar{Q}}
\newcommand{\barz}{\bar{Z}}

\def\lm{\lambda}
\def\tinf{\rightarrow\infty}

\def\AA{\mathbb{A}}

\endlocaldefs

\begin{document}


\begin{frontmatter}

\title{An ODE for an Overloaded X Model Involving a Stochastic Averaging Principle\protect }
\runtitle{ODE via an Averaging Principle}

\begin{aug}
\author{\fnms{Ohad} \snm{Perry}\thanksref{t1}\ead[label=e1]{o.perry@cwi.nl}}
\and
\author{\fnms{Ward} \snm{Whitt}\thanksref{t2}\ead[label=e2]{ww2040@columbia.edu}}
\runauthor{O. Perry and W. Whitt}

\thankstext{t1}{Centrum Wiskunde \& Informatica (CWI)}
\thankstext{t2}{Department of Industrial Engineering and Operations Research, Columbia University}

\affiliation{CWI and Columbia University}

\address{CWI, Science Park 123, 1098 XG Amsterdam, \\
The Netherlands,\\
\printead{e1}}

\address{Department of Industrial Engineering and Operations Research, \\
Columbia University
New York, New York 10027-6699, \\
\printead{e2}}

\end{aug}

\begin{abstract}
We study an ordinary differential equation (ODE) arising as
the many-server heavy-traffic fluid limit of a sequence of overloaded
Markovian queueing models with two customer classes and two service pools.
The system, known as the X model in the call-center literature,
operates under the
fixed-queue-ratio-with-thresholds (FQR-T) control, which we
proposed in a recent paper as a way for one service system to help
another in face of an unanticipated overload.  Each pool serves only its own class
until a threshold is exceeded; then one-way sharing is activated
with all customer-server assignments then driving
the two queues toward a fixed ratio.  For large systems, that fixed ratio
is achieved approximately.
The ODE describes
system performance during an overload.
The control
is driven by a queue-difference stochastic process, which operates in a faster time scale
than the queueing processes themselves, thus achieving a time-dependent
steady state instantaneously in the limit.  As a result, for the ODE, the
driving process is replaced by its long-run average behavior at
each instant of time;
i.e., the ODE involves a heavy-traffic averaging
principle (AP).
\end{abstract}



\end{frontmatter}

\section{Introduction}\label{secIntro}

We study an {\em ordinary differential equation} (ODE)
that arises as the {\em many-server heavy-traffic} (MS-HT) fluid limit
of a sequence of overloaded Markovian X queueing models under the {\em fixed-queue-ratio-with-thresholds} (FQR-T) control.
The ODE is especially interesting, because it involves a heavy-traffic {\em averaging principle} (AP).

The system consists of
two large service pools that are designed to operate independently, but can help each other when one of the pools, or both,
encounter an unexpected overload, manifested by an instantaneous shift in the arrival rates.
We assume that the time that the arrival rates shift and the values of the new arrival rates are not
known when the overload occurs.  We want the control to automatically detect the overload.
The FQR-T control is designed to prevent sharing of customers (i.e., sending customers to be served
at the other-class service pool) when sharing is not needed, and automatically activate sharing when the system becomes overloaded due to
a sudden shift in the arrival rates.

This paper is the third in a series of five papers.  First,
In
\cite{PeW09} we initiated study of this overload-control problem and
proposed the FQR-T control; see \cite{PeW09} for a discussion of related literature.  We used a heuristic stationary fluid approximation to
derive
the optimal control when a convex holding cost is charged to the two queues during the overload incident.
Within that framework, we showed with simulations that FQR-T outperforms the best fixed allocation of servers,
even when the new arrival rates are known.
The stationary point of the fluid model was derived using a heuristic flow-balance argument, which equates the rate of flow
into the system to the rate of flow out of the system, when the system is in steady state.

Second, in \cite{PeW08} we applied the heavy-traffic AP as
an engineering principle in order to justify the ODE considered here to describe the
transient fluid approximation of the X system under FQR-T after an overload has occurred.  We observed
that the FQR-T control
is driven by a queue-difference stochastic process, which operates in a faster time scale
than the queueing processes themselves, so that it should
achieve a time-dependent
steady state instantaneously in the MS-HT
limit, i.e., as the scale (arrival rate and number of servers) increases; see \S \ref{secHT}.  We argued heuristically that
 the ODE should arise as the limit of a properly-scaled sequence of overloaded X-model systems, provided
 that the driving process is replaced by its long-run average behavior at
each instant of time.  We performed simulation to justify the approximations.

The present paper and the next two provide mathematical support.
The present paper establishes important properties of the ODE suggested in \cite{PeW08}.
The fourth and fifth papers prove limits.
In \cite{PeW10a} we prove that the fluid approximation, as a deterministic function of time, arises as
the MS-HT limit of a sequence of $X$ models; i.e., we prove a {\em functional weak law of large numbers} (FWLLN).
This FWLLN is based on the AP; see \cite{CPR95,HK94} for previous examples.
In \cite{PeW10b} we prove the corresponding {\em functional central limit theorem} (FCLT) that describes the stochastic fluctuations
about the deterministic fluid path.

We prove convergence to the ODE in \cite{PeW10a} by the standard two-step procedure, described in
Ethier and Kurtz ~\cite{EK86}:
(i) establishing tightness and (ii) uniquely characterizing the limit process.
The tightness argument follows familiar lines, but characterizing the limit process
turns out to be challenging.   Indeed, characterizing the limit process depends on the results here.
Thus, the present paper provides a crucial ingredient for the limits established in \cite{PeW10a,PeW10b}.

The AP makes the ODE unconventional.
The AP creates
a singularity region, causing the ODE not to be continuous in its
full state space. Hence, classical results of ODE theory, such as
those establishing
existence, uniqueness and stability of solutions, cannot be
applied directly. Moreover, existing algorithms for numerically
solving ODE's cannot be applied directly either, since the
solution to the ODE requires that the time-dependent steady state
of the {\em fast-time-scale process} (FTSP) be computed at each instant.
Nevertheless, we establish the existence of a
unique solution to the ODE, show that there exists a
unique stationary point; and show that
the fluid process converges to its stationary point as time evolves. Moreover,
we show that the convergence to stationarity is
exponentially fast.  The key is a careful analysis of the FTSP,
which we represent as a {\em quasi-birth-and-death} (QBD) process.
Finally, we provide a numerical algorithm for solving the ODE
based on the matrix-geometric method \cite{LR99}.

\textbf{Here is how the rest of this paper is organized:}  The next two sections provide background.
In \S \ref{secModel}
we elaborate on the X queueing model
and the FQR-T control;
that primarily is a review of ~\cite{PeW09}.  In \S \ref{secRep}
we provide a brief overview of the MS-HT scaling and a heuristic explanation of the AP.
In \S \ref{secODE} we introduce the ODE that we study in subsequent sections.
In \S \ref{secMain} we state out main result, establishing the existence of a unique solution.
In \S \ref{secFast} we establish properties of the FTSP, which depends on the state of the ODE,
and whose steady-state distribution influences the evolution of the ODE. In \S \ref{secUnique} we define the
state space of the ODE, and prove the main theorem about existence of a unique solution. In \S \ref{secStat} we
establish the existence of a unique stationary point and show that the fluid solution converges to
that stationary point as time evolves.
In \S \ref{secExpo} we prove that a solution
converges to stationarity exponentially fast.
In \S \ref{secSufficient} we provide conditions for global state space collapse, i.e., for having the AP
operate for all $t \ge 0$.
In \S \ref{secAlg} we develop an algorithm to numerically solve the ODE (given an initial condition), based on the theory
developed in the previous sections.
We conclude in \S \ref{secProofThLipsz} with one postponed proof.

Additional material appears in an appendix, available from the authors' web pages.
There we analyze the system with an underloaded initial state, and show that the approximating fluid models then
lead to our main ODE in finite time.
We elaborate on the algorithm and give two more numerical examples.
We also provide a few omitted proofs.
Finally, we mention remaining open problems.

\section{Preliminaries} \label{secModel}

This section reviews the highlights of ~\cite{PeW09}, starting with
a definition of the original X queueing model, for which the ODE serves as
an approximation.

\subsection{The Original Queueing Model}

The Markovian X model has two classes of customers,
arriving according to independent Poisson processes with rates $\tilde{\lm}_1$ and $\tilde{\lm}_2$.
There are two queues, one for each class, in which customers that are not routed to service immediately upon arrival wait
to be served.  Customers are served from each queue in order of arrival.
 Each class-$i$ customer has limited patience, which is assumed to be exponentially distributed with rate $\theta_i$,
$i = 1,2$. If a customer does not enter service before he runs out of patience, then he abandons the queue.
The abandonment keep the system stable for all arrival and service rates.

There are two service pools, with pool $j$ having $m_j$ homogenous servers (or agents) working in parallel.
This X model was introduced to study two large systems that
are designed to operate independently under normal loads, but can help each other in face of
unanticipated overloads.  We assume that all servers are cross-trained, so that they can serve both classes.
The service times depend on both the customer class $i$ and the server type $j$, and are exponentially distributed;
the mean service time for each class-$i$ customer by each pool-$j$ agent is $1/\mu_{i,j}$.
All service times, abandonment times and arrival processes are assumed to be mutually independent.
The FQR-T control described below assigns customers to servers.

We assume that, at some unanticipated
point of time, the arrival rates change, with at least one increasing.
We further assume that the staffing cannot be changed (in the time scale under consideration) to respond
to this unexpected change of arrival rates. Hence, the arrival processes change from Poisson
with rates $\tilde{\lm}_1$ and $\tilde{\lm}_2$ to Poisson processes with {\em unknown} (but fixed) rates $\lm_1$ and $\lm_2$,
 where $\tilde{\lm}_i < m_i / \mu_{i,i}$, $i = 1,2$ (normal loading),
but $\lm_i > \mu_{i,i} m_i$ for at least one $i$ (the unanticipated overload).
Without loss of generality, we assume that pool $1$ (and class-$1$) is the overloaded (or more overloaded) pool.
The fluid model (ODE) is an approximation for the system performance after the overload has occurred,
so that we start with the new arrival rate pair $(\lm_1,\lm_2)$.

\subsection{The FQR-T Control for the Original Queueing Model}\label{secFQRorig}

The FQR-T control is based on two positive thresholds, $k_{1,2}$ and $k_{2,1}$, and the two queue-ratio parameters, $r_{1,2}$ and
$r_{2,1}$. We define two queue-difference stochastic processes $\tilde{D}_{1,2}(t) \equiv Q_1(t) - r_{1,2} Q_2(t)$ and
$\tilde{D}_{2,1} \equiv r_{2,1} Q_2(t) - Q_1(t)$.
As long as $\tilde{D}_{1,2}(t) \le k_{1,2}$ and $\tilde{D}_{2,1}(t) \le k_{2,1}$ we consider the system to be normally loaded (i.e., not overloaded)
so that no sharing is allowed. Hence, in that case, the two classes operate independently.
Once one of these inequalities is violated, the system is considered to be overloaded, and sharing is initialized.
For example, if $\tilde{D}_{1,2}(t) > k_{1,2}$, then class $1$ is judged to be overloaded
and service-pool $2$ is allowed to start helping queue $1$.
As soon as the first class-$1$ customer starts his service in pool $2$, we drop the threshold $k_{1,2}$, but keep the other threshold $k_{2,1}$.
Now, the sharing of customers is done as follows:
If a type-$2$ server becomes available at time $t$, then it will take its next customer from the head of queue $1$ if
$\tilde{D}_{1,2}(t) > 0$. Otherwise, it will take its next customer from the head of queue $2$.
If at some time $t$ after sharing has started queue $1$ empties, or $\tilde{D}_{2,1}(t) = k_{2,1}$ then
the threshold $k_{1,2}$ is reinstated.
The control works similarly if class $2$ is overloaded, but with pool-$1$ servers helping queue $2$, and with the threshold $k_{2,1}$
dropped once it is crossed.

In addition, we impose the condition of {\em one-way sharing}:  we allow sharing in only one direction at any time.
Thus, in the example above, where sharing is done with pool $2$ helping class $1$, we do not later allow pool $1$ to help class $2$
until there are no more pool-$2$ agents serving class-$1$ customers.  Sharing is initiated with pool $1$ helping class $2$
when $\tilde{D}_{2,1}(t) > k_{2,1}$ and there are no pool-$2$ agents serving class-$1$ customers.  And similarly in the other direction.

Let $Q_i(t)$ be the number of customers in the class-$i$ queue at time $t$, and let $Z_{i,j}(t)$ be the number of class-$i$
customers being served in pool $j$ at time $t$, $i,j = 1,2$.
Let $q_i (t)$ and $z_{i,j} (t)$ be the fluid approximations of $Q_i (t)$ and $Z_{i,j} (t)$, respectively.
With the assumptions on the X system and the FQR-T control, the six-dimensional stochastic process
$(Q_i(t), Z_{i,j}(t) ; i,j = 1,2)$ describing the overloaded system becomes a {\em continuous-time Markov chain} (CTMC) (with stationary transition rates).

Once sharing is initialized, the control makes the overloaded $X$ model operate as an overloaded $N$ model,
and keeps the two queues at approximately the target ratio, e.g.,
if queue $1$ is being helped, then $Q_1(t) \approx r_{1,2} Q_2(t)$. If sharing is done in the opposite direction,
then $r_{2,1} Q_2(t) \approx Q_1(t)$ for all $t \ge 0$.  That is substantiated by
simulation experiments, some of which are reported in ~\cite{PeW09,PeW08}.

In addition to the thresholds $k_{1,2}$ and $k_{2,1}$, discussed above, the model also includes shifting constants
$\kappa_{1,2}$ and $\kappa_{2,1}$.  The shifting constants may be introduced after the threshold is dropped, because it may be dictated by the
optimal ratio function in \cite{PeW09}.
Let $q^*_i$ and $z^*_{i,j}$, $i,j = 1,2$ denote the fluid steady state values of $q_i (t)$ and $z_{i,j} (t)$.
(We will show that a unique steady state, or stationary point, exists for the fluid approximation in \S \ref{secStat} below.)
If the optimal relation between the steady state fluid queues is
$q^*_1 = r^*_{1,2} q^*_2 + \kappa_{1,2}$ for some $\kappa_{1,2} \in \RR$, where $r^*_{1,2}$ denotes the fluid optimal ratio,
(assuming that pool $2$ needs to help class $1$), as is the case
when the holding cost is separable and quadratic with non-zero constant and linear terms, then we use the {\em shifted FQR-T} control.
Shifted FQR-T centers about $\kappa_{1,2}$ instead at about zero. For example, if class $1$ is overloaded, then every
server takes his new customer from the head of queue $1$ if $\tilde{D}_{i,j}(t) > \kappa_{1,2}$.
Otherwise, it takes the new customer
from the head of its own class queue.
We call that control {\em shifted FQR-T} since it keeps the two queues at a fixed ratio, but shifted by the constant $\kappa_{1,2}$.
We can think of FQR-T as the special case of shifted FQR-T with $\kappa_{1,2} = 0$.

The beauty of the control is that, for large-scale service systems,
FQR-T and shifted FQR-T tend to achieve their purpose; i.e.,
they keep the two queues approximately in fixed relation.
In the stochastic system this means that the two-dimensional vector $(Q_1(t), Q_2(t))$ evolves approximately as a
one-dimensional process.  In the fluid model this approximation becomes exact;
We no longer need to consider the three-dimensional process $x(t) \equiv (q_1(t), q_2(t), z_{1,2}(t))$,
since it is enough to consider $z_{1,2}(t)$ together with only one of the queues. The other queue is determined by the first via the
{\em state-space collapse} (SSC) equation $q_1(t) = r_{i,j}q_2(t) + \kappa_{i,j}$, depending on which way the sharing is performed.
In ~\cite{PeW10a} SSC is shown to hold
asymptotically in the MS-HT limit. 

\section{The Many-Server Heavy-Traffic Fluid Limit} \label{secRep}

In this section we briefly describe the convergence of the sequence of stochastic systems to the fluid limit,
as established in ~\cite{PeW10a}.
Without loss of generality {\em we assume that class $1$ is overloaded, and receives help from service-pool $2$}.
(Class $2$ may also be overloaded, but less than class $1$, so that pool $2$ should be serving some class-$1$ customers.)

\subsection{Many-Server Heavy-Traffic (MS-HT) Scaling}\label{secHT}

To develop the fluid limit in \cite{PeW10a}, we consider a sequence of X systems, indexed by $n$ (denoted by superscript), with arrival rates and number of servers growing
proportionally to $n$, i.e.,
\bequ \label{MS-HTscale}
\bar{\lm}^n_i \equiv \frac{\lm^n_i}{n} \ra \lm_i \qandq \bar{m}^n_i \equiv \frac{m^n_i}{n} \ra m_i \qasq n \tinf,
\eeq
with the service and abandonment rates held fixed.
We then define the associated fluid-scaled stochastic processes
\bequ \label{fluidScale}
\barq^n_i (t) \equiv \frac{Q^n_i (t)}{n} \qandq \barz^n_{i,j} (t) \equiv \frac{Z^n_{i,j} (t)}{n}, \quad i,j = 1,2, \quad t \ge 0.
\eeq

For each system $n$, there are threshold $k^n_{1,2}$ and $k^n_{2,1}$, scaled so that
\bequ \label{thresholds}
\frac{k^n_{i,j}}{n} \ra 0 \qandq \frac{k^n_{i,j}}{\sqrt{n}} \ra \infty \qasq n \tinf, \quad i,j = 1,2.
\eeq
The first scaling by $n$ is chosen to make the thresholds asymptotically negligible in MS-HT fluid scaling, so they detect
overloads immediately when they occur.  The second scaling by $\sqrt{n}$ is chosen to make
the thresholds asymptotically infinite in MS-HT diffusion scaling, so that asymptotically the thresholds will
not be exceeded under normal loading.  It is significant that MS-HT scaling shows that we should be able to
simultaneously satisfy both conflicting objectives in large systems.

There are also the shifting thresholds $\kappa^n_{i,j}$, arising from consideration of
separable quadratic cost functions; see \S \ref{secFQRorig}, but we do not specify their scale.
If sharing is taking place, then at some time it was activated by sending the first class-$1$
customer to service pool $2$. We thus need only consider $\kappa^n_{1,2}$ and the weighted-difference process
$\tilde{D}^n_{1,2}(t) \equiv Q^n_1(t) - r^*_{1,2} Q^n_2(t)$.
Note that if $\kappa^n_{1,2} \ra \infty$, then $\tilde{D}^n_{1,2} \ra \infty$ as $n \tinf$.
Hence, we redefine the difference process. Let
\bequ \label{Dprocess}
D^n(t) \equiv (Q^n_1(t) - \kappa^n) - r Q^n_2(t), \quad t \ge 0,
\eeq
where $\kappa \equiv \kappa_{1,2}$ and $r \equiv r^*_{1,2}$.

With the new definition in \eqref{Dprocess}, we allow $\kappa^n$ to be of any order less than or equal to $O(n)$;
in particular, we assume that $\kappa^n / n \ra \kappa$ for $0 \le \kappa < \infty$.
There are two principle cases:  $\kappa = 0$ and $\kappa >0$.
The first case produces FQR (after sharing has began);
the second case produces shifted FQR (shifted by the constant $\kappa^n$).

With the new process $D^n$ in \eqref{Dprocess}, we can apply the same FQR routing rule for both the FQR and shifted FQR cases:
if $D^n(t) > 0$, then every newly available agent (in either pool) takes
his new customer from the head of the class-$1$ queue. If $D^n(t) \le 0$, then every newly available agent takes his new customer
from the head of his own queue.

\subsection{A Heuristic View of the AP}\label{secAP}

The AP is concerned with the system behavior when sharing is taking place; i.e., when some, but not all,
of the pool 2 agents are serving class 1.  That takes place when $q_1 = r q_2 + \kappa$.
In that situation, it can be shown that the queue-difference process $D^n$ in \eqref{Dprocess}
is an order $O(1)$ process, without any spatial
scaling, i.e., for each $t$, the sequence of unscaled random variables $\{D^n (t): n \ge 1\}$
turns out to be stochastically bounded (or tight) in $\RR$.
That implies that $D^n$ operates in a time scale that is different from the other processes $Q^n_i$ and $Z^n_{1,2}$, which are scaled
by dividing by $n$ in \eqref{fluidScale}. 
With the MS-HT scaling in \eqref{MS-HTscale},
in order for the two queues to change significantly (in a relative sense,
which is captured by the scaling in \eqref{fluidScale}),
there needs to be $O(n)$ arrivals and departures
from the queues.  In contrast, the difference process $D^n$ can never go far from $0$, because it has
drift pointing towards $0$ from both above and below.  Thus, the difference process oscillates more and more
rapidly about $0$ as $n$ increases.  
Thus, over short time intervals in which $X^n$ remains nearly unchanged for large $n$,
the process $D^n$ moves rapidly in its state space, nearly achieving a local steady state.
As $n$ increases, the speed of the difference process increases, so that in the limit, it achieves a steady state instantaneously.
That steady state is a local steady state, because it depends on $x(t)$, the fluid limit $x$ at time $t$.

To formalize this separation of time scales,
we define a family of {\em time-expanded} difference processes:  for each $n \ge 1$ and $t \ge 0$, let
\bequ \label{fast102}
D^n_t (s) \equiv D^n (t + s / n), \quad s \ge 0.
\eeq
Dividing $s$ by $n$ in \eqref{fast102} allows us to examine what is happening right after time $t$ in the faster time scale.
For each $t$, a different process $D^n_t$ is defined.  For every $t \ge 0$ and $s > 0$,
the time increment $[t, t + s / n)$ becomes infinitesimal in the limit.
A main result in \cite{PeW10a} (Theorem 5.3) is that, for each $t \ge 0$,
\bequ \label{fastlim}
D^n_t \equiv \{D^n_t (s): s \ge 0\} \Rightarrow D_t (s) \equiv \{D_t (s): s \ge 0\} \qinq D,
\eeq
as $n \tinf$, where the limit
$D_t \equiv \{ D_t(s) : s \ge 0 \}$ is the FTSP.

For each $n$, the control depends on whether or not $D^n (t) > 0$.  In turn, the limiting ODE
depends on the corresponding steady-state probability of the FTSP,
\bequ \label{fast105}
\pi_{1,2}(x(t)) \equiv \lim_{s \tinf} P(D_t(s) > 0) 
\eeq
which depends on $x$ because the distribution of $\{D_t(s) : s \ge 0 \}$ depends on the value of $x(t) \in \RR^3$.

\section{The ODE} \label{secODE}

We now specify the ODE,
which is the main subject of this paper.
We assume that class $1$ is overloaded, even after receiving help from pool $2$.
Hence both pools are fully busy and some pool-$2$ agents are helping class $1$, so that $z_{1,1} (t) = m_1$, $z_{2,1} (t) = 0$
and $z_{1,2} (t) + z_{2,2} (t) = m_2$.  As a consequence, we only need consider $z_{1,2}$ among these four variables.

We introduce an ODE to describe the evolution of the system state, which here is the vector $x(t) \equiv (q_1 (t), q_2 (t), z_{1,2} (t))$.  The
associated state space is
$\rS \equiv [0,\infty)^2 \times [0, m_2]$.  In particular,
we consider the autonomous ODE
\bequ \label{ode}
\dot{x}(t) \equiv (\dot{q}_1(t), \dot{q}_2(t), \dot{z}_{1,2}(t)) = \Psi(x(t)) \equiv \Psi(q_1 (t), q_2 (t), z_{1,2} (t)), \quad t \ge 0,
\eeq
where $\Psi : [0, \infty)^2 \times [0, m_2] \ra \RR^3$ can be displayed via
\bequ \label{odeDetails}
\bsplit
\dot{q}_{1} (t)   & \equiv  \lambda_1  - m_1 \mu_{1,1} - \pi_{1,2} (x(t))\left[z_{1,2} (t) \mu_{1,2} + (m_2 - z_{1,2} (t)) \mu_{2,2}\right]
- \theta_1 q_1 (t) \\
\dot{q}_{2} (t)   & \equiv  \lambda_2   - (1 - \pi_{1,2}(x(t))) \left[(m_2 - z_{1,2} (t)) \mu_{2,2} + z_{1,2} (t) \mu_{1,2}\right]
- \theta_2 q_2 (t) \\
\dot{z}_{1,2} (t) & \equiv  \pi_{1,2}(x(t))(m_2 - z_{1,2} (t)) \mu_{2,2} - (1 - \pi_{1,2}(x(t))) z_{1,2} (t) \mu_{1,2},
\end{split}
\eeq
with $\pi_{1,2}: [0,\infty)^2 \times [0, m_2] \ra [0,1]$ defined by \eqref{fast105} when $q_1 - r q_2 = \kappa$, $\pi_{1,2} (x) \equiv 1$ when $q_1 - r q_2 > \kappa$ and
$\pi_{1,2} (x) \equiv 0$ when $q_1 - r q_2 < \kappa$.
We also consider the associated
{\em initial value problem} (IVP)
\bequ \label{IVP}
\dot{x}(t)  = \Psi(x(t)), \quad x(0) =  w_0
\eeq
for $\Psi(x)$ in \eqref{ode} - \eqref{odeDetails}.

\section{The Main Result}\label{secMain}

The state space $\rS$ is a subset of $\RR^3$ with the boundary constraints:  $q_1 \ge 0$, $q_2 \ge 0$ and $0 \le z_{1,2} (t) \le m_2$.
The differential equation for $z_{1,2}$ prevents its boundary states $0$ and $m_2$ from being active,
because $\dot{z}_{1,2} (t) = \pi_{1,2}(x(t))m_2 \mu_{2,2} \ge 0$ when $z_{1,2} (t) = 0$ and $\dot{z}_{1,2} (t) = (1 - \pi_{1,2}(x(t))m_2 \mu_{1,2} \le 0$
when $z_{1,2} (t) = m_2$.
However, the queue-length constraints can alter the evolution.  In general, we can have $\dot{q}_i (t) < 0$ when $q_i (t) = 0$, which we understand
as leaving $q_i (t)$ fixed at $0$.  However, we are primarily interested in overloaded cases, in which these boundaries are not reached.
Then we can consider the ODE without constraints.

Recall that the shifting constant satisfies $\kappa \ge 0$.
We consider the {\em restricted state space} $S \equiv [\kappa, \infty) \times [0, \infty) \times [0,m_2]$.
We thus avoid the transient region in which $q_1 < r q_2 + \kappa$ with $q_2 = 0$, where $\dot{q}_1 (t) > 0$
and $\dot{q}_2 (t) < 0$, but $q_2$ remains at $0$ while $q_1$ increases to the shifting constant $\kappa$.
The restricted state space, with $q_1 \ge \kappa$ is shown to be the space of the fluid limit of the system
in \cite{PeW10a}. We will also show in Theorem \ref{thInS} below that the ODE cannot leave this restricted state space.

It is convenient to specify the conditions on the model parameters in terms of the steady-state formulas for the queues in isolation.
For that purpose, let $q_i^a$ be the length of fluid-queue $i$ and
let $s^a_i$ be the amount of spare service capacity in service-pool $i$, in steady state, when there is no sharing,
$i = 1,2$.  The quantities
$q_i^a$ and $s^a_i$ are well known, since they are the steady state quantities of the fluid model for the Erlang-A model ($M/M/m_i + M$) with arrival-rate
$\lm_i$, service-rate $\mu_{i,i}$ and abandonment-rate $\theta_i$; see Theorem 2.3 in \cite{W04}, especially equation (2.19), and \S 5.1 in \cite{PeW09}.
In particular,
\bequ \label{Qalone}
q_i^a \equiv \frac{(\lm_i - \mu_{i,i} m_i)^+}{\theta_i} \qandq s^a_i \equiv \left( m_i - \frac{\lm_i}{\mu_{i,i}} \right)^+, \quad i = 1,2,
\eeq
where $(x)^+ \equiv \max\{ x, 0 \}$.
It is easy to see that
$q^a_i s^a_i = 0$, $i = 1,2$. We thus make the following assumption,
which is {\em assumed to hold henceforth}.
\\ \\
{\sc Assumption A.}
{\em
\begin{enumerate}
\item[(I)]  The model parameters satisfy $\theta_1(q^a_1 - \kappa) \ge \mu_{1,2} s^a_2$.
\item[(II)] The initial conditions satisfy $x (0) \in \rS \equiv [\kappa, \infty) \times [0, \infty) \times [0, m_2]$.
\end{enumerate}
}

We now explain these assumptions.
Clearly, a sufficient condition for both pools to be overloaded is to have
$s^a_1 = s^a_2 = 0$, i.e., to have no spare service capacity in either pool in their individual steady states.
However, if $s^a_2 > 0$, both pools can still be overloaded,
provided that enough class-$1$ fluid is processed in pool $2$.
To have the solution be eventually in $\rS$, we require that $\theta_1(q^a_1 - \kappa) \ge \mu_{1,2} s^a_2$.
This condition ensures that service pool $2$
is also full of fluid when sharing is taking place, i.e., $z_{1,2}(t) + z_{2,2}(t) = m_2$ for all $t \ge 0$ (assuming that pool $2$
is full at time $0$).
To see why, note that when service-pool $2$ has spare service capacity ($s^a_2 > 0$), sharing will be activated if $q^a_1 > \kappa$,
because $q^a_2 = 0$.
Now, the maximum amount of class-$1$ fluid that pool $2$ can process, while still processing all of the class-$2$ fluid
(so that $q_2$ is kept at zero), is $\mu_{1,2} s^a_2$. hence, $\mu_{1,2} s^a_2$ is the minimal amount of
class-$1$ fluid that should flow to pool $2$.
On the other hand, $\theta_1 q^a_1 = \lm_1 - \mu_{1,1} m_1$  is equal to
the ``extra'' class-$1$ fluid input, i.e., all the class-$1$ fluid that pool $1$ cannot process.
Some of this ``extra'' class-$1$ fluid might abandon (if $q_1 > 0$). The minimal amount of class-$1$ fluid that abandons is
$\theta_1 \kappa$ (but $\kappa$ can be equal to zero).

We thus require that all the class-$1$ fluid, {\em that is not served in pool $1$},
minus the minimal amount of class-$1$ fluid that abandons, is larger than $\mu_{1,2} s^a_2$. With this requirement, pool $2$ is
assured to be full, assuming that it is initialized full. (If pool $2$ is not initialized full, then it will fill up after some
finite time period; see the appendix.)

\begin{remark}{$($class $1$ need not be more overloaded than class $2$$)$}
{\em  In this paper we are interested in analyzing the ODE \eqref{odeDetails} as given.
Hence, in Assumption A we do not assume that class $1$ is more overloaded than class $2$; i.e.,
we do not require that $q^a_1 - \kappa \ge r q^a_2$.
This extra assumption is not needed for our results for the specified ODE.
In contrast, this assumption is needed in order to show that the ODE holds as the fluid limit, with class $1$ receiving help;
see Assumption 1 in \cite{PeW10a}.
}
\end{remark}

We exploit Assumption A to show that the boundaries of $\rS$ in $\RR^3$ play no role.
\begin{theorem} \label{thInS}
$x (t) \in \rS$ for all $t \ge 0$.
\end{theorem}
We give the proof in \S \ref{secStayInS} after the necessary tools have been developed.

Our main result establishes the existence of a unique solution.

\begin{theorem}{\em $($existence and uniqueness$)$} \label{th1}
For any $w_0 \in \rS$, there exists a unique function
$x: [0, \infty) \ra \rS$ such that, (i) for all $t \ge 0$, there exist $\delta(t) > 0$ such that
$x$ is right-differentiable at $t$, differentiable on $(t, t+ \delta (t))$ and satisfies the
IVP \eqref{IVP} based on the ODE \eqref{ode} over $[t, t+ \delta (t))$ with initial value $x(t)$, and
(ii) x is continuous and differentiable almost everywhere.
\end{theorem}

Theorem \ref{th1} has two parts:  First, there is (i) establishing the local existence and uniqueness of a conventional differentiable solution
on each interval $[t, t + \delta(t))$, for which it suffices to consider
a single $t$, e.g., $t=0$.  Second, there is (ii) justifying an overall continuous solution.

We prove Theorem \ref{th1}
in the
next two sections.
The proof is
tied to the characterization of $\pi_{1,2}$ in \eqref{odeDetails} and \eqref{fast105}, and thus the FTSP $D_t$.
We need to determine conditions for the FTSP $D_t$ to be positive recurrent, so that the AP holds,
and then calculate its steady-state distribution in order to find $\pi_{1,2}$. Moreover, we need to
establish topological properties of the function $\pi_{1,2}$, such as continuity and differentiability.
We discuss the FTSP $D_t$ next.

\section{The Fast-Time-Scale Process} \label{secFast}

Recall
that the FTSP $D_t$ is the limit of $D^n_t$  without any scaling (see \eqref{fastlim}),
where $D^n_t$ is the time-expanded difference process defined in \eqref{fast102}
associated with the queue-difference stochastic process $D^n \equiv (Q^n_1 - \kappa^n) - r Q^n_2$ in
\eqref{Dprocess}.   Since there is no scaling of space,
the state space for the FTSP $D_t$ is the countable lattice $\{\pm j \pm k r: j, k \in \ZZ \}$ in $\RR$.
 To see this, first observe from \eqref{Dprocess} that $D^n$ has state space $\{\pm j \pm k r - \kappa^n: j, k \in \ZZ \}$.
 Next, because of the subtraction in \eqref{fast102}, $D^n_t$ has state space $\{\pm j \pm kr: j, k \in \ZZ \}$.
 Finally, because of the convergence in \eqref{fastlim}, the FTSP $D_t$ has this same state space.

\subsection{The Fast-Time-Scale CTMC}

We fix a time $t$ and assume that we are given the value $x(t) \equiv (q_1(t), q_2(t), z_{1,2}(t))$.
In order to simplify the analysis we assume that $r$ is rational. That clearly is without
any practical loss of generality. Specifically, we assume that $r = j/k$ for some positive integers $j$ and $k$ without any common
factors. We then multiply the process by $k$, so that all transitions can be expressed as $\pm j$ or $\pm k$ in the state space $\ZZ$.
In that case, the FTSP $D_t \equiv \{D_t(s): s \ge 0\}$ becomes a CTMC.

Let $\lm^{(j)}_+(m, x(t))$, $\lm^{(k)}_+ (m, x(t))$, $\mu^{(j)}_+ (m, x(t))$ and $\mu^{(k)}_+ (m, x(t))$
be the transition rates of the FTSP $D_t$ for transitions of $+j$, $+k$, $-j$ and $-k$, respectively,
when $D_t (s) = m > 0$. Similarly, we define the transitions when $D_t (s) = m \le 0$: $\lm^{(j)}_-(m, x(t))$, $\lm^{(k)}_- (m, x(t))$,
$\mu^{(j)}_- (m, x(t))$ and $\mu^{(k)}_- (m, x(t))$.
These rates are the limits of the rates of $D^n_t$ as $n \tinf$ with $\bar{X}^{n}(t) \Rightarrow x (t)$.

First, for $D_t (s) = m \in (-\infty, 0]$, the upward rates are
\bequ \label{bd1}
\bsplit
\lm^{(k)}_- (m, x(t)) & =  \lm_1, \\
\lm_-^{(j)} (m, x(t)) & =  \mu_{1,2}z_{1,2}(t) + \mu_{2,2} (m_2 -z_{1,2}(t)) + \theta_2 q_2(t),
\end{split}
\eeq
corresponding, first, to a class-$1$ arrival and, second, to a departure from the class-$2$ queue,
caused by a type-$2$ agent service completion (of either customer type) or by a class-$2$
customer abandonment.  Similarly, the downward rates are
\bequ \label{bd2}
\mu^{(k)}_- (m, x(t)) =  \mu_{1,1}z_{1,1} (t) + \theta_1 q_1(t), \quad \quad
\mu_-^{(j)} (m, x(t)) = \lm_2,
\eeq
corresponding, first, to a departure from the class-$1$ customer queue, caused by a class-$1$ agent service completion or by a class-$1$
customer abandonment, and, second, to a
class-$2$ arrival.

Next, for $D_t (s) = m \in (0, \infty)$, we have upward rates
\bequ \label{bd3}
\lm^{(k)}_+ (m, x(t)) = \lm_1,\quad \quad \lm_{+}^{(j)} (m, x(t)) = \theta_2 q_2(t),
\eeq
corresponding, first, to a class-$1$ arrival and, second, to a departure from the class-$2$ customer queue caused
by a class-$2$ customer abandonment. The downward rates are
\bequ \label{bd4}
\bsplit
\mu^{(k)}_+ (m, x(t)) & =  \mu_{1,1}z_{1,1} (t) + \mu_{1,2}z_{1,2}(t) + \mu_{2,2}(m_2 -z_{1,2}(t)) + \theta_1 q_1(t), \\
\mu_+^{(j)} (m, x(t)) & = \lm_2,
\end{split}
\eeq
corresponding, first, to a departure from the class-$1$ customer queue, caused by (i) a type-$1$ agent service completion, (ii) a
type-$2$ agent service completion (of either customer type), or (iii) by a class-$1$ customer abandonment and, second, to a class-$2$ arrival.

\subsection{Representing the FTSP $D_t$ as a QBD} \label{secQBD}

Further analysis is simplified by exploiting matrix geometric methods, as in \cite{LR99}.
In particular, we represent the integer-valued CTMC
 $D_t \equiv \{D_t (s): s \ge 0\}$ just constructed
as a homogeneous continuous-time {\em quasi-birth-and-death} (QBD) process, as in Definition 1.3.1 and \S 6.4 of \cite{LR99}.
In passing, note that the special case $r=1$ is especially tractable,
because then the QBD process reduces to an ordinary {\em birth-and-death} (BD) process.

To represent $D_t$ as a QBD process,
 we must re-order the states appropriately.  We order
the states so that the infinitesimal generator matrix $Q$ can be written in
block-tridiagonal form, as in Definition 1.3.1 and (6.19) of \cite{LR99}
(imitating the shape of a generator matrix of a BD process).  In particular,
we write
\beql{QBD}
Q  \equiv
\left( \begin{array}{ccccc}
   B     &  A_0 & 0   & 0    &  \ldots    \\
   A_2   &  A_1 & A_0 & 0    & \ldots    \\
   0     &  A_2 & A_1 & A_0  & \ldots    \\
   0     &  0   & A_2 & A_1  & \ldots    \\
\vdots  &  \vdots & \vdots  & \vdots
\end{array} \right)
\eeq
where the four component submatrices $B, A_0, A_1$ and $A_2$
are all $2m \times 2m$ submatrices for $m \equiv \max{\{j,k\}}$.
In particular,
These $2m \times 2m$ matrices $B, A_0, A_1$ and $A_2$ in turn can be written in block-triangular form
composed of four $m \times m$ submatrices, i.e.,
\beql{sub}
\begin{array}{lccr}
B \equiv
  \left( \begin{array}{cc}
   A_1^{+}    &   B_{\mu}    \\
   B_{\lm}   &   A_1^{-}    \\
                    \end{array} \right)
& \qandq & &
A_i \equiv
  \left( \begin{array}{cc}
   A_i^{+}    &   0    \\
  0   &   A_i^{-}    \\
                    \end{array} \right)
\end{array}
\eeq
for $i = 0, 1, 2$. (All matrices are also functions of $x(t)$.)

To achieve this representation, we need to re-order the states into levels.
The main idea is to represent transitions of $D_t$ above
and below $0$ within common blocks. Let $L(n)$ denote level $n$, $n = 0, 1, 2, \dots$
We assign original states $\phi(n)$ to positive integers $n$ according to the mapping:
\beql{statemap}
\phi (2nm + i) \equiv nm + i \qandq \phi ((2n+1)m + i) \equiv -nm - i + 1, \quad 1 \le i \le m.
\eeq
Then we order the states in levels as follows
\beas \label{order}
L(0) &\equiv& \{1, 2, 3, 4, \dots m, 0, -1, -2, \dots, -(m-1)\}, \\
L(1) &\equiv& \{m+1, m+2, \dots, 2m, -m, -(m+1), \dots, -(2m-1)\}, \quad \ldots \\
\eeas
With this representation, the generator-matrix $Q$ can be written in the form \eqref{QBD} above, where $A_1$ groups all the
transitions within a level, $A_0$ groups the transitions from level $L(n)$ to level $L(n+1)$ and $A_2$ groups all transitions from
level $L(n)$ to level $L(n-1)$. Matrix $B$ groups the transitions within the boundary level $L(0)$, and is thus different than $A_1$.

To illustrate, consider an example with
$r = 0.8$, so that we can choose $j = 4$ and $k = 5$, yielding $m = 5$.
The states are ordered in levels as follows
\beas \label{order2}
L(0) &=& \{1, 2, 3, 4, 5, 0, -1, -2, -3, -4\}, \\
L(1) &=& \{6, 7, 8, 9, 10, -5, -6, -7, -8, -9\}, \\
L(2) &=& \{11, 12, 13, 14, 15, -10, -11, -12, -13, -14\}, \quad \ldots \\
\eeas

Then the submatrices $B_{\mu}$, $B_{\lm}$, $A^+_i$ and $A^-_i$, which form the block matrices $B$ and $A_i$, $i = 0,1,2$,
have the form in \eqref{matrices} below,
where
\beql{diag}
\sigma_{+} = \lambda^{(5)}_{+} + \lambda^{(4)}_{+} + \mu^{(5)}_{+} + \mu^{(4)}_{+} \qandq
\sigma_{-} = \lambda^{(5)}_{-} + \lambda^{(4)}_{-} + \mu^{(5)}_{-} + \mu^{(4)}_{-}.
\eeq
(We solve a full numerical example with these matrices in \S \ref{secExample}.)

Henceforth, we refer to $D_t$ interchangeably as the QBD or the FTSP.

\subsection{Positive Recurrence} \label{secPosRec}

We show that positive recurrence depends only on the constant drift rates in the two regions, as one would expect.
Let $\delta_{+}$ and $\delta_{-}$ be the drift in the positive and negative region, respectively; i.e., let
\bequ \label{drift}
\bsplit
\delta_{+}(x(t)) &\equiv j \left( \lambda^{(j)}_{+}(x(t)) - \mu^{(j)}_{+}(x(t))\right) +
k \left(\lambda^{(k)}_{+}(x(t)) - \mu^{(k)}_{+}(x(t)) \right) \\
\delta_{-}(x(t)) &\equiv j \left( \lambda^{(j)}_{-}(x(t)) - \mu^{(j)}_{-}(x(t))\right) + k \left(\lambda^{(k)}_{-}(x(t)) - \mu^{(k)}_{-}(x(t)) \right).
\end{split}
\eeq

\begin{theorem}
The QBD $D_t$ is positive recurrent
 if and only if
\bequ \label{posrec}
\delta_{-}(x(t)) > 0 > \delta_{+}(x(t)).
\eeq
\end{theorem}

\begin{proof}
We employ the
theory in \S 7 of \cite{LR99}, modified for the continuous-time QBD.
We first construct the aggregate matrices $A \equiv A_0 + A_1 + A_2$, $A^{+} \equiv A_0^{+} + A_1^{+} + A_2^{+}$
and $A^{-} \equiv A_0^{-} + A_1^{-} + A_2^{-}$.
We then observe that the aggregate matrix $A$ is reducible, so we need to consider the component matrices $A^{+}$ and $A^{-}$,
which both are irreducible CTMC infinitesimal generators in their own right.  Let $\nu^{+}$ and $\nu^{-}$ be the unique stationary
probability vectors of $A^{+}$ and $A^{-}$, respectively, e.g., with $\nu^{+}A^{+} = 0$ and $\nu^{+}\bf{1} = 1$.
The theory concludes that our QBD is positive recurrent if and only if
\bequ \label{cond1}
\nu^{+}A^{+}_0 \mathbf{1} < \nu^{+}A^{+}_{2} \mathbf{1} \qandq \nu^{-}A^{-}_0 \mathbf{1} < \nu^{-}A^{-}_{2} \mathbf{1}.
\eeq
In our application it is easy to see that both $\nu^{+}$ and $\nu^{-}$  are the uniform probability vector,
attaching probability $1/m$ to each of the $m$ states, from which the conclusion follows directly.
\end{proof}

\beql{matrices}
\begin{footnotesize}
\begin{array}{lcr}
B_{\mu} =
    \left( \begin{array}{ccccc}
    0           & 0           & 0               & \mu_+^{(4)} & \mu_+^{(5)} \\
    0           & 0           & \mu_+^{(4)}     & \mu_+^{(5)} & 0           \\
    0           & \mu_+^{(4)} & \mu_+^{(5)}     & 0           & 0           \\
    \mu_+^{(4)} & \mu_+^{(5)} & 0               & 0           & 0           \\
    \mu_+^{(5)} & 0           & 0               & 0           & 0
    \end{array} \right)
& &
B_{\lm} =
     \left( \begin{array}{ccccc}
     0           & 0           & 0               & \lm_-^{(4)} & \lm_-^{(5)} \\
     0           & 0           & \lm_-^{(4)}     & \lm_-^{(5)} & 0           \\
     0           & \lm_-^{(4)} & \lm_-^{(5)}     & 0           & 0           \\
     \lm_-^{(4)} & \lm_-^{(5)} & 0               & 0           & 0           \\
     \lm_-^{(5)} & 0           & 0               & 0           & 0
     \end{array} \right)
\\
\\
A_0^{+}  =
  \left( \begin{array}{ccccc}
   \lambda^{(5)}_{+} &   0                &  0                 & 0                  & 0                    \\
   \lambda^{(4)}_{+} &  \lambda^{(5)}_{+} &  0                 & 0                  & 0                    \\
  0                  &  \lambda^{(4)}_{+} &  \lambda^{(5)}_{+} & 0                  & 0                    \\
  0                  &   0                &  \lambda^{(4)}_{+} & \lambda^{(5)}_{+}  & 0                    \\
  0                  &   0                &  0                 & \lambda^{(4)}_{+}  &  \lambda^{(5)}_{+}   \\
                    \end{array} \right)
& &
A_0^{-}  =
  \left( \begin{array}{ccccc}
  \mu^{(5)}_{-}      &   0                &  0                 & 0                  & 0                    \\
  \mu^{(4)}_{-}      &  \mu^{(5)}_{-}     &  0                 & 0                  & 0                    \\
  0                  &  \mu^{(4)}_{-}     &  \mu^{(5)}_{-}     & 0                  & 0                    \\
  0                  &   0                &  \mu^{(4)}_{-}     & \mu^{(5)}_{-}      & 0                    \\
  0                  &   0                &  0                 & \mu^{(4)}_{-}      &  \mu^{(5)}_{-}       \\
                    \end{array} \right)
\\
\\
A_1^{+}  =
  \left( \begin{array}{ccccc}
  -\sigma_{+}    &   0                &  0                 & 0                  & \lambda^{(4)}_{+}    \\
  0              &  -\sigma_{+}       &  0                 & 0                  & 0                    \\
  0              &   0                &  -\sigma_{+}       & 0                  & 0                    \\
  0              &   0                &  0                 & -\sigma_{+}        & 0                    \\
  \mu^{(4)}_{+}  & 0                  &  0                 & 0                  &  -\sigma_{+}   \\
                    \end{array} \right)
& &
A_1^{-}  =
   \left( \begin{array}{ccccc}
 -\sigma_{-}        &   0                &  0                 & 0                  & \mu^{(4)}_{-}        \\
  0                 &  -\sigma_{-}       &  0                 & 0                  & 0                    \\
  0                 &   0                &  -\sigma_{-}       & 0                  & 0                    \\
  0                 &   0                &  0                 & -\sigma_{-}        & 0                    \\
  \lambda^{(4)}_{-} &   0                &  0                 & 0                  &  -\sigma_{-}        \\
                    \end{array} \right)
\\ \\
A_2^{+}  =
  \left( \begin{array}{ccccc}
  \mu^{(5)}_{+}      &  \mu^{(4)}_{+}     &  0                 & 0                  & 0                    \\
  0                  &  \mu^{(5)}_{+}     &  \mu^{(4)}_{+}     & 0                  & 0                    \\
  0                  &   0                &  \mu^{(5)}_{+}     & \mu^{(4)}_{+}      & 0                    \\
  0                  &   0                &  0                 & \mu^{(5)}_{+}      & \mu^{(4)}_{+}        \\
  0                  &   0                &  0                 & 0                  & \mu^{(5)}_{+}        \\
                    \end{array} \right)
& &
A_2^{-}  =
  \left( \begin{array}{ccccc}
  \lambda^{(5)}_{-}  &  \lm^{(4)}_{-}     &  0                   & 0                       & 0                    \\
  0                  &  \lambda^{(5)}_{-} &  \lm^{(4)}_{-}       & 0                       & 0                    \\
  0                  &  0                 &  \lambda^{(5)}_{-}   & \lm^{(4)}_{-}           & 0                    \\
  0                  &  0                 &  0                   & \lambda^{(5)}_{-}       & \lm^{(4)}_{-}        \\
  0                  &  0                 &  0                   & 0                       & \lambda^{(5)}_{-}    \\
                    \end{array} \right)
\end{array}
\end{footnotesize}
\eeq

The alternative cases are simplified by the following relation:
\bequ \label{DriftDiff}
\bsplit
\delta_{-}(x(t)) - \delta_{+}(x(t)) & = (j+k)(\mu_{1,2}z_{1,2} + (m_2 - z_{1,2})) \\
& \quad > (j+k)m_2 (\mu_{1,2} \wedge \mu_{2,2}) > 0.
\end{split}
\eeq
Hence there are only two cases in which the drift does not point inward:
  (i)  $\delta_{+}(x(t)) \ge 0$ and $\delta_{-}(x(t)) > 0$, (ii) $\delta_{-}(x(t)) \le 0$ and $\delta_{+}(x(t)) < 0$.
  In both cases the behavior is unambiguous:
In case (i), clearly $\pi_{1,2} (x(t)) = 1$; in case (ii), clearly $\pi_{1,2} (x(t)) = 0$.

\subsection{Computing $\pi_{1,2}$} \label{secPi}

When the QBD is positive recurrent, the stationary vector of the QBD can be expressed as
$\af \equiv \{\af_n: n \ge 0\} \equiv \{\af_{n, j}: n \ge 0, 1 \le j \le m\}$, where
$\af_n \equiv (\af_n^{+}, \af_n^{-})$ for each $n$,
with $\af_n^+$ and $\af_n^-$ both being $1 \times m$ vectors.
Then the desired probability $\pi_{1,2}$ can be expressed as
\bequ \label{desired}
\pi_{1,2} = \sum_{n= 0}^{\infty} \sum_{j = 1}^{m} \af_{n,j}^{+} = \sum_{n=0}^{\infty} \af_n^{+} \mathbf{1} =
\sum_{n=0}^{\infty} \af_n \mathbf{1_+},
\eeq
where $\textbf{1}$ denotes a column vector with all entries $1$ of the right dimension (here $m \times 1$), while $\bf{1}_{+}$ represents a $2m \times 1$
column vector, with $m$ $1's$ followed by $m$ $0's$.

By Theorem 6.4.1 and Lemma 6.4.3 of \cite{LR99}, the
steady-state distribution has the matrix-geometric form
\bequ \label{steady1}
\af_n = \af_0 R^n,
\eeq
where $R$ is the $2m \times 2m$ {\em rate matrix}, which is the minimal nonnegative solutions to the quadratic matrix equation
$A_0 + R A_1 + R^2 A_2 = 0$,
and can be found efficiently by existing algorithms, as in \cite{LR99} (see \S \ref{secAlg} below).
Since the matrices $A_0$, $A_1$ and $A_2$ have the block-diagonal form in \eqn{sub}, so does $R$,
with submatrices $R^+$ and $R^-$.

Since the spectral radius of the rate matrix $R$ is strictly less than $1$ (Corollary 6.2.4 of \cite{LR99}), the sum of powers of
$R$ is finite, yielding
\bes
\sum_{n=0}^{\infty}{R^n} = (I-R)^{-1}.
\ees
Also, by Lemma 6.3.1 of \cite{LR99}, the boundary probability vector $\af_0$ in \eqn{steady1} is the unique solution to the system
\bequ \label{boundary}
\af_0(B + R A_2) = 0 \qandq \af \textbf{1} = \af_0 (I-R)^{-1} {\bf 1} = 1.
\eeq

Finally, given the above, and using \eqref{desired}, we see that the
desired quantity $\pi_{1,2}$ can be represented as
\bequ \label{desired2}
\pi_{1,2} = \af_0 (I - R)^{-1} \textbf{1}_{+}.
\eeq

For further analysis, it is convenient to have alternative representations for $\pi_{1,2} (x)$.  Let
the vector $\textbf{1}$ have the appropriate dimension in \eqn{piRep2} below.

\begin{theorem}{$($alternative representations for $\pi_{1,2})$}\label{piRep}
Assume that $\delta_{+}(x) < 0 < \delta_{-}(x)$, so that the $QBD$ is positive recurrent at $x$.
$($a$)$ For $r = 1$,
\bequ \label{piRep1}
\pi_{1,2} (x) = \frac{\delta_{-}(x)}{\delta_{-}(x) - \delta_{+}(x)}.
\eeq
$($b$)$ For rational $r$, we have the sub-block representation
\bequ \label{piRep2}
\pi_{1,2} (x) = \frac{\af_0^+ (x) (I - R^+ (x))^{-1} \textbf{1}}{\af_0^+ (x) (I - R^+ (x))^{-1} \textbf{1} + \af_0^- (x) (I - R^- (x))^{-1} \textbf{1}},
\eeq
where we choose $\af_0 (x)$ to satisfy $\af_0 (B (x) + R (x) A_2 (x)) = 0$, renormalize to $\af_0 (x) \textbf{1}  = 1$,
which corresponds to multiplying the original $\af_0 (x) $  by a constant, decompose $\af_0 (x)$
consistent with the blocks as $\af_0 (x) = (\af_0^+ (x), \af_0^- (x))$.
\end{theorem}

\begin{proof} (a)
When $r = 1$, the FTSP $D_t \equiv D_t (x)$ evolves as an $M/M/1$ queue in each of the regions
$D_t (s) > 0$ and $D_t \le 0$.  Thus, we can look at the system at the successive times at which $D_t$ transitions from state $0$ to state $1$,
and then again from state $1$ to state $0$.  That construction produces an alternating renewal process of occupation times in each region,
where these occupation times are distributed as the busy periods of the corresponding $M/M/1$ queues.  Hence, $\pi_{1,2} (x)$ can be expressed as
\bequ \label{mm1}
\pi_{1,2} (x) = \frac{E[T^+ (x)]}{E[T^+ (x)] + E[T^- (x)]},
\eeq
where $T^+ (x)$ is the busy period of the $M/M/1$ queue in the upper region,
while $T^- (x)$ is the busy period of the $M/M/1$ queue in the lower region.  By the definition of $\AA$, these mean busy periods are finite in each region.
In particular,
\bequ \label{piRep3}
E[T^\pm (x)] = \frac{1}{\mu^\pm (x)(1 - \rho^\pm (x))} = \frac{1}{\mu^\pm (x) - \lambda^\pm (x)} = \frac{1}{|\delta_{\pm}(x)|},
\eeq
where $\rho^\pm (x) \equiv \lambda^\pm (x)/\mu^\pm (x)$,
$\lambda^+ (x)$ and $\mu^+ (x)$ are the constant drift rates up (away from the boundary)
and down (toward the boundary) in
the upper region in \eqn{bd3} and \eqn{bd4}, depending on state $x$, while
$\lambda^- (x)$ and $\mu^- (x)$ are the constant drift rates down (away from the boundary) and up (toward the boundary) in the lower region in \eqn{bd1} and \eqn{bd2};
e.g., $\lambda^- (x) \equiv \mu_{-}^{(j)} (x) + \mu_{-}^{(k)} (x)$ with $j = k = 1$ from \eqn{bd2}.

(b)  We first observe that we can reason as in the case $r = 1$, using a regenerative argument.
We can let the regeneration times be successive transitions from one specific QBD state in level $0$ with $D_t \le 0$ to
a specific state in level $1$ where $D_t > 0$.  The intervals between successive transitions will be i.i.d. random variables
with finite mean.  Hence, we can represent $\pi_{1,2} (x)$ just as in \eqref{mm1}, but
where now $T^+ (x)$ is the total occupation time in the upper region with $D_t (s) > 0$ during a regeneration cycle, while
 $T^- (x)$ is the total occupation time in the lower region with $D_t (s) \le 0$ during a regeneration cycle.
 Each of these occupation times can be broken up into first passage times.
For example, $T^+ (x)$ is the sum of first passage times from some state at level $0$ to some other state in level $1$ where $D_t (s) > 0$.
The regenerative cycle will end when the starting and ending states within levels $0$ and $1$ are the designated pair associated with the specified regeneration time.
The successive pairs (i,j) of starting and ending states within the levels $0$ and $1$ evolve according to a positive-recurrent finite-state discrete-time Markov chain.

Paralleling that regenerative argument, we can work with the QBD matrices, as in \eqn{desired2}, but now using an alternative representation.
Since $\textbf{1}_{+} + \textbf{1}_{-} = \textbf{1}$, where all column vectors are $2m \times 1$, we can apply the second equation in \eqn{boundary} to write
\bes
\pi_{1,2} = \frac{\af_0 (I - R)^{-1} \textbf{1}_{+}}{\af_0 (I - R)^{-1} \textbf{1}_{+} + \af_0 (I - R)^{-1} \textbf{1}_{-}}.
\ees
Then we can choose $\af_0$ to satisfy $\af_0(B + R A_2) = 0$, renormalize to $\af_0 \mathbf{1}  = 1$
(which corresponds to multiplying the original $\af_0$  by a constant), decompose $\af_0$
consistent with the blocks, letting $\af_0 = (\af_0^+, \af_0^-)$, to obtain \eqn{piRep2}.
\end{proof}

With the QBD representation, we can determine when the FTSP $D_t$ is positive recurrent, for a given $x(t)$,
using \eqref{posrec}, and then numerically calculate $\pi_{1,2}$. That allows us to numerically solve the ODE \eqref{ode} in \S \ref{secAlg}.
We will also use the representations \eqref{desired2}, \eqref{piRep1}, \eqref{piRep2} and other QBD properties
to deduce topological properties of $\pi_{1,2}$.

\section{Existence and Uniqueness of Solutions} \label{secUnique}

This section is devoted to proving Theorem \ref{th1}.
For the local existence and uniqueness in Theorem \ref{th1} (i), we will show that the function $ \Psi$ in \eqref{odeDetails} is locally Lipschitz continuous
in Theorem \ref{thLipsz} below.  That
 allows us to apply the classical Picard-Lindel\"{o}f theorem
to deduce the desired existence and uniqueness of solutions to the IVP \eqref{IVP};
see Theorem 2.2 of Teschl \cite{T09} or
 Theorem 3.1 in \cite{Khalil}.  Afterwards, in \S \ref{secGlobal} we establish the global properties in Theorem \ref{th1} (ii).

\subsection{Properties of $\Psi$}\label{secPsi}

We divide
the state space $\rS \equiv [\kappa, \infty) \times [0, \infty) \times [0, m_2] \equiv \{(q_1, q_2, z_{1,2})\}$ of the ODE
into three regions:
\bequ \label{space}
\bsplit
\mathbf{\rS^b} \equiv \{ q_1 - r q_2 = \kappa \}, \quad
\mathbf{\rS^+} \equiv \{ q_1 - r q_2 > \kappa \}, \quad
\mathbf{\rS^-} \equiv \{ q_1 - r q_2 < \kappa \},
\end{split}
\eeq
with $\rS = \rS^b \cup \rS^+ \cup \rS^-$.
The boundary subset $\rS^b$ is a hyperplane in the state space $\rS$, and is therefore a closed subset.
It is the subset of $\rS$ in which SSC and the AP are taking place (in fluid scale).
In $\rS^b$ the function $\pi_{1,2}$ can assume its full range of values, $0 \le \pi_{1,2} (x) \le 1$.

The region $\rS^{+}$ above the boundary is an open subset of $\rS$.
For all $x \in \rS^+$, $\pi_{1,2}(x) = 1$.
The region $\rS^-$ below the boundary is also an open subset of $\rS$.
For all $x \in \rS^-$, $\pi_{1,2}(x) = 0$.
It is important to keep in mind that, in order for $\rS^-$ to be a proper subspace of $\rS$,
both service pools must be constantly full (in the fluid limit).
Thus, if
$x \in \rS^-$, then $z_{1,1} = m_1$
and $z_{1,2} + z_{2,2} = m_2$ (but $q_1$ and $q_2$ are allowed to be equal to zero).

It is immediate that the function $\Psi$ in \eqref{odeDetails}
is Lipschitz continuous on $\rS^+$ and $\rS^-$, because $\pi_{1,2}(x) = 1$
when $x \in \rS^+$, and $\pi_{1,2}(x) = 0$ when $x \in \rS^-$, so that $\Psi$ is linear in each region.
However, $\Psi$ is not linear on $\rS^b$.
To analyze $\Psi$ on $\rS^b$, we exploit properties of the QBD introduced in \S \ref{secFast}.
We partition $\rS^b$ into three subsets, depending on the drift rates in \eqref{drift}.
Let $\AA$ be the set of all $x \in \rS^b$ for which the QBD is positive recurrent, as given in \eqref{posrec}; i.e., let
\bequ \label{Aset}
\AA \equiv \{x \in \rS^b \; \mid \; \delta_{-}(x) > 0 > \delta_{+}(x) \}.
\eeq
Let the other two subsets be
\bequ \label{AplusSet}
\AA^{+} \equiv \{x \in \rS^b \; \mid \; \delta_{+}(x) \ge 0\} \qandq
\AA^{-} \equiv \{x \in \rS^b \; \mid \; \delta_{-}(x) \le 0\}.
\eeq
By the relation \eqn{DriftDiff}, there are no other alternatives; i.e., $\rS^b = \AA \cup \AA^+ \cup \AA^-$.
Observe that $\pi_{1,2} (x) = 1$ in $\AA^+$, while $\pi_{1,2} (x) = 0$ in $\AA^-$.

From the continuity of the QBD drift-rates in \eqref{drift}, if follows that $\AA$ is an open and connected subset of $\rS^b$.
Hence, $\AA$ can be regarded as an open connected subset of $\RR^2_+$, since $\rS^b$ is homoeomorphic to $\RR_+ \times [0, m_2]$.
Just as for the open subsets $\rS^+$ and $\rS-$, if the initial value is in $\AA$, then the ODE will remain within $\AA$ over some initial subinterval.
In contrast, the situation is more complicated in $\AA^+$ and $\AA^-$.

There is potential movement out of the region $\rS^b$ only from the sets $\AA^+$ and $\AA^-$.
To understand what can happen, let
$d (x(t)) \equiv q_1 (t) - r q_2 (t)$ and $d' (x(t)) \equiv \dot{q}_1 (t) - r \dot{q}_2 (t)$, from \eqref{odeDetails}.
On $\AA^+$ and $\AA^-$, the possibilities can be determined from the following lemma.

\begin{lemma}\label{lmDriftDeriv}  On $\rS^b$,
if $\pi_{1,2} (x) = 1$, then $d'(x) = \delta_{+}(x)$; if $\pi_{1,2} (x) = 0$, then  $d'(x) = \delta_{-}(x)$.
Hence, on $\AA^+$, $d'(x) \ge 0$, while on $\AA^-$, $d'(x) \le 0$.
\end{lemma}

\begin{proof}
Substitute the appropriate values of $\pi_{1,2} (x (t))$ into \eqref{odeDetails} and compute $\delta_{\pm} (x)$
from \eqref{bd1}--\eqref{bd4}, recalling that $r \equiv j/k$.
\end{proof}

We next separate equality from strict inequality for the weak inequalities in Lemma \ref{lmDriftDeriv}.  For that purpose, we decompose
the sets $\AA^+$ and $\AA^-$ by letting
\bea \label{AplusSubset}
\AA_+^{+} & \equiv \{x \in \AA^+ \; \mid \; \delta_+ (x) > 0\}, \quad \AA_0^{+} & \equiv \{x \in \AA^+ \; \mid \; \delta_+ (x) = 0\}, \nonumber \\
\AA_-^{-} & \equiv \{x \in \AA^- \; \mid \; \delta_- (x) < 0\}, \quad \AA_0^{-} & \equiv \{x \in \AA^- \; \mid \; \delta_- (x) = 0\}.
\eea

\begin{lemma}\label{possibilities}
Suppose that a solution exists for the ODE over a sufficiently small interval starting at $x(0)$.
If $x(0) \in \AA^+_+$, then $x(0+) \in \rS^+$; if $x(0) \in \AA^+_0$, then $x(0+) \in \rS - \rS^- - \AA^-$;
if $x(0) \in \AA^-_-$, then $x(0+) \in \rS^-$; if $x(0) \in \AA^-_0$, then $x(0+) \in \rS - \rS^+ - \AA^+$.
\end{lemma}

\begin{proof}  We only treat the first two cases, because the reasoning for the last two is the same.
If $x(0) \in \AA^+_+$, then $d'(x (0)) > 0$ by
Lemma \ref{lmDriftDeriv}, which implies the result.
If $x(0) \in \AA^+_0$, then $d'(x (0)) = 0$ by
Lemma \ref{lmDriftDeriv}.
To see why we cannot have $x(0+) \in \AA^- \cup \rS^-$, note that then $\pi_{1,2} (x)$ would jump from $1$ to $0$, which would cause
a jump in $d'(x)$ because of
Lemma \ref{lmDriftDeriv} and the inequality in \eqref{DriftDiff}.
\end{proof}

We now are ready to establish local Lipschitz continuity.

\begin{definition}{\em $($local Lipschitz continuity$)$} \label{defLip}
A function $f : \Omega_2 \ra \RR^m$, where $\Omega_1 \subseteq \Omega_2 \subseteq \RR^n$, is locally Lipschitz continuous on $\Omega_1$
within $\Omega_2$ if, for every $v_0 \in \Omega_1$, there exists a neighborhood $U \subseteq \Omega_2$
of $v_0$ such that $f$ restricted to $U$ is Lipschitz continuous; i.e., there exists a constant $K \equiv K(U)$ such that
$\| f(v_1) - f(v_2) \| \le K \| v_1 - v_2 \|$ for every $v_1, v_2 \in U$.
\end{definition}

\begin{theorem} \label{thLipsz}
The function $\Psi$ in {\em \eqref{odeDetails}} is locally Lipschitz continuous
on $\rS^+$ within $\rS^+$, on $\rS^-$ within $\rS^-$, on $\rA $ within $\rS^b$, on $\rA^+$ within $\rS - \rS^-$ and
on $\rA^{-}$ within $\rS - \rS^+$.
\end{theorem}

Note that all of $\rS$ is covered by the five cases in Theorem \ref{thLipsz}.  Also note that, in each case,
when we conclude that $\Psi$ is Lipschitz continuous on $\Omega_1$
within $\Omega_2$, if a solution exists starting at some point in $\Omega_1$, then it will necessarily remain in $\Omega_2$ for a short
interval, by the reasoning above.
We postpone the relatively long proof until \S \ref{secProofThLipsz}.

\subsection{Global Existence and Uniqueness}\label{secGlobal}

This section is devoted to completing the proof of Theorem \ref{th1} by establishing global existence and uniqueness.
We first observe that,
in general, one overall differentiable solution to the ODE over $[0, \infty)$ may not exist.
From either $\rS^-$ or $\rS^+$, the solution $x$ can hit $\rS^b$, i.e., one of the sets $\AA$, $\AA^-$ and $\AA^+$,
and have drifts that are inconsistent with the drifts in the new destination set.
For example, in $\rS^+$ we necessarily have $\pi_{1,2} (x) = 1$.  However, in general there is
nothing preventing $x(t) \ra x(t_b)$, where $x(t) \in \rS^+$ with $\pi_{1,2} (x(t)) = 1$ but also
$\delta_{+} (x(t)) < 0 < \delta_{-} (x(t))$
while  $x(t_b) \in \AA$, necessarily with $\delta_{+} (x(t_b) < 0 < \delta_{-} (x(t_b))$.
The probability $\pi_{1,2} (x (t))$ jumps instantaneously from $1$
to some value strictly between $0$ and $1$ when $\AA$ is hit.
A numerical example is given in the appendix.
To treat that case, We can
start a new ODE at this hitting time of $\AA$.

We first show that the possible values of $x$ are contained in a compact subset of $\rS$,
provided that the initial values of the queue lengths are constrained.
That is accomplished by
proving that a solution to the IVP \eqref{IVP} is bounded.
We use the notation: $a \vee b \equiv \max \{ a, b \}$.
\begin{theorem}{\em $($boundedness$)$} \label{thBound}
Every solution to the {\em IVP} \eqref{IVP} is bounded. In particular, the following upper bounds for the fluid queues hold:
\bequ \label{bound1}
q_i(t) \le q^{bd}_i \equiv q_{i}(0) \vee \lm_i / \theta_i \quad t \ge 0, \quad i = 1,2.
\eeq
\end{theorem}

\begin{proof}
Since $0 \le z_{1,2} \le m_2$ and $q_i \ge 0$ in $\rS$,
we only need to establish \eqref{bound1}.
To do so, it suffices to consider the bounding function describing the queue-length process of each queue in a modified system with no service processes,
so that all the fluid output is due to abandonment, which produces a simple one-dimensional ODE for each queue;
for the remaining details, see \S \ref{secProofs} in the appendix.
\end{proof}

\paragraph{\sc Proof of Theorem \ref{th1} $(ii)$} 

It follows from Theorem \ref{th1} $(i)$ established above, and Theorems \ref{thLipsz} and \ref{thBound},
that any solution $x$ on $[0, \delta)$ can be extended to an interval $[0, \delta')$, $\delta' > \delta$ (even $\delta' = \infty$),
with the solution $\{x (t) : t \in [0, \delta')\}$ again being unique, provided that that the solution $x$
makes no transitions from $\rS - \rS^b$ to $\AA$, causing a discontinuity in $\pi_{1,2} (x)$ and thus $\Psi$ in \eqn{odeDetails}.
(See Theorem 3.3 in \cite{Khalil} and its proof for supporting details.)

Moreover, the solution in $\rS^+$ or $\rS^-$ has a left limit at the time it hits $\AA$.  The left limit exists because,
by Theorem \ref{thBound}, the solution is bounded, and because the derivative in either $\rS^+$ or $\rS^-$ is bounded, by \eqn{odeDetails}.
At each such hitting time, a new ODE is constructed starting in $\AA$.  That ensures the overall continuity of $x$.
In general, there can be accumulation points of such hitting times of the set $\AA$ from $\rS-\rS^b$.
However, any such accumulation point $t$ must be in either $\AA^+$ or $\AA^-$.  That is so, because there then are sequences
$\{t^i_n: n\ge 1\}$, $i = 1,2$ with $x(t^1_n) \in \rS - \rS^b$ and $x(t^2_n) \in \AA$ for all $n$
with $t^i_n \uparrow t$ and $x(t^i_n) \ra x(t) \in \rS^b$ as $n\tinf$ for $i = 1,2$.
Finally, by Theorem \ref{thLipsz}, the function $\Psi$ is locally Lipschitz continuous at each point in $\AA^+ \cup \AA^-$.
Hence, the solution $x$ must actually be differentiable at each of these accumulation times of hitting times.
As a consequence, $x$ is continuous and differentiable almost everywhere throughout $[0, \infty)$.
\qed

In the proof of Theorem \ref{th1} $(ii)$, just completed, we have also established the following result.
\begin{theorem}{\em $($extension to a global solution$)$} \label{thODE2}
Let $x$ be the unique differentiable solution to the {\em IVP} \eqref{IVP} on an interval $[0, \delta)$, established in \S {\em \ref{secPsi}}.
If it is known that the solution can never transition from $\rS^+$ or $\rS^-$ to $\AA$
then there exists a unique differentiable solution to the {\em IVP} \eqref{IVP} on $[0, \infty)$.
\end{theorem}

\section{Fluid Stationarity} \label{secStat}

We now define a stationary point for an ODE and then show that there exists a unique one %
for the ODE \eqref{odeDetails}.
We then give conditions under which the fluid solution $x \equiv \{x(t): t \ge 0\}$  converges to stationarity as $t \ra \infty$.
In \S \ref{secSufficient}, we show that it does so exponentially fast.

\begin{definition}{\em (stationary point for the fluid)} \label{defStaPt}
We say that $x^*$ is a stationary point for the ODE $($or fluid model$)$ if  $x(t) = x^*$ for all $t \ge 0$ when $x(0) = x^*$.
That is, $x^*$ is a stationary point if $\Psi(x^*) = 0$ for $\Psi$ in \eqref{ode} and \eqref{odeDetails}.
If $x(t) = x^*$ for all $t$, then we say that the fluid solution is stationary, or in steady state.
\end{definition}

\subsection{Characterization of the Stationary Point}\label{secUnique2}

By definition, a stationary point $x^* \in \rS$ satisfies $\Psi(x^*) = 0$.
From \eqref{odeDetails}, we see that this gives a system of three equations with three unknowns, namely, $q^*_1$, $q^*_2$ and $z^*_{1,2}$.
The apparent fourth variable $\pi^*_{1,2} \equiv \pi_{1,2}(x^*)$ is a function of the other three variables and
its value is determined by $x^*$.
In principle, the three equations in $\Psi(x) = 0$ can be solved directly to find all the roots of $\Psi$.
However, $\pi_{1,2}^*$ is a complicated function of $x^*$ having the complicated closed-form expression in \eqref{desired} and \eqref{desired2}.

Theorem \ref{thODEsteady} below states that, if there exists a stationary point for the fluid ODE \eqref{odeDetails},
then this point is unique, and must have the specified form.
The uniqueness of $x^*$ is proved by treating $\pi^*_{1,2}$
as a fourth variable, and adding a fourth equation to the three equations $\Psi(x) = 0$.
However, it does not prove that a stationary point exists. In general, the solution $\pi^*_{1,2}$ we get from the
system of four equations may not equal to $\pi_{1,2}(x^*)$, for the function $\pi_{1,2}$ defined in \eqref{fast105}.
The existence of a stationary point is proved in the next section.

The proof of existence is immediate from the proof of uniqueness when $\pi_{1,2}(x^*)$ is known in advance to be $0$ or $1$,
with the value determined.
That occurs everywhere except the region $\AA$; it occurs in the two regions $\rS^+$ and $\rS^-$, but it also occurs
in $\rS^b - \AA$.  Since the QBD is not positive recurrent in $\rS^b - \AA$, it follows that $\pi_{1,2}(x^*)$ can only assume one of the values,
$0$ or $1$, achieving the same value as in the neighboring region $\rS^+$ or $\rS^-$.  (We omit detailed demonstration.)  But we will have to work harder in $\AA$.

We now focus on uniqueness.
Although $\pi^*_{1,2}$ is treated as a variable, we still impose conditions on it so that it can be a legitimate solution
to \eqref{fast105}. In particular, if $q_1^* - r q^*_2 > \kappa$ then we let $\pi^*_{1,2} = 1$; if $q_1^* - r q^*_2 < \kappa$,
then we let $\pi^*_{1,2} = 0$.  Equation \eqref{piSS} below shows that $0 \le \pi^*_{1,2} \le 1$ whenever
$q_1^* - r q^*_2 = \kappa$, i.e., whenever $x^* \in \rS^b$.

For $a, b \in \RR$, recall that
$a \vee b \equiv \max\{a, b\}$ and let $a \wedge b \equiv \min\{a ,b\}$.  Let
\bequ \label{z}
z \equiv \frac{\theta_2(\lm_1 - m_1\mu_{1,1}) - r \theta_1(\lm_2 - m_2\mu_{2,2}) - \theta_1 \theta_2 \kappa} {r \theta_1\mu_{2,2} + \theta_2\mu_{1,2}}.
\eeq

\begin{theorem}{\em $($uniqueness of the stationary point$)$}\label{thODEsteady}
There can be at most one stationary point $x^* \equiv (q^*_1, q^*_2, z^*_{1,2})$ for the IVP \eqref{IVP},
which must take the form
\bequ \label{EqBalance}
\bsplit
z_{1,2}^* &= 0 \vee z \wedge m_2, \quad
q_1^* = \frac{\lm_1 - m_1\mu_{1,1} - \mu_{1,2} z_{1,2}^*}{\theta_1}, \quad
q_2^* = \frac{\lm_2 - \mu_{2,2} (m_2 - z_{1,2}^*)}{\theta_2},
\end{split}
\eeq
for $z$ in \eqref{z}.  Moreover,
\bequ \label{piSS}
\pi^*_{1,2} = \frac{\mu_{1,2}z^*_{1,2}}{\mu_{1,2}z^*_{1,2} + (m_2 - z^*_{1,2}) \mu_{2,2}}.
\eeq
\end{theorem}

\begin{proof}
We start with \eqref{piSS}. This expression is easily derived from the third equation in \eqref{odeDetails}, by equating $\dot{z}_{1,2}(t)$ to zero.
Observe that if $z^*_{1,2} = m_2$ then $\pi^*_{1,2}$ in \eqref{piSS} is equal to $1$, and if $z^*_{1,2} = 0$ then $\pi^*_{1,2} = 0$.
Now, by plugging the value of $\pi^*_{1,2}$ in the ODE's for $\dot{q}_1(t)$ and $\dot{q}_2(t)$ in \eqref{odeDetails}
we get the expressions of $q^*_1$ and $q^*_2$ in \eqref{EqBalance}. We now have the two equations for the stationary queues, but there
are three unknowns: $z^*_{1,2}$, $q^*_1$ and $q^*_2$. We introduce a third equation to resolve this difficulty.

Consider the following three equations with the three unknowns: $z$$, q_1(z)$ and $q_2(z)$. (here $q_1$ and $q_2$ are treated as
functions of the variable $z$, not to be confused with the fluid solution which is a function of the time argument $t$.)
\bequ \label{EqBalance2}
\bsplit
q_1(z) & = \frac{\lm_1 - \mu_{1,1} m_1 - \mu_{1,2} z}{\theta_1}, \quad
q_2(z)  = \frac{\lm_2 - \mu_{2,2} (m_2 - z)}{\theta_2}, \quad \\
\kappa  & = q_1(z) - r q_2(z).
\end{split}
\eeq
Notice that $q_1(z)$ is decreasing with $z$, whereas $q_2(z)$ is increasing with $z$.
Thus, there exists a unique solution to these three equations, which has $z$ as in \eqref{z}.
We can recover $x^*$ from the solution to \eqref{EqBalance2}, and by doing so show that $x^*$ is unique
and is always in one of the three regions $\rS^-$, $\rS^+$ or $\rS^b$ (so that $x^* \in \rS$).

Let $(q_1(z), q_2(z), z)$ be the unique solution to \eqref{EqBalance2}.
First assume that $z > m_2$, which implies that $q_2(z) > 0$, and, by the third equation, $q_1(z) > \kappa$.
By replacing $z$ with $m_2$, $q_1(\cdot)$ is increased and $q_2(\cdot)$ is decreased (but is still positive),
so that $q_1(m_2) - r q_2(m_2) > \kappa$ (and, trivially, $q_1(m_2) > \kappa$, $q_2(m_2) > 0$).
This implies that $x^* \equiv (q_1(m_2), q_2(m_2), m_2) \in \rS^+$ and, if it is indeed a solution to
$\Psi(x) = 0$, then $x^*$ is the unique stationary point for the ODE.

Now assume that the unique solution to \eqref{EqBalance2} has $z < 0$. By replacing $z$ with $0$ we have
$q_1(0) < q_1(z)$ and $q_2(0) > q_2(z)$, which imply that $q_1(0) - r q_2(0) < \kappa$.
Now, since $q_1 (0) = q^a_1$ we have that $q_1 (0) \ge \kappa$ by Assumption A.
This implies that $q_1 (z) > \kappa$, which further implies that $r q_2 (z) = q_1 (z) - \kappa > 0$,
so that $r q_2 (0) > r q_2 (z) > 0$. Taking $x^* \equiv (q_1 (0), q_2 (0), 0)$, we see that $x^* \in \rS^-$,
and if $x^*$ is indeed a solution to $\Psi(x) = 0$, then $x^*$ is the unique stationary point for the ODE.

Finally, assume that the solution $x(z) \equiv (q_1(z), q_2(z), z)$ to \eqref{EqBalance2} has $0 \le z \le m_2$.
To conclude that $x(z)$ is in $\rS^b$ we need to show that $q(z), q_2(z) \ge 0$, so that $q^*_1 = q_1(z)$ and $q^*_2 = q_2(z)$
are legitimate queue-length solutions. We now show that is the case under Assumption A.

Let $S^a_2 \equiv m_2 - \lm_2 / \mu_{2,2}$. Note that, if $S^a_2 \ge 0$, then $S^a_2 = s^a_2$, for $s^a_2$ in \eqref{Qalone}.
We start by rewriting $q_1(z)$ and $q_2(z)$ in \eqref{EqBalance2} as
\bequ \label{EqBalance3}
\bsplit
q_1(z) & = q^a_1 - \frac{\mu_{1,2}}{\theta_1} z, \quad
q_2(z)  = \frac{\mu_{2,2}}{\theta_2} (z - S^a_2).
\end{split}
\eeq
Now, it follows from Assumption A that
\bequ \label{ineq1}
\bsplit
\kappa & \le q^a_1 - \frac{\mu_{1,2}}{\theta_1} s^a_2
 \le q^a_1 - \frac{\mu_{1,2}}{\theta_1} S^a_2,
\end{split}
\eeq
where the second inequality follows trivially, since $S^a_2 \le s^a_2$.
From the third equation of \eqref{EqBalance2}, $\kappa = q_1(z) - r q_2(z)$.  Combining this with \eqref{EqBalance3}, we see that
\bequ \label{ineq2}
\kappa = q_1(z) - r q_2(z) = q^a_1 - \frac{\mu_{1,2}}{\theta_1} z - r \frac{\mu_{2,2}}{\theta_2} (z - S^a_2). 
\eeq
Combining \eqref{ineq1} and \eqref{ineq2}, we get
\bes
q^a_1 - \frac{\mu_{1,2}}{\theta_1} z - r \frac{\mu_{2,2}}{\theta_2} (z - S^a_2) \le q^a_1 - \frac{\mu_{1,2}}{\theta_1} S^a_2,
\ees
which is equivalent to
\bes
0 \le \left( \frac{\mu_{1,2}}{\theta_1} + r \frac{\mu_{2,2}}{\theta_2} \right) (z - S^a_2).
\ees
This, together with the fact that the solution has $z \ge 0$, implies that $z \ge \max\{0, S^a_2\} = s^a_2$.
It follows from \eqref{EqBalance3} that $q_2(z) \ge 0$ and, by using the third equation in \eqref{EqBalance2}
again, $q_1(z) = r q_2(z) + \kappa \ge \kappa \ge 0$.
\end{proof}

An immediate consequence of the proof of Theorem \ref{thODEsteady} is that, in order to find the candidate stationary point $x^*$,
one has to solve the three equations in \eqref{EqBalance2}.
The next corollary summarizes the values $x^*$ may take, depending on its region; the proof appears in the appendix.

\begin{corollary} \label{corStat}
Let $x^* = (q^*_1, q^*_2, z^*_{1,2})$ be the point defined in Theorem \ref{thODEsteady}.
\begin{enumerate}
\item If $x^* \in \rS^b$, then, for $z$ defined in \eqref{z},
\bes
\bsplit
z_{1,2}^* & = z = \frac{\theta_1 \theta_2 (q^a_1 - \kappa) - r \theta_1 (\lm_2 - \mu_{2,2} m_2)}{r \theta_1\mu_{2,2} + \theta_2\mu_{1,2}} \\
& \qquad = \left\{\begin{array}{ll}
\frac{\theta_1 \theta_2 (q^a_1 - r q^a_2 - \kappa)}{r \theta_1\mu_{2,2} + \theta_2\mu_{1,2}}, & \mbox{if \, $q^a_2 \ge 0$, $s^a_2 = 0$.} \\ \\
\frac{\theta_1 \theta_2 (q^a_1 + r \mu_{2,2} s^a_2 / \theta_2 - \kappa)}{r \theta_1\mu_{2,2} + \theta_2\mu_{1,2}}, &
\mbox{if \, $q^a_2 = 0$, $s^a_2 > 0$.}
\end{array}\right. \\
q_1^* &= \frac{\lm_1 - m_1\mu_{1,1} - z_{1,2}^*\mu_{1,2}}{\theta_1}, \qquad
q_2^* = \frac{\lm_2 - (m_2 - z_{1,2}^*)\mu_{2,2}}{\theta_2}.
\end{split}
\ees

\item If $x^* = \rS^+$, then
\bes
z^*_{1,2} = m_2, \qquad q_1^* = \frac{\lm_1 - m_1\mu_{1,1} - m_2\mu_{1,2}}{\theta_1}, \qquad q_2^* = \frac{\lm_2}{\theta_2}.
\ees

\item If $x^* \in \rS^-$, then
\bes
z^*_{1,2} = 0, \qquad q^*_1 = \frac{\lm_1 - m_1\mu_{1,1}}{\theta_1}, \qquad  q^*_2 = \frac{\lm_2 - m_2\mu_{2,2}}{\theta_2}.
\ees
\end{enumerate}
\end{corollary}

If $x^* \in \rS^+$, as in $(ii)$, then the system does not have enough service capacity
to keep the weighted difference between the two queues at $\kappa$, even when all agents are working with class $1$. In this case, the only
output from queue $2$ is due to abandonment, since no class-$2$ fluid is being served (in steady state). Queue $2$ is then
equivalent to the fluid approximation for an $M/M/\infty$ system with service rate $\theta_2$ and arrival rate $\lm_2$. On the other hand, queue $1$ is
equivalent to an overloaded inverted-$V$ model: a system in which one class, having one queue, is served by two different service pools.

The next corollary gives necessary and sufficient conditions for $x^*$ to be in each region. It shows that the region of
$x^*$ can be determined from rate considerations alone.  We give the proof in the appendix.

\begin{corollary} \label{corRegions}
Let $x^*$ be as in \eqref{EqBalance}. Then
\begin{enumerate}
\item $x^* \in \rS^b$ if and only if
\bequ \label{RbSta}
\frac{\mu_{1,2}s^a_2}{\theta_1} \vee r q^a_2 \; \le \; q^a_1 - \kappa \; \stackrel{}{\le} \; \frac{r \lm_2}{\theta_2} + \frac{\mu_{1,2} m_2}{\theta_1};
\eeq
$x^* \in \AA$ if and only if both inequalities are strict.
\item $x^* \in \rS^+$ if and only if \,
$q^a_1 - \kappa > \frac{r \lm_2}{\theta_2} + \frac{\mu_{1,2} m_2}{\theta_1}.$
\item $x^* \in \rS^-$ if and only if \,
$r q^a_2 > q^a_1 - \kappa$.
\end{enumerate}
\end{corollary}

\begin{remark}{$($most likely region in applications$)$}
{\em
It follows from Corollary \ref{corRegions} that, in applications, $\AA$
is the most likely region for the stationary point when the system is overloaded, provided that
the arrival rates are about $10-50\%$ larger than planned during an overload incident. Typically,
a much higher overload is needed in order for the stationary point to be in $\rS^+$.
As an example, consider the canonical example from \cite{PeW09}:
There are $100$ servers in each pool, serving their own class at rates $\mu_{1,1} = \mu_{2,2} = 1$.
Type-$2$ servers serve class-$1$ customers at rate $\mu_{1,2} = 0.8$.
Also, $\theta_1 = \theta_2 = 0.3$, $r = 0.8$ and $\kappa = 0$. Suppose that class $2$ is not overloaded with $\lm_2 = 90$.
Then, for the stationary point to be in $\rS^+$, we need to have
$\lm_1 > \mu_{1,1} m_1 + \mu_{1,2} m_2 + \theta_1 r \lm_2 / \theta_2 = 252$, i.e., the class-$1$ arrival rate is $252 \%$
larger than the total service rate of pool $1$. If $\lm_2 > 90$, especially if pool $2$ is also overloaded,
then $\lm_1$ needs to be even larger than that.
}
\end{remark}

\subsection{Existence of a Stationary Point}\label{secExist}

We have just established uniqueness of the stationary point in $\rS$, and characterized it.
In the process, we have also established existence in
$\rS - \rA$.
Now we will establish existence of the stationary point in $\rA$.
First, we calculate the drift rates at $x^* \in \AA$.

\begin{lemma}{\em $($the drift rates at $x^*)$}\label{lmDriftStat}
For $x^*$ in Corollary {\em \ref{corStat}} $(i)$, where $0 < z^*_{1,2} < m_2$,
\bequ \label{StatDrift}
\delta_{+} (x^*) = -(j+k)\mu_{2,2}(m_2 - z^{*}_{1,2}) < 0,
\quad  \delta_{-} (x^*) = +(j+k)\mu_{1,2}z^{*}_{1,2} > 0.
\eeq
\end{lemma}

\begin{proof}
Substitute $x^*$ in Corollary \ref{corStat} (i) into \eqn{drift}, using \eqref{bd1}-\eqref{bd4}.
\end{proof}

We now are ready to prove existence.

\begin{theorem}$($existence$)$
If the model parameters produce $x^* \in \AA$, i.e., as in Corollary {\em \ref{corStat}} $(i)$, where $0 < z^*_{1,2} < m_2$,
then $x^*$ is the unique stationary point.
\end{theorem}

\begin{proof}
We will prove that there must exist at least one stationary point.
Given that result, by Theorem \ref{thODEsteady} and Corollary \ref{corStat}, there must be exactly one fixed point and that must be the $x^*$ given there.
To establish existence, we will apply the Brouwer fixed point theorem.
It concludes that a continuous function mapping a convex compact subset of Euclidean space $\RR^k$
into itself has at least one fixed point.
We will let our domain be the set
\bequ \label{convexSet}
C (\eta) \equiv \{x \in \rA \cap \rB: \delta_{+} (x) \le - \eta \qandq \delta_{-} (x) \ge \eta\}
\eeq
for an appropriate small positive $\eta$, where $\rB \equiv [0,q^{bd}_1] \times [0,q^{bd}_2] \times [0,m^2]$
with $q^{bd}_i$ being the bound on $q_i$ from Theorem \ref{thBound}.  Choose $\eta$ sufficiently small that $x^* \in C(\eta)$;
that is easily ensured by Lemma \ref{lmDriftStat}.
Since the rates in \eqref{bd1}--\eqref{bd4} and the drift in \eqref{drift} are linear functions of $x$,
we see that $C(\eta)$ is a convex subset of $\AA$ for each $\eta > 0$.  Since the inequalities in \eqref{convexSet} are weak,
the set is closed.  The intersection with $\rB$ guarantees that the set $C (\eta)$ is also bounded.  Hence, $C(\eta)$ is compact.

By Theorem \ref{th1}, for any $x(0) \in C(\eta)$, there exists a unique solution to the ODE over $[0,\delta]$ for some positive $\delta$.
Hence, for any $t$ with $0 < t < \delta$, the map from $x(0)$ to $x(t)$ is continuous; see \S 2.4 of \cite{T09}.
Let $x^*_L \equiv (q^*_{1,L}, q^*_{2,L}, z^*_{1,2,L})$  and $x^*_U \equiv (q^*_{1,U}, q^*_{2,U}, z^*_{1,2,U})$,
where $q^*_{1,L} \equiv q^*_{1} - \epsilon$, $q^*_{2,L} \equiv q^*_{2} - \epsilon$, $z^*_{1,2,L} \equiv z^*_{1,2} - \epsilon$,
$q^*_{1,U} \equiv q^*_{1} + \epsilon$, $q^*_{2,U} \equiv q^*_{2} + \epsilon$ and $z^*_{1,2,U} \equiv z^*_{1,2} + \epsilon$.
Let $\phi_t: C(\eta) \ra C(\eta)$ be the continuous function defined by $\phi_t (x(0)) \equiv (q_{1,t},q_{2,t}, z_{1,2,t})$, where
\bequ \label{bdded}
q_{i,t} \equiv q_i (t) \vee q^*_{i,L}  \wedge q^*_{i,U} \qandq  z_{1,2,t} \equiv z_{1,2,t} \vee z^*_{1,2,L}  \wedge z^*_{1,2,U},
\eeq
for $i = 1,2$.  We can choose $\eta> 0$ and  $\epsilon> 0$ sufficiently small so that, first, $x^* \in C(\eta)$ and, second,
that $x_{i,t} \in C(\eta)$ for each $x(0) \in C(\eta)$.  Hence, the pair $(C(\eta), \phi_t)$ satisfies the conditions for the Brouwer fixed point theorem.
Hence, there exists $x(0) \in C(\eta)$ such that $x(t) = x(0)$.

Now let $\{t_n: n \ge 1\}$ be a sequence of time points decreasing toward $0$.  We can apply the argument above to deduce that,
for each $n$, there exists $x_n (0)$ in $C(\eta)$ such that $x_n (t_n) = x_n (0)$.
However, from the ODE, we have the relation $|x (t) - x(0) - \Psi(x(0))t| \le M t^2$ for all sufficiently small $t$.
Since $\{x_n (0): n \ge 1\}$ is bounded, there exists a convergent subsequence.
Let $x(0)$ be the limit of that convergent subsequence.  For that limit, we necessarily have $\Psi (x(0)) = 0$.
Hence, that $x(0)$ must be a stationary point for the ODE.  By Theorem \ref{thODEsteady}, we must have $x(0) = x^*$.
\end{proof}

\subsection{Global Asymptotic Stability}

Having a unique stationary point does not imply that a fluid solution necessarily converges to that point as $t\tinf$.
It does not even guarantee that a solution to the IVP \eqref{IVP} is asymptotically stable in the sense that, if $\| x(0) - x^* \| < \epsilon$,
then $x(t) \ra x^*$ as $t \tinf$, no matter how small $\epsilon$ is. In fact, there is not even a guarantee that $x(t)$ will remain
in the $\epsilon$-neighborhood of $x^*$ for all $t \ge 0$.
We will establish all of these properties in Theorem \ref{thODElimit} below by showing that $x^{*}$ in \S \ref{secUnique2}
is globally asymptotically stable, as defined below:

\begin{definition}{\em $($global asymptotic stability$)$} \label{defGlobal}
A point $x^*$ is said to be globally asymptotically stable if it is a stationary point and if,
for any initial state $x(0)$ and any $\epsilon > 0$,
there exists a time $T \equiv T(x(0), \epsilon) \ge 0$ such that
$\| x(t) - x^* \| < \epsilon$ for all $t \ge T$.
\end{definition}

Global asymptotic stability goes beyond simple convergence by also requiring
that the limit be a stationary point.

\begin{theorem}{\em $($global asymptotic stability of $x^*$$)$} \label{thODElimit}
The unique stationary point $x^*$ is globally asymptotically stable.
\end{theorem}

The proof of Theorem \ref{thODElimit} relies on
Lyapunov stability theory for deterministic dynamical systems; see Chapter 4 of Khalil \cite{Khalil}.
Let $E$ be an open and connected subset of $\mathbb{R}^n$ containing the origin.
We use standard vector notation to denote the inner product of vectors $a, b \in \RR^n$,
i.e., $a \cdot b = \sum_{i = 1}^{n}{a_i b_i}$.

\begin{definition}{\em $($Lie derivative$)$} \label{defLie}
For a continuously differentiable function $V : E \ra \mathbb{R}$, and a function $\Psi : E \ra \mathbb{R}^n$,
the Lie derivative of $V$ along $\Psi$ is defined by
\bes
\dot{V}(x) \equiv \frac{\partial V}{\partial x} \Psi(x) = \nabla V \cdot \Psi(x),
\ees
where $\nabla V \equiv ( \frac{\partial V}{\partial x_1}, \dots, \frac{\partial V}{\partial x_n} )$ is the gradient of $V$.
\end{definition}

\begin{definition}{\em $($Lyapunov-function candidate$)$} \label{defLyap}
A continuously differentiable function $V: E \ra \mathbb{R}$ is a Lyapunov-function candidate if:
\begin{enumerate}
\item $V(0) = 0$
\item $V(x) > 0$ \quad for all $x$ in $E - \{0\}$
\end{enumerate}
\end{definition}

In proving Theorem \ref{thODElimit} we use the following theorem, which is Theorem 4.2 pg. 124 in \cite{Khalil}:

\begin{theorem}{\em (global asymptotic stability for nonlinear ODE)} \label{thStability}
Let $x = 0$ be a stationary point of $\dot{x} = \Psi(x)$, $\Psi : E \ra \mathbb{R}^n$, and let $V : \mathbb{R}_+^n \ra \mathbb{R}$ be
a Lyapunov-function candidate. If
\begin{enumerate}
\item $V(x) \ra \infty$ \quad as $||x|| \ra \infty$ and
\item $\dot{V}(x) < 0$ \quad for all $x \ne 0$,
\end{enumerate}
then $x = 0$ is globally asymptotically stable as in Definition \ref{defGlobal}.
\end{theorem}

Notice that, under the conditions of Theorem \ref{thStability},
 the Lyapunov-function candidate $V$ provides a form of monotonicity:  We necessarily have $V(0) = 0$ and $V(x(t))$
strictly decreasing in $t$ for $x(t) \not= 0$.  To elaborate, we introduce the notion of
a $V$-ball.  We say that $\beta_V(\alpha)$ is the $\alpha$ $V$-ball with
center at $x^*$ and radius $\alpha$ if
\beql{Vball}
\beta_V(\alpha) \equiv \{ x \in \RR^n : \| V(x) - V(x^*) \| \le \alpha \}.
\eeq
If $x(t_0) \in \beta_V(\alpha)$ for some $\alpha \ge 0$ and $t_0 \ge 0$, then $x(t) \in \beta_V(\alpha)$ for all $t \ge t_0$.

\paragraph{\sc Proof of Theorem \ref{thODElimit}}
Theorem \ref{thStability} applies directly only within one region, starting at a point in $\rS^+$, $\rS^-$, $\AA$, $\AA^-$ or $\AA^+$.
However, we will show that the same Lyapunov function $V$ can be used in all regions,
leading to global decrease of $V$ as $x^*$ is being approached.

Let $x$ be the unique solution to \eqref{IVP}.
Let $x^* \equiv(q_1^*, q_2^*, z^*_{1,2})$ be the stationary point for the system \eqref{ode}.
Without loss of generality, we perform a change of variables and define a new system whose unique
stationary point is $x = 0$. To this end, let $y = x - x^*$ so that $\dot{y} = \dot{x} = \Psi(x)$. Hence,
$\Psi(x) = \Psi(y+x^*) \equiv g(y)$ and we have that $g(0) = \Psi(0+x^*) = \Psi(x^*) = 0$. That is, if $x^*$
is a stationary point for the original system $\dot{x} = \Psi(x)$, then the stationary point
for the new system, $\dot{y} = g(y)$, is $y^* = 0$.
We distinguish between two cases: $(i)\, \mu_{1,2} > \mu_{2,2}$ and $(ii)\, \mu_{1,2} \le \mu_{2,2}$.

$(i)$\ First, if $\mu_{1,2} > \mu_{2,2}$, then choose $V_1(x) \equiv x_1 + x_2$ and apply its Lie derivative along $g(y) = \Psi(y+x^*)$
where $y+x^* = ( q_1(t) + q_1^*, q_2(t) + q_2^*, z_{1,2}(t) + z_{1,2}^* )$ and $x^*$ is given in \eqref{EqBalance}.
By the definition of the Lie derivative, $\dot{V}_1(y)$ is equal to the inner product
\bes
\dot{V}_1(y) = (1, 1, 0) \cdot (\dot{q}_1(t), \dot{q}_2(t), \dot{z}_{1,2}(t))' = \dot{q}_1(t) + \dot{q}_2(t),
\ees
for $\dot{q}_1$, $\dot{q}_2$ and $\dot{z}_{1,2}$ in \eqref{odeDetails}, after the change of variables.
Let $\tilde{z}_{1,2}(t) \equiv z_{1,2}(t) + z^*$.
Then, for $x^* = (q_1^*, q_2^*, z_{1,2}^*)$ as in \eqref{EqBalance}
\bes
\begin{split}
\dot{V}_1(y) &= \lm_1 - m_1\mu_{1,1} - \pi_{1,2}(y(t))[\tilde{z}_{1,2}(t) \mu_{1,2} + (m_2 - \tilde{z}_{1,2}(t))\mu_{2,2}] \\
&  \quad - \theta_1 (q_1(t) + q_1^*)  - (1 - \pi_{1,2}(y(t)))[(m_2 - \tilde{z}_{1,2}(t))\mu_{2,2} + \tilde{z}_{1,2}(t)\mu_{1,2}] \\
&  \quad + \lm_2  - \theta_2 (q_2(t) + q^*) \\
&= \lm_1 + \lm_2 - m_1\mu_{1,1} - m_2\mu_{2,2} + z_{1,2}(t)\mu_{2,2} + z^*\mu_{2,2} - z_{1,2}(t)\mu_{1,2} \\
&  \quad - z_{1,2}^*\mu_{1,2}
 - \theta_1 q_1(t) - \theta_1 q_1^* - \theta_2 q_2(t) - \theta_2 q_2^* \\
&= -\theta_1q_1(t) - \theta_2q_2(t) - z_{1,2}(t)(\mu_{1,2} - \mu_{2,2}).
\end{split}
\ees
Thus, $\dot{V}_1(y) < 0$ for all $y \in \mathbb{R}^3$ unless $y = 0$.

$(ii)$\ When $\mu_{1,2} \le \mu_{2,2}$, there exists a $B \ge 1$ such that $\mu_{2,2} = B\mu_{1,2}$. We next show that
for any $C > B$ the candidate-function $V_2(x) \equiv Cx_1 + x_2 + (C-1)x_3$ is a Lyapunov function.
The Lie derivative of $V_2(x)$ for the modified system $g(y)$ is
\bes
\dot{V}_2(y) = (C, 1, C-1) \cdot (\dot{q}_1(t), \dot{q}_2(t), \dot{z}_{1,2}(t)) = C\dot{q}_1(t) + \dot{q}_2(t) + (C-1)\dot{z}_{1,2}(t).
\ees
Hence,
\bes
\bsplit
\dot{V}_2(y) &= C\left[\lm_1 - m_1\mu_{1,1} - \pi_{1,2}(y(t)) (\tilde{z}_{1,2}(t)\mu_{1,2} + (m_2 - \tilde{z}_{1,2}(t))\mu_{2,2})\right]
 \\
& \quad  - \theta_1(q_1(t) + q_1^*) +\lm_2   - \theta_2(q_2(t) + q_2^*) \\
& \quad - (1 - \pi_{1,2}(y(t)))(\tilde{z}_{1,2}(t)\mu_{1,2} + (m_2 - \tilde{z}_{1,2}(t)\mu_{2,2})) \\
& \quad + (C-1)\left[\pi_{1,2}(y(t))(m_2 - \tilde{z}_{1,2}(t))\mu_{2,2} - (1 - \pi_{1,2}(y(t)))\tilde{z}_{1,2}(t)\mu_{1,2}\right] \\
&= -C\theta_1q_1(t) - \theta_2q_2(t) - z_{1,2}(t)(C\mu_{1,2} - \mu_{2,2}), \\
\end{split}
\ees
so that $\dot{V}_2(y) < 0$ for all $y \ne 0$.

By Theorem \ref{thStability}, $y^* = 0$ is globally asymptotically stable for the modified system $g(y)$.
Hence, $x^*$ is globally asymptotically stable for the original system $\Psi(x)$. That is, for every initial value
$x(0)$ we have that $x(t) \ra x^*$.
\qed

\subsection{Staying in $\rS$} \label{secStayInS}

We also use the Lyapunov argument to prove Theorem \ref{thInS}, i.e., show that the solution to the ODE can never leave $\rS$.

\paragraph{\sc Proof of Theorem \ref{thInS}}
We are given $x (0) \in \rS$.
Consider $t \ge 0$.
It is easy to see that, if $z_{1,2} (t) = 0$, then $\dot{z}_{1,2} (t) \ge 0$, so that $z_{1,2}(t+) \ge 0$.
Similarly, if $z_{1,2} (t) = m_2$, then $\dot{z}_{1,2} (t) \le 0$, so that $z_{1,2} (t+) \le m_2$.

Turning to the queues, note that to leave $\rS$ at time $t+$ we must have $q_1(t) = \kappa$ or $q_2 (t) = 0$
(or both). If $q_1 (t) = \kappa$ and $q_2 (t) > 0$, then $x(t) \in \rS^-$ so that $\pi_{1,2}(x(t)) = 0$.
Plugging this value of $\pi_{1,2} (x (t))$ in the ODE for $q_1 (t)$ in \eqref{odeDetails}, we see that
$\dot{q}_1 (t) \ge \lm_1 - \mu_{1,1} m_1 - \theta_1 \kappa \ge 0$ by Assumption A. Hence, $q_1 (t)$
is nondecreasing.
If $q_1 (t) > \kappa$ and $q_2 (t) = 0$, then $x (t) \in \rS^+$ and $\pi_{1,2} (x (t)) = 1$, which gives
$\dot{q}_2 (t) = \lm_2 > 0$. Hence $q_2$ is increasing at time $t$.

Now consider the case $q_1 (t) = \kappa$ and $q_{2} (t) = 0$, so that $x (t) \in \rS^b$.
For one of the queues to become negative at time $t+$, we need to have its derivative be negative at time $t$.
We will consider various subcases.

First assume that $\dot{q}_1 (t) < 0$ and $\dot{q}_2 (t) \ge 0$.
In that case $q_1 (t+) < q_2 (t+)$, so that $\pi_{1,2} (x (t+)) = 0$. Plugging this value of $\pi_{1,2} (x (t+))$
in the ODE \eqref{odeDetails}, together with $q_{1} (t+) = \kappa$, we see that $\dot{q}_1 (t+) > 0$
by Assumption A.
Next assume that $\dot{q}_1 (t) \ge 0$ and $\dot{q}_2 (t) < 0$. Then $q_1 (t+) > q_2 (t+)$, so that $\pi_{1,2} (x(t+)) = 1$.
Plugging this value of $\pi_{1,2} (x(t+))$, together with $q_2 (t+) = 0$, we see that $\dot{q}_2 (t+) > 0$.

We finally consider the remaining more challenging subcase: $\dot{q}_1 (t) < 0$ and $\dot{q}_2 (t) < 0$.
We will show that this subcase is not possible. To that end,
we further divide this case into three subcases:
$x (t) \in \AA^+$, $x (t) \in \AA^-$ and $x (t) \in \AA$. (Recall that $\rS^b = \AA \cup \AA^+ \cup \AA^-$.)
However, $x (t)$ cannot be in $\AA^-$, since then $\pi_{1,2}(x (t)) = 0$, which implies that $q_1 (t)$
is nondecreasing (plug $\pi_{1,2} (x (t)) = 0$ and $q_1 (t) = \kappa$ into the ODE \eqref{odeDetails}).
Moreover, $x (t)$ cannot be in $\AA^+$, since then $\pi_{1,2} (x (t)) = 1$, which implies that $q_2 (t)$ is strictly increasing.

Now assume the remaining possibility, $x (t) \in \AA$, and recall that $\Psi$ is Lipschitz continuous in $\AA$,
so that the Lyapunov argument holds over $[t, t + \eta)$, for some $\eta > 0$.
Specifically, the Lyapunov function $V$ is monotone increasing in $x (t)$, because $x^* > 0$.
(The inequality holds componentwise.)
If $\mu_{1,2} > \mu_{2,2}$, then we take the Lyapunov function $V_1 (x (t)) = q_1 (t) + q_2 (t)$.
The monotonicity of $V_1$ at $x (t)$ implies that at least one of the queues must be increasing,
which contradicts the assumption that the derivative of both queues is negative at $t$.
If $\mu_{1,2} \le \mu_{2,2}$, then we take the Lyapunov function $V_2 (x (t)) = C q_1 (t) + q_2 (t) + (C-1) z_{1,2} (t)$.
We then choose $C = 1 + \ep$ with $\ep$ small enough, such that $\dot{V}_2 (x (t)) < 0$
(assuming the derivatives of both queues are strictly negative at $t$).
Once again, this contradicts the positive monotonicity of $V$ at $x (t)$.
This concludes the proof.
\qed

\section{Exponential Stability} \label{secExpo}

\begin{definition}{\em $($exponential stability$)$ }
A stationary point $x^*$ is said to be exponentially stable if there exist two real constants $\vartheta$, $\beta > 0$ such that
\bes
\| x(t) - x^* \| \le \vartheta \| x(0) - x^* \| e^{-\beta t},
\ees
for all $t \ge 0$ and for all $x(0)$, where $\| \cdot \|$ is a norm on $\RR^n$.
\end{definition}

We use Theorem 3.4 on p. 82 of Marquez ~\cite{M03},
stated below.

\begin{theorem} {\em $($exponential stability of the origin$)$} \label{thExpZero}
Suppose that all the conditions of Theorem \ref{thStability} are satisfied.  In addition, assume that there exist positive constants
$K_1$, $K_2$, $K_3$ and $p$ such that
\beas
K_1 \| x \|^p &\le& V(x), \quad \le K_2 \| x \|^p \qandq
\dot{V}(x) \le -K_3 \| x \|^p.
\eeas
Then the origin is exponentially stable, and
\bes
\| x(t) \| \le \| x(0) \| \left(K_2/K_1\right)^{1/p} e^{-(K_3/2K_2)t} \qforallq t \qandq x(0).
\ees
\end{theorem}

We use the
$L_1$ norm: $\| x \| = |x_1| + |x_2| + |x_3|$ for $x \in \RR^3$.

\begin{theorem} {\em $($exponential stability of $x^*$$)$} \label{thExpFixed}
Each $x^*$ in $\rS$ is exponentially stable.
\begin{enumerate}
\item If $\mu_{1,2} > \mu_{2,2}$, then
\bes
\| x(t) - x^*\| \le \| x(0) - x^* \| e^{-(K_3/2)t} \qforallq t \qandq x(0)
\ees
for all $x(0) \in \rS$ and $t \ge 0$, where $K_3 \equiv \max\{ \theta_1, \theta_2, \mu_{1,2} - \mu_{2,2} \}$.

\item If $\mu_{2,2} = B\mu_{1,2}$, $B \ge 1$, then
\bes
\| x(t) - x^* \| \le \| x(0) - x^* \| (C / K_1) e^{-(K_4/2)t}
\ees
for all $x(0) \in \rS$, $t \ge 0$ and $C > B$, where $K_1 \equiv \min\{ 1, C-1 \}$ and $K_4 \equiv \max\{ C\theta_1, \theta_2, (C\mu_{1,2} - \mu_{2,2}) \}$.
\end{enumerate}
\end{theorem}

\begin{proof} As in the
proof of Theorem \ref{thODElimit},
Theorem \ref{thExpFixed} applies directly only within one region, starting at a point in $\rS^+$, $\rS^-$, $\AA$, $\AA^-$ or $\AA^+$.
However, again, the same Lyapunov function $V$ can be used in all regions.

We consider the two cases in turn:
$(i)$ In the proof of Theorem \ref{thStability}, $V_1(x) \equiv x_1 + x_2$ was shown to be a Lyapunov function
with a strictly negative Lie derivative.  Since $x \ge 0$, we can take $K_1 = K_2 = 1$ and $p = 1$.
Since $\dot{V}_1(x) = -\theta_1 q_1(t) - \theta_2 q_2(t) - (\mu_{1,2} - \mu_{2,2})z_{1,2}(t)$, we can take $K_3$ in \eqref{K3},
and the result follows from Theorem \ref{thExpZero}.

$(ii)$ We use the Lyapunov function $V_2(x) = Cx_1 + x_2 + (C-1)x_3$.
Then $K_1 \| x \| \le V_2(x) < C \| x \|$ for $K_1 \equiv \min\{ 1, C-1 \}$.
From the proof of Theorem \ref{thStability}, we know that
$\dot{V}_2(x) = -C\theta_1 q_1(t) - \theta_2 q_2(t) - (C\mu_{1,2} - \mu_{2,2})z_{1,2}(t)$, so that
$\dot{V}_2(x) \le -K_4 \| x \|$.
\end{proof}

\section{Conditions for State-Space Collapse} \label{secSufficient}

In this section we give ways of verifying that $x$ lies entirely in $\AA$,
given that $x(0)$ and $x^*$ are both in $\AA$.
In the appendix we provide conditions for the solution to eventually reach $\AA$
after an initial transient.  The results here are intended to apply after this initial transient period
has concluded.

\begin{theorem}{\em $($sufficient conditions for global SSC$)$} \label{ThglobalSuff}
Let ${\nu} \equiv \mu_{1,2} \wedge \mu_{2,2}$,
and suppose that $x(0) \in \AA$.
Also assume that
\beql{Qcond}
q_2(0) \le \lm_2/\theta_2 \qandq  q_1(0) \le (\lm_1 - m_1\mu_{1,1})/\theta_1.
\eeq
If, in addition, the following inequalities are satisfied, then the solution to the IVP \eqref{IVP} is in $\AA$ for all $t$:
\bea \label{suffCond}
(i)  \quad \lm_1 & < & \nu m_2 + m_1\mu_{1,1} \qandq
(ii) \quad \lm_2  >  \nu m_2.
\eea
\end{theorem}


\begin{proof}
We start by showing, under Condition $(i)$, that $\delta_+(x(t))$ in \eqref{drift} is strictly negative for each $t$.
For a fixed $t$,
\bes
\delta_{+}(x(t)) \equiv j \left(\lambda^{(j)}_{+}(t) - \mu^{(j)}_{+}(t)\right) + k \left(\lambda^{(k)}_{+}(t) - \mu^{(k)}_{+}(t)\right) < 0
\ees
if and only if
\bequ \label{driftIneq1}
(\mu_{2,2}- \mu_{1,2})z_{1,2}(t) - m_2\mu_{2,2} < -(\lm_1 - m_1\mu_{1,1}) + r(\lm_2 - \theta_2q_2(t)) + \theta_1q_1(t).
\eeq
If $\mu_{2,2} > \mu_{1,2}$, then the left-hand side (LHS) of \eqref{driftIneq1} is maximized at $z_{1,2}(t) = m_2$, and is equal to $-\mu_{1,2}m_2$.
If $\mu_{2,2} < \mu_{1,2}$, the the LHS is maximized at $z_{1,2}(t) = 0$, and is equal to $-\mu_{2,2}m_2$. When $\mu_{2,2} = \mu_{1,2}$ the LHS
is equal to $-\mu_{2,2}m_2 = -\mu_{1,2}m_2$. Overall, the LHS of \eqref{driftIneq1} is smaller than or equal to $-\nu m_2$.

Since $q_2(0) \le \lm_2/\theta_2$, we conclude, using the bound in \eqref{bound1}, that $\theta_2q_2(t) \le \lm_2$ for all $t \ge 0$. This,
together with the fact that $q_1(t) \ge 0$ for all $t$,
implies that the RHS of \eqref{driftIneq1} is larger than or equal to
$-(\lm_1 - m_1\mu_{1,1})$, so that
\bes
\begin{split}
(\mu_{2,2} - \mu_{1,2})z_{1,2}(t) - \mu_{2,2}m_2 &\le -\nu m_2 
< -(\lm_1 - m_1\mu_{1,1}) \\
& \le -(\lm_1 - m_1\mu_{1,1}) + r(\lm_2 - \theta_2q_2(t)) + \theta_1q_1(t)
\end{split}
\ees
where the second inequality is due to condition $(i)$.

To show that condition $(ii)$ is sufficient to have $\delta_{-}(x(t)) > 0$ for all $t$, fix $t \ge 0$ and note that, for $\delta_{-}(x(t))$
in \eqref{drift}, we have
\bes
\delta_-(x(t)) \equiv j\left(\lm^{(j)}_{-}(t) - \mu^{(j)}_{-}(t)\right) + k\left(\lm^{(k)}_{-}(t) - \mu^{(k)}_{-}(t)\right) > 0
\ees
if and only if
\bequ \label{driftIneq2}
r(\mu_{1,2} - \mu_{2,2})z_{1,2}(t) + r\mu_{2,2}m_2 > -(\lm_1 - m_1\mu_{1,1}) + r(\lm_2 - \theta_2q_2(t)) + \theta_1q_1(t).
\eeq
It is easy to see that the LHS of \eqref{driftIneq2} has a minimum value of $r (\mu_{1,2} \wedge \mu_{2,2}) m_2 \equiv r \nu m_2$.
By essentially the same arguments as in Theorem \ref{thBound} we can show that $q_1(t) \le q_1(0) \vee (\lm_1 - m_1\mu_{1,1})/\theta_1$.
Since we assume that $q_1(0) \le (\lm_1 - m_1\mu_{1,1})/\theta_1$, we have the bound $q_1(t) \le (\lm_1 - m_1\mu_{1,1})/\theta_1$ for all $t \ge 0$.
With this bound, we see that the RHS of \eqref{driftIneq2} is
smaller than or equal to $r\lm_2$. Overall, we have
\bes
\begin{split}
r(\mu_{1,2}-\mu_{2,2})z_{1,2}(t) + r\mu_{2,2} m_2 &\ge r \nu m_2 
> r\lm_2 \\
& \ge -(\lm_1 - m_1\mu_{1,1}) + r(\lm_2 - \theta_2q_2(t)) + \theta_1q_1(t),
\end{split}
\ees
where the second inequality is due to Condition $(ii)$.
Since \eqref{posrec} holds for all $t \ge 0$, we also have $0 < \pi_{1,2}(t) < 1$ for all $t$.
Hence, every solution to the IVP in \eqref{IVP}
must lie entirely in $\AA$.
\end{proof}

For $x^* \in \AA$, we will now show that there exist
$\af > 0$ and $T \equiv T(\af)$, such that global SSC can be inferred once $\| x(T) - x^* \| < \af$.
We exploit
the drift rates at stationarity, defined by
$\delta_{+}^* \equiv \delta_{+}(x^*)$ and $\delta_{-}^* \equiv \delta_{-}(x^*)$.
It follows from the expressions in \eqref{drift} that
\bequ \label{SSrates}
\delta_{+}^* \equiv \delta_{+}(x^*) = -\mu_{2,2}(r+1)(m_2 - z_{1,2}^*), \quad
\delta_{-}^* \equiv \delta_{-}(x^*) = \mu_{1,2}(r+1)z_{1,2}^*.
\eeq
Thus, if $0 < z_{1,2}^* < m_2$, then the positive recurrence condition \eqref{posrec} holds at the stationary point $x^*$. (This agrees with \eqref{piSS}
which has $0 < \pi^*_{1,2} < 1$ if and only if $0 < z^*_{1,2} < m_2$.)

In the next theorem we give explicit expressions for $\af$.
For reasonable rates, such as will hold in applications, $\af$ is quite large.
In fact, in the numerical example considered in \S \ref{secExample} we show that, typically in applications, $\af$ is so large, that
we can infer that $x$ lies entirely in $\AA$ without even solving the IVP; i.e., $x(0) \in \beta_V(\af)$.

\begin{theorem} \label{thVerify}
Suppose that $x^* \in \AA$ and let $\xi \equiv \min\{|\delta^*_+|, \delta^*_-\}$.
\begin{enumerate}
\item If $\mu_{2,2} \ge \mu_{1,2}$, then let $\alpha = \xi / r\theta_2$ 
\item If $\mu_{2,2} < \mu_{1,2}$, then let $\alpha = \xi / \varsigma$,  where $\varsigma \equiv \mu_{1,2} - \mu_{2,2} + \theta_1 + r \theta_2 > 0$.
\end{enumerate}
In both cases, if there exists $T \ge 0$ such that $x(T) \in \beta_V(\alpha)$, then $\{x(t): t \ge T\}$ lies entirely in $\AA$.
\end{theorem}

\begin{proof}
We use
$\beta_V (\alpha)$, the $\alpha$ $V$-ball with
center at $x^*$ and radius $\alpha$, in \eqref{Vball}.
To find a proper $\af$ for the $V$-ball $\beta_V(\af)$, we once again use the conditions \eqref{driftIneq1} and \eqref{driftIneq2}.
We first show how to find $\af$ for the case $\mu_{2,2} = B\mu_{1,2}$ for some $B \ge 1$, i.e., when $\mu_{1,2} \le \mu_{2,2}$.
Recall (proof of Theorem \ref{thODElimit}) that in this case, $V_2(x) = Cx_1 + x_2 + (C-1)x_3$ is a Lyapunov function
for any $C > B$. Also, the Lyapunov function was defined for the modified system in which the origin was the stationary point.

Let $x^* = (q^*_1, q^*_2, z^*_{1,2})$ be the stationary point in $\AA$.
First assume that, at some time $T$, $V_2(x(T)) = \epsilon_1$, i.e., $C q_1(T) + q_2(T) + (C-1)z_{1,2}(T) = \epsilon_1$.
If $x(t) \in \beta_{V_2}(\epsilon_1)$ for all $t > T$, then it must hold that
\bequ \label{Vball2}
\bsplit
q^*_1 - \frac{\epsilon_1}{C} &< q_1(t) < q_1 + \frac{\epsilon_1}{C}, \qquad q^*_2 - \epsilon_1 < q_2(t) < q^*_2 + \epsilon_1 \qandq \\
z^*_{1,2} - \frac{\epsilon_1}{C-1} &< z_{1,2}(t) < z^*_{1,2} + \frac{\epsilon_1}{C-1}, \quad t \ge T.
\end{split}
\eeq
To make sure $\delta_+(x(t)) < 0$, we use \eqref{driftIneq1}, reorganizing the terms. We need to have
\bes
(\mu_{2,2} - \mu_{1,2})z_{1,2}(t) + r \theta_2 q_2(t) - \theta_1 q_1(t) < -(\lm_1 - \mu_{1,1}m_1) + r \lm_2 + \mu_{2,2}m_2.
\ees
By \eqref{Vball2}, the above inequality holds if
\bes
\bsplit
& (\mu_{2,2} - \mu_{1,2}) \left( z^*_{1,2} + \frac{\epsilon_1}{C-1} \right) + r\theta_2 (q^*_2 + \epsilon_1) -
\theta_1 \left(q^*_1 - \frac{\epsilon_1}{C} \right)  \\
& \quad < -(\lm_1 - \mu_{1,1}m_1) + r \lm_2 + \mu_{2,2}m_2.
\end{split}
\ees
Plugging in the expressions for $q_1^*$, $q^*_2$ and $z^*_{1,2}$, we see that we need to find an $\epsilon_1 > 0$ such that
\bes
(\mu_{2,2} - \mu_{1,2})\frac{\epsilon_1}{C-1} + r \theta_2 \epsilon_1 + \theta_1 \frac{\epsilon_1}{C} < \mu_{2,2}(r+1)(m_2 - z^*_{1,2}).
\ees
We can take $C$ as large as needed, so that the only term that matters on the LHS is $r \theta_2 \epsilon_1$. Hence, we need to have
\bes \label{eps1}
\epsilon_1 < \frac{\mu_{2,2}(r+1)(m_2 - z^*_{1,2})}{r \theta_2} = \frac{| \delta_+^* |}{r\theta_2}.
\ees
Similarly, to make sure that $\delta_{-}(x(t)) > 0$, we use \eqref{driftIneq2}, reorganizing the terms. We need to have
\bes
\bsplit
& r(\mu_{1,2} - \mu_{2,2})z_{1,2}(t) + r\theta_2 q_2(t) - \theta_1 q_1(t) \\
& \quad \quad > -(\lm_1 - \mu_{1,1}m_1) + r(\lm_2 - \mu_{2,2}m_2).
\end{split}
\ees
Using \eqref{Vball2} again (with a different $\ep_2$), we see that it suffices to show that
\beas
\bsplit
& r(\mu_{1,2} - \mu_{2,2}) \left(z_{1,2}^* + \frac{\epsilon_2}{C-1}\right) + r\theta_2 (q^*_2 - \epsilon_2) -
\theta_1 \left( q^*_1 + \frac{\epsilon_2}{C} \right) \\
& \quad  > -(\lm_1 - \mu_{1,1}m_1) + r(\lm_2 - \mu_{2,2}m_2).
\end{split}
\eeas
Once again, plugging in the values of $q_1^*$, $q^*_2$ and $z^*_{1,2}$, and taking $C$ as large as needed, we can choose
$\epsilon_2 > 0$ such that
\bes \label{eps2}
\epsilon_2 < \frac{\mu_{1,2}(r+1)z^*_{1,2}}{r\theta_2} = \frac{\delta_{-}^*}{r\theta_2}.
\ees
Hence, we can take $\af$ as in $(i)$.

For the second case, when $\mu_{1,2} > \mu_{2,2}$, we use the Lyapunov function $V_1(x) = x_1 + x_2$. Using similar reasoning as above, we get
\bes
\epsilon_1 < \frac{\mu_{2,2}(r+1)(m_2 - z^*_{1,2})}{\mu_{1,2} - \mu_{2,2} + \theta_1 + r \theta_2} = \frac{|\delta^*_+|}{\varsigma} \qandq
\epsilon_2 < \frac{\mu_{1,2}(r+1)z^*_{1,2}}{\mu_{1,2} - \mu_{2,2} + \theta_1 + r \theta_2} = \frac{\delta^*_-}{\varsigma}.
\ees
Hence, in this case we can take $\af$ in $(ii)$.
\end{proof}

\section{A Numerical Algorithm to Solve the IVP} \label{secAlg}


\subsection{Computing $\pi_{1,2}(x)$ at a point $x$} \label{secComp1}

The QBD structure in \S \ref{secQBD} allows
us to use established efficient numerical algorithms from \cite{LR99} to solve for the steady state of the QBD to compute $\pi_{1,2}(x)$,
for any given $x \equiv x(t) \in \AA$.
We start by computing the rate matrix $R \equiv R(x)$.
(To simplify notation, we drop the argument $x$, with the understanding that all matrices,
are functions of $x$.)
By Proposition 6.4.2 of \cite{LR99}, $R$ is related to matrices $G$ and $U$ via
\begin{eqnarray}\label{RUG}
G = (-U)^{-1} A_2, \quad
U = A_1 + A_0 G \qandq 
R = A_0 (-U)^{-1}.
\end{eqnarray}
In addition, the matrices $G$ and $R$ are the minimal nonnegative solutions to the quadratic matrix equations
\bequ \label{quadratic}
A_2 + A_1 G + A_0 G^2 = 0 \qandq  A_0 + R A_1 + R^2 A_2 = 0.
\eeq
Hence, if can compute the matrix $G$, then the rate matrix $R$ can be found via \eqref{RUG}. Once $R$ is known,
we use \eqref{boundary} to compute $\af_0$. With $\af_0$ and $R$ in hand, $\pi_{1,2}(x)$ is easily computed
via \eqref{desired2}.

It remains to compute the matrix $G$.
We use the {\em logarithmic reduction} (LR) {\em algorithm} in \S 8.4 of \cite{LR99}, modified to the continuous
case, as in \S 8.7 of \cite{LR99}.  The LR algorithm is quadratically convergent and is numerically well behaved.
These two properties are important, because the matrix $R(x)$ needs to be computed for many values of $x$ when we
numerically solve the IVP \eqref{IVP}. From our experience with this algorithm, it takes fewer than ten iterations to achieve
a $10^{-6}$ precision (when calculating $G$).

\subsection{Computing the Solution $x$}

To compute the solution $x$,
we combine the forward Euler method for solving an ODE with the algorithm to solve for $\pi_{1,2}(x(t))$ described above.
Specifically, we start with a specified initial value $x(0)$, a step-size $h$ and number of iterations $n$, such that $nh = T$.
First, assume that $z_{1,1}(0) = m_1$ and $z_{1,2}(0) + z_{2,2}(0) = m_2$, so that $x(0) \in \rS$.
If $\bar{D}(0) \equiv (q_1(0) - \kappa) - r q_2(0) > 0$ then
$\pi_{1,2}(x(0)) = 1$. If $\bar{D}(0) < 0$ then $\pi_{1,2}(x(0)) = 0$ and if $\bar{D}(0) = 0$ then we check to see whether \eqref{posrec} holds.
If it does, then $x(0) \in \AA$ and we calculate $\pi_{1,2}(x(0))$ as described above.
If $x(0) \in \rS^b - \AA$ then we can still determine the value of $\pi_{1,2}(x(0))$ in the following way:
If
$\delta_{-}(x(t)) = 0 > \delta_{+}(x(t))$,
then we let $\pi_{1,2} (x(t)) = 0$; if instead
$\delta_{-}(x(t)) > 0 = \delta_{+}(x(t))$,
then we let $\pi_{1,2} (x(t)) = 1$.

Given $x(0)$ and $\pi_{1,2}(x(0))$ we can calculate $\Psi(x(0))$ explicitly, and perform the Euler step
$x(h) = x(0) + h \Psi(x(0))$.
We then repeat the procedure for each $k$, $0 \le k \le n-1$, i.e.,
\bequ \label{step}
x((k+1)h) = x(k h) + h \Psi(x(kh)), \quad 0 \le k \le n,
\eeq
where $x(kh)$ is given from the previous iteration, and $\Psi(x(kh))$ can be computed once $\pi_{1,2}(x(kh))$ is found.

If $z_{1,1}(0) < m_1$ or $z_{1,2}(0) + z_{2,2}(0) < m_2$, so that $x(0) \notin \rS$, we use the appropriate fluid model for the alternative region,
as specified in the appendix, where at each Euler step we check to see which fluid model should be applied.

The forward Euler algorithm is known to have an error proportional to the step size $h$,
and to be relatively numerically unstable at times, but it was found to be adequate.
It would be easy to apply more sophisticated algorithms, such as general linear methods,
which have a smaller error, and can be more numerically stable. The only adjustment required is to replace the Euler step in \eqref{step} by
the alternative method.

In the numerical example in \S \ref{secExample} below we let the ratio be $r = 0.8 = 4/5$, so that all the matrices, used in the computations for
$\pi_{1,2}$, are of size $10 \times 10$. It took less than $10$ seconds for the algorithm to terminate
(using a relatively slow, $1$ GB memory, laptop).
The same example, but with $r = 20/25$, so that the matrices are now $50 \times 50$, the algorithm took less than a minute to terminate. Moreover,
the answers to both trials were exactly the same, up to the $7th$ digit.
In both cases, we performed $5000$ Euler steps (each of size $h = 0.01$, so that the termination time is $T = 50$). It is easily
seen that $\pi_{1,2}$ had to be calculated for over $4500$ different points, starting at the time $\pi_{1,2}$ becomes positive
(see Figure \ref{figQBDpi} in the following example).

The validity of the solution can be verified by comparing it to simulation results, as in the example below and others in \cite{PeW09,PeW08}.
There are two other ways to verify the validity:
First, we can check that the solution converges to the stationary point $x^*$,
which can be computed explicitly using \eqref{EqBalance}.
Second, within $\AA$ we can see that the two queues keep at the target ratio $r$,
even though this relation between the two queues is not forced explicitly by the algorithm.

\subsection{A Numerical Example} \label{secExample}

We now provide a numerical example of the algorithm for solving the ODE in \eqn{ode}.
In addition, we added the sample paths of the stochastic processes $Q^n_1$ and $Z^n_{1,2}$, after scaling as in \eqn{fluidScale}, on top of the trajectories of the solution
to their fluid counterparts $q_1$ and $z_{1,2}$.

The model has the same target ratio $r = 0.8$ as in the example in \S \ref{secQBD} with component rate matrices in \eqref{matrices}.
We chose a large queueing system with scaling factor $n = 1000$, so that the stochastic fluctuations do not to hide the general structure of the
simulated sample paths.  We let the ODE model
parameters be $m_1 = m_2 = 1$, $\lm_1 = 1.3$, $\lm_2 = 0.9$, $\mu_{1,1} = \mu_{2,2} = 1$,
$\mu_{1,2} = \mu_{2,1} = 0.8$, $\theta_1 = \theta_2 = 0.3$ and $\kappa = 0$.
The associated queueing model has the same parameters $\mu_{i,j}$ and $\theta_i$,
but the other parameters are multiplied by $n$.  The plots are shown without dividing by $n$.

We ran the algorithm and the simulation for $50$ time units.  We used an Euler step of size $h = 0.01$, so
we performed $5000$ Euler iterations.  In each Euler iteration we performed several iterations to calculate
the matrix $G$ in \eqref{RUG}, which is used to calculate the instantaneous steady-state probability $\pi_{1,2}$.

Figures \ref{figQBDratio}-\ref{figQBDZ12} show $q_1 (t)/q_2(t)$,
 $\pi_{1,2} (x(t)$, $q_1 (t)$ and $z_{1,2} (t)$ as functions of time $t$
for a system initialized empty.
After a short period, the pools fill up.  Then $q_1(t)$ starts to grow,
and immediately then fluid (customers) starts flowing to pool $2$, causing $z_{1,2}(t)$ to grow.
Figures \ref{figQBDratio}-\ref{figQBDZ12} show that, for practical purposes, steady state is achieved for $t \in [10, 20]$.

In Figure \ref{figQBDratio} we see that once $\rS^b$ is hit,
the ratio between the queues is kept at the target ratio $0.8$. This is an evidence for the validity
of the numerical solution, and a strong demonstration of the AP.
In Figure \ref{figQBDpi} we see that initially, while $q_1 = 0$, $\pi_{1,2} = 0$.
This lasts until $z_{2,2}(t) + z_{1,2}(t) = m_2$,
at which time the space $\rS$ is hit, specifically $\rS^b)$, and the averaging begins. Once $\rS^b$
is hit, $\pi_{1,2}$ becomes almost constant, even before the system reaches steady state.
Thus the functions $q_1$, $q_2$ and $z_{1,2}$ have exponential form,
supporting the results of \S \ref{secExpo}.

We got $x(t_n) \equiv  (q_1(t_n), q_2(t_n), z_{1,2}(t_n)) = (0.3639, 0.4550, 0.2385)$
and $\pi_{1,2}(t_n) = 0.2$ when the algorithm terminated.  From \eqref{EqBalance},
$x^* \equiv (q^*_1, q^*_2, z^*_{1,2}) = (0.3667, 0.4595, 0.2375)$.  From \eqref{piSS},
we get $\pi_{1,2}^* = 0.2$.

\begin{figure}[h!]
  \begin{minipage}[t]{.45\textwidth}
    \begin{center}
      \epsfig{file=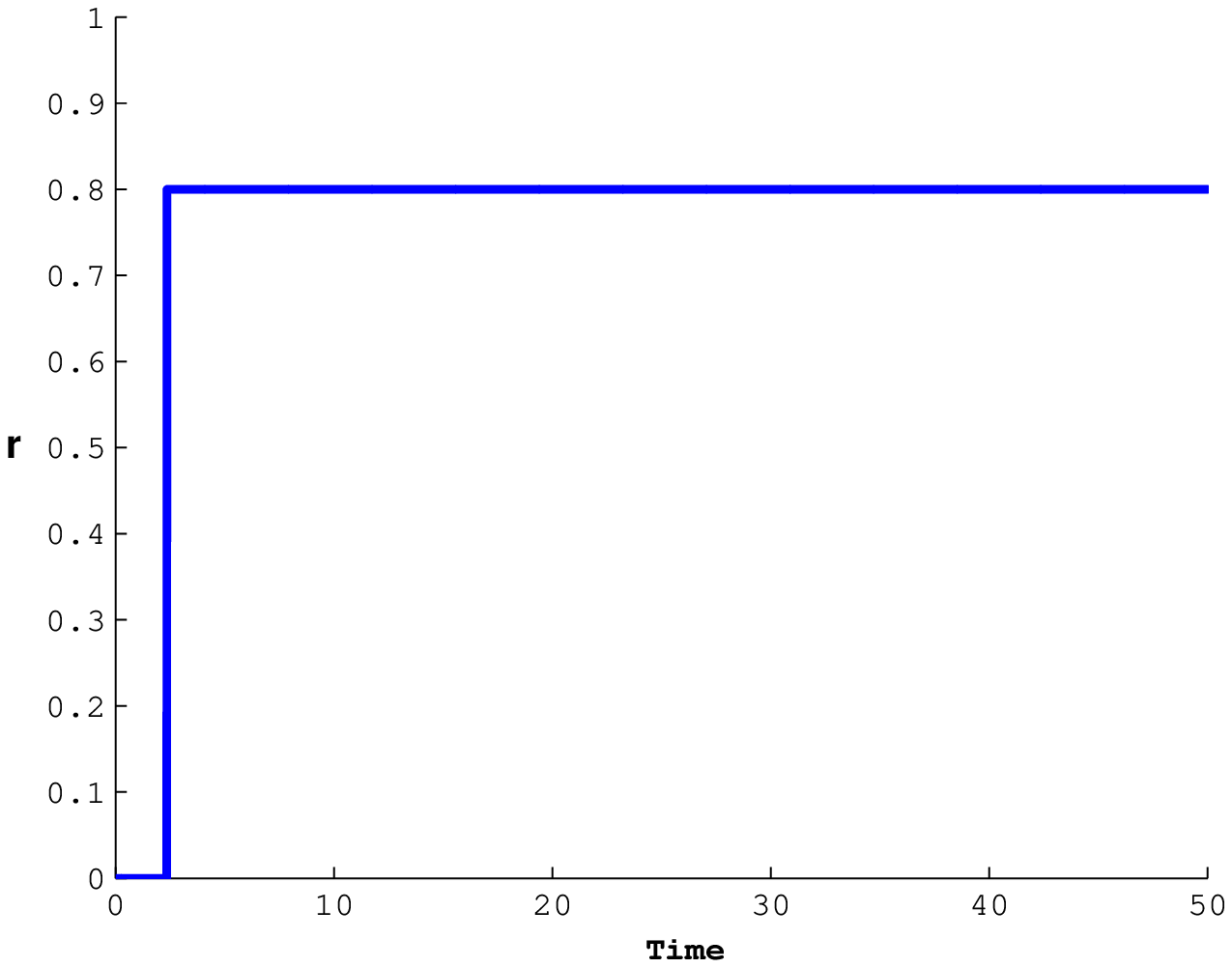, width=1.2\textwidth, height=.15\textheight}
      \caption{ratio between the queues.}
      \label{figQBDratio}
    \end{center}
  \end{minipage}
  \hfill
   \begin{minipage}[t]{.45\textwidth}
    \begin{center}
      \epsfig{file=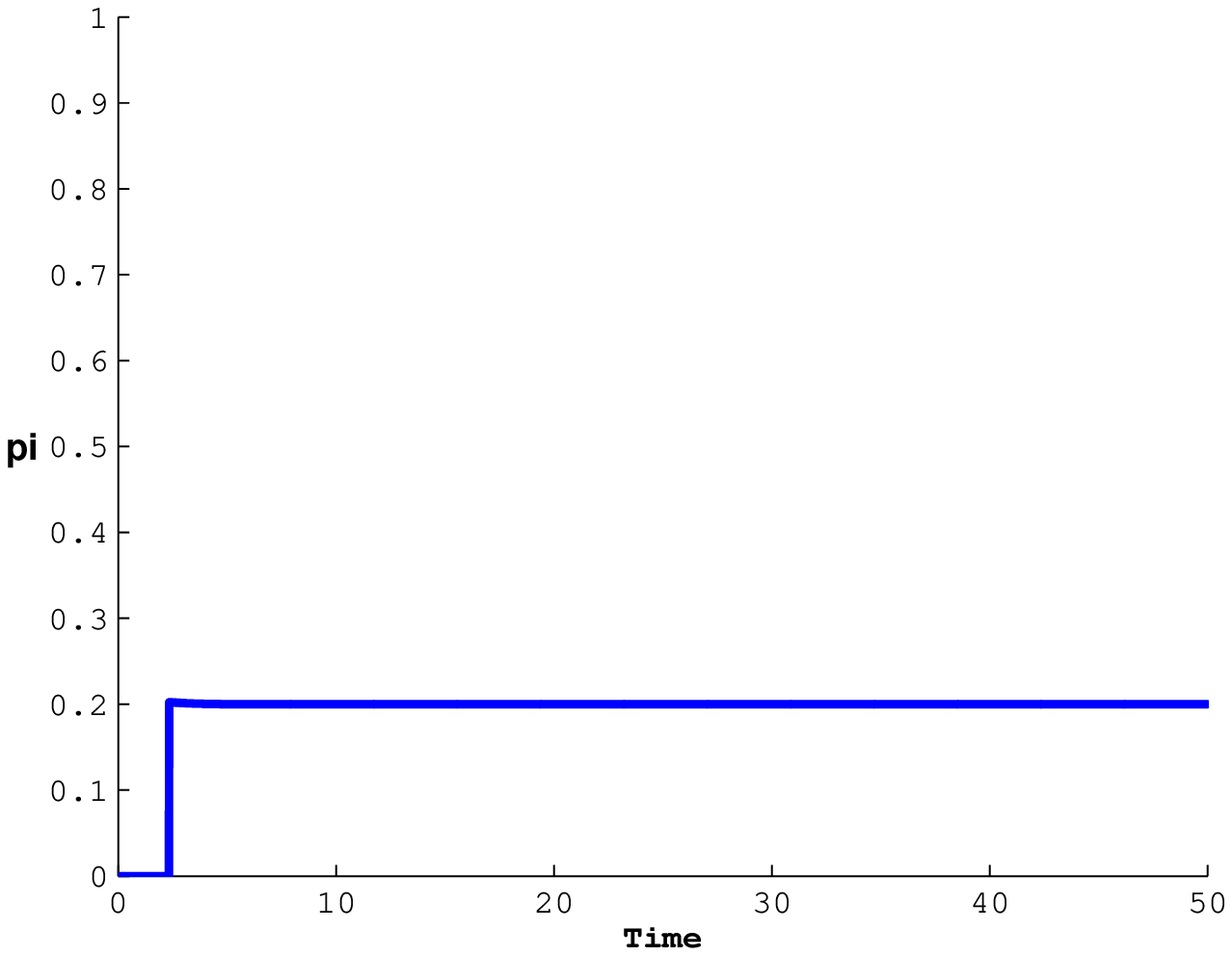, width=1.2\textwidth, height=.15\textheight}
      \caption{$\pi_{1,2}$ calculated at each iteration.}
      \label{figQBDpi}
    \end{center}
  \end{minipage}
  \hfill
\end{figure}

\begin{figure}[h!]
  \begin{minipage}[t]{.45\textwidth}
    \begin{center}
      \epsfig{file=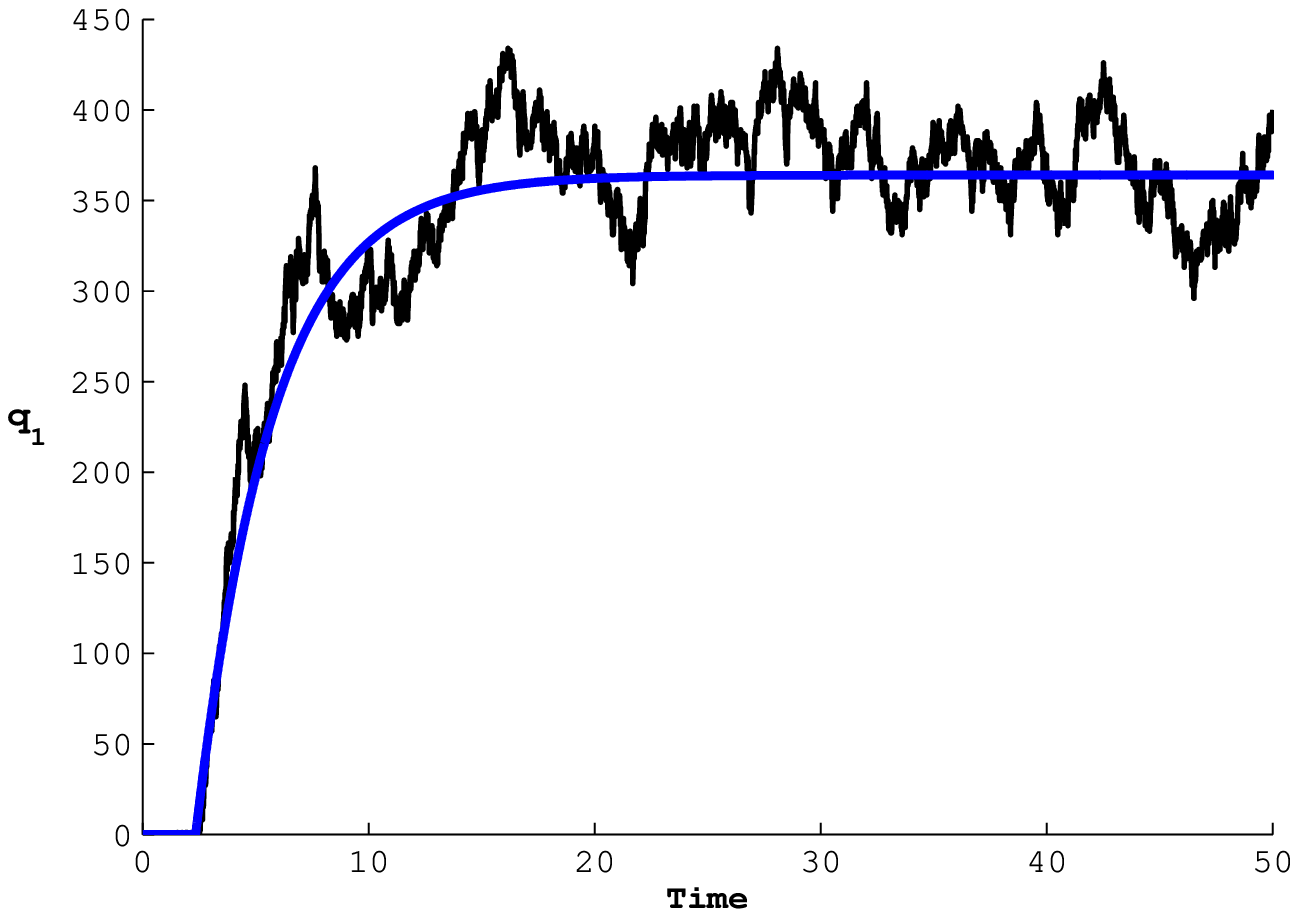, width=1.2\textwidth, height=.15\textheight}
      \caption{trajectory of $q_1$ together with a simulated sample path of the stochastic process $Q_1$
      in a system initializing empty.}
      \label{figQBDQ1}
    \end{center}
  \end{minipage}
  \hfill
  \begin{minipage}[t]{.45\textwidth}
    \begin{center}
      \epsfig{file=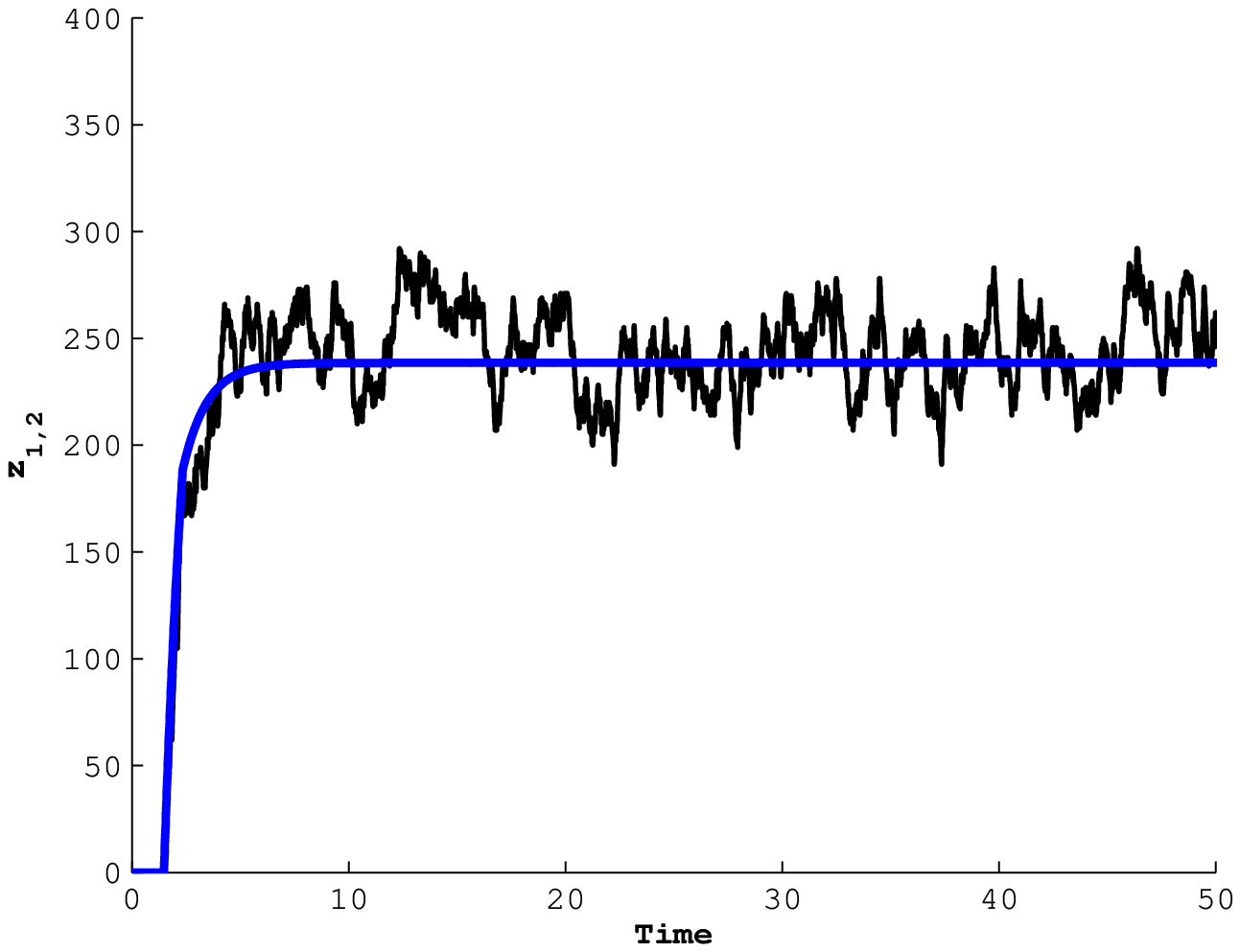, width=1.2\textwidth, height=.15\textheight}
      \caption{trajectory of $z_{1,2}$ together with a simulated sample path of the stochastic process $Z_{1,2}$
      in a system initializing empty.}
      \label{figQBDZ12}
    \end{center}
  \end{minipage}
  \hfill
\end{figure}

Before solving the ODE, we can apply
Theorem \ref{thVerify} to conclude that the solution will remain in $\AA$ after it first hits $\AA$.

\section{Proof of Theorem \ref{thLipsz}}\label{secProofThLipsz}
We have previously observed that the first two conclusions involving $\rS^+$ and $\rS^-$ are valid.
We now prove the three conclusions involving $\AA$, $\AA^{+}$ and $\AA^{-}$.  We will
use the fact that a function mapping a convex compact subset of $\RR^m$ to $\RR^n$ is
Lipschitz on that domain if it has a bounded derivative.
Since we can always work with balls in $\RR^m$ (which are convex with compact closure), that in turn implies that a function
mapping an open subset of $\RR^m$ to $\RR^n$
is locally Lipschitz whenever it has a bounded derivative on each ball in the domain;
e.g., see Lemma 3.2 of \cite{Khalil}.  The three sets $\AA$, $\AA^{+}$ and $\AA^{-}$ are convex.
The key is what happens in $\AA$.

For understanding, it is helpful to first
verify this theorem in the special case $r =1$, where the QBD process reduces to a BD process.
Thus we first give a proof for that special case.

\paragraph{Proof for the special case $r = 1$}
We use the fact that the ODE remains within $\AA$ if it starts in $\AA$, so we are regarding $\AA$ as an open connected convex subset of $\RR^2$.
The key component of the function $\Psi$ in $\AA$ is $\pi_{1,2}$.  We exploit the explicit representations in \eqref{piRep1} and \eqref{piRep3}.
From \eqn{bd1}--\eqn{bd4}, the partial derivatives of $\lambda^\pm (x)$ and $\mu^\pm (x)$ with respect to the three components of $x$, i.e., $q_1$, $q_2$ and $z_{1,2}$,
are constants.
From \eqref{piRep1} and \eqref{piRep3},
we see that the partial derivatives of $\pi_{1,2} (x)$ with respect to each of the three components of $x$ exist, are finite and continuous.
That takes care of $\AA$.

We next consider $\AA^{-}$ and $\AA^{+}$; the reasoning for these two cases is essentially the same, with \eqref{piRep1} making it quite elementary.
 We see that  $\pi_{1,2} (x) \ra 0$ and these partial derivatives approach finite limits
as $x \ra x_b \in \AA^-$ for $x \in \AA$, while $\pi_{1,2} (x) \ra 1$ and these partial derivatives approach finite limits
as $x \ra x_b \in \AA^+$ for $x \in \AA$.  In both cases we have a conventional heavy-traffic limit:  $\rho^\pm (x) \uparrow 1$
as $x \ra x_b$.
Hence, the partial derivatives of $\pi_{1,2} (x)$ are continuous and bounded on $\rS^b$.
As a consequence, for any $\epsilon$-ball in $\rS - \rS^-$ about $x$ in $\AA^{+}$, there exists a constant $K$
such that $|\pi_{1,2} (x_1) - \pi_{1,2} (x_2)| \le K \| x_1 - x_2\|_3$
for all $x_1$ and $x_2$ in the $\epsilon$-ball, where $\| \cdot\|_3$ is the maximum norm on $\RR^3$.
A similar statement applies to $\AA^{-}$.

Hence we have completed the proof for $r = 1$.
In closing, note that we cannot conclude that $\pi_{1,2} (x)$ is even continuous on all of $\rS$, because for $x \in \AA$
we may have a sequence $\{x_n: n \ge 1\}$ with $x_n \in \rS^+$ for all $n$ (or  $x_n \in \rS^-$ for all $n$), with $x_n \ra x$ as $n \tinf$,
$\pi_{1,2} (x_n) = 1$ for all $n$ (or $=0$), while $0 <\pi_{1,2} (x) < 1$.

We now treat the general case.
\paragraph{Proof of Theorem \ref{thLipsz} in the general case}
We first consider $\AA$.  As in the case $r=1$,
we use the fact that the ODE remains within $\AA$ if it starts in $\AA$, so we are regarding $\AA$ as an open connected convex subset of $\RR^2$.
We will look at $\pi_{1,2}$, and thus the QBD, as a function of the variable $x \in \rA$, which is an element of $\RR^3$.
By the definition of the matrices $A_0$, $A_1$ and $A_2$ in \eqref{sub}
(see also the example in \S \ref{secQBD}), these matrices are twice differentiable with respect to any of their elements.
By the definition of the rates
in \eqref{bd1}-\eqref{bd4}, which are the elements of the matrices $A_0$, $A_1$ and $A_2$, these matrix elements in turn
have constant partial derivatives with respect to each of the three real components of $x$ at each $x \in \rA$, i.e., with respect to $q_1$, $q_2$ and $z_{1,2}$.
It follows from Theorem 2.3 in He ~\cite{H93} that the rate matrix $R$ in \eqref{steady1}, which is the minimal
nonnegative solution to the quadratic matrix equation $A_0 + R A_1 + R^2 A_2 = 0$, is also twice differentiable
with respect to the matrix elements of $A_0$, $A_1$ and $A_2$, and thus also with respect to the three real components of $x$ at each $x \in \rA$.

It thus suffices to look at the derivatives with respect to one of the elements of the matrices $A_0$, $A_1$ or $A_2$.
It follows from the normalizing expression in \eqref{boundary}
and the differentiability of $R$, that $\af_0$ is also differentiable. Hence, from \eqref{desired2}, we see that
$\pi_{1,2}$ is differentiable at each $x \in \rA$, with
\bequ \label{piDeriv}
\pi_{1,2}' = \af_0' (I - R)^{-1} \mathbf{1_+} + \af_0 (I - R)^{-1} R' (I - R)^{-1} \mathbf{1_+}.
\eeq
By differentiating \eqref{boundary}, we have
\bequ \label{piDeriv2}
\af_0' (I - R)^{-1} \mathbf{1} + \af_0 (I - R)^{-1} R' (I - R)^{-1} \mathbf{1} = 0,
\eeq
so that $\af'_0$ is continuous. The continuity of $R'$ and $\af'_0$ with respect to one of the elements of the matrices $A_0$, $A_1$ or $A_2$
 implies that the derivative $\pi_{1,2}'$ with respect to one of the elements of the matrices $A_0$, $A_1$ or $A_2$ is finite and continuous on $\rA$,
which in turn implies that the partial derivatives with respect to the three real components of $x$ at each $x \in \rA$ are finite and continuous as well.
Hence, $\Psi$ is locally Lipschitz continuous on $\rA$, as claimed.

We next show that $\pi_{1,2}$ and thus $\Psi$ are locally Lipschitz continuous in neighborhoods of points in
$\AA^{+}$ within $\rS - \rS^-$ and of points in $\AA^{-}$ within $\rS - \rS^+$.  We will only consider $\AA^{+}$, because the two cases are essentially the same.
In both cases, the situation is complicated starting from \eqn{piDeriv}
because the entries of $\alpha_0 (x)$ become negligible, while the entries of $(I - R)^{-1} (x)$ explode as $x \ra x_b$.
However, the two different limits cancel their effect.  We exploit \eqref{piRep2}.
  The representation in \eqn{piRep2} is convenient because now $\af_0 (x) \ra \af_0 (x_b)$ as $x \ra x_b$, where $\af_0 (x_b)$
is finite.  All key asymptotics take place in $R^+$.

Since the crucial asymptotics involves only $\RR^+$, we see that we only need carefully consider one of the two regions, in this case the upper one.
To obtain results about $\RR^+$, from a process perspective, it suffices to replace the given QBD by a new QBD with the upper region and reflection at the lower boundary.
The new QBD model involving only $\RR^+$ is equivalent to a relatively simple single-server queue.  The net input is a linear combination of four Poisson processes,
and so has stationary and independent increments.  The queue length process in the revised model is an elementary $MAP/MSP/1$ queue, as in \S 4 of \cite{ACW94}, which has
as QBD representation with rate matrix $R^+$.

For the asymptotics, the key quantities are the spectral radii of the matrices $R^+ (x)$ and $R^- (x)$, say $\eta^+ (x)$ and $\eta^- (x)$, and the way that these depend on the
drifts $\delta_{+} (x)$ and $\delta_{-} (x)$ as $x \ra x_b$.  The spectral radius $\eta^+ (x)$ is the unique root in the interval $(0,1)$
of the equation $det[A_0^+ (x) + A_1^+ (x) \eta + A_2^+ (x) \eta^2] = 0$, and similarly for $\eta^- (x)$; see (39) on p. 241 of \cite{N86},
the Appendix of \cite{N89} and \S 4 of \cite{ACW94}.
We see that $\eta^+ (x) \ra \eta^+ (x_b) = 1$ and $\eta^- (x) \ra \eta^- (x_b) < 1$ as $x \ra x_b \in \AA^{+}$.
In general, we can represent powers of the matrix $R$ (and similarly for $R^+$ and $R^-$) asymptotically as
\bequ \label{eigen1}
R^n = v u \eta^n + o( \eta^n) \qasq n \tinf,
\eeq
where $u$ and $v$ are the left and right eigenvectors of the eigenvalue $\eta$, respectively, normalized so that $u \bf{1} = 1$ and $u v = 1$.
Moreover, as $\eta \ra 1$, the matrix inverse $(I-R)^{-1}$ is dominated by these terms.

Hence, we can do a heavy-traffic expansion of $\eta^+ (x)$ and the related quantities as $x \ra x_b \in \AA^{+}$ with $x \in \AA$,
as in \cite{CW94}; see the Appendix of \cite{N89}.
As $x \ra x_b$, all quantities in \eqn{piRep2} have finite continuous limits as $x \ra x_b \in \AA^{+}$ except $(I - R^+ (x))^{-1}$.
We first have $|\delta_{+} (x)| \ra 0$ and $\delta_{-} (x) \ra \delta_{-} (x_b)$, where $0 < \delta_{-} (x_b) < \infty$.
We then obtain
\bequ \label{expan}
\bsplit
1 - \eta^+ (x) & = c(x_b) |\delta_{+} (x)| + o ( |\delta_{+} (x)|)  \\
(I - R (x)^+)^{-1}  & = \frac{v^+ (x_b) u^+ (x_b)}{1 - \eta^+ (x)} +  o((1 - \eta^+ (x))^{-1}) \\
&  = \frac{v^+ (x_b) u^+ (x_b)}{c (x_b) |\delta_{+} (x)|} + o(|\delta_{+} (x)|^{-1})
\end{split}
\eeq
as $x \ra x_b$ and $|\delta_{+} (x)| \ra 0$, where $c$, $v^+$ and $u^+$ are continuous functions of $x_b$ on $\AA^{+}$.  The asymptotic relations in \eqn{expan}
together with \eqn{piRep2} imply that
\bequ \label{expan2}
|\pi_{1,2} (x) - \pi_{1,2} (x_b)|  = |\pi_{1,2} (x) - 1| =  |- r(x)/(1 + r(x)|,
\eeq
where
\bequ \label{expan3}
r (x)  \equiv \frac{\af_0^- (I - R^-)^{-1} \bf{1}}{\af_0^+ (I - R^+)^{-1} \bf{1}} \sim h(x_b)  |\delta_{+} (x)|
\eeq
as $x \ra x_b$ and $|\delta_{+} (x)| \ra 0$, where $h$ is a continuous function on $\AA^{+}$.
Hence, there exist constants $K_1$ and $K_2$ such that
\bequ \label{expan4}
|\pi_{1,2} (x) - \pi_{1,2} (x_b)| \le  K_1 |\delta_{+} (x)| \le K_2 \| x - x_b\|_3
\eeq
for all $x$ sufficiently close to $x_b$.
Finally, we can apply the triangle inequality with \eqn{expan4} to obtain $|\pi_{1,2} (x_1) - \pi_{1,2} (x_2)| \le 2 K_2 \| x_1 - x_2\|_3$
for $x_1, x_2$ in an $\epsilon$ ball about $x_b$ in $\rS - \rS^-$.
Hence, $\pi_{1,2} (x)$ and thus $\Psi$ are locally Lipschitz continuous on $\AA^{+}$ within $\rS -\rS^-$.
Hence the proof is complete.
\qed

\section*{Acknowledgments}
This research began while the first author was completing his Ph.D. in the Department of Industrial Engineering and Operations Research
at Columbia University and was completed while he held a postdoctoral fellowship at C.W.I. in Amsterdam.
This research was partly supported by NSF grants DMI-0457095 and CMMI 0948190.



\newpage

\begin{center}
\textbf{Appendix}
\end{center}

\begin{appendix}

\section{Overview}\label{secOverview}

In this appendix we present some supplementary material.
In \S \ref{secTransient} we analyze the system with an initial underloaded state.  In that case we show that the approximating fluid models
lead to our main ODE in finite time.
In \S \ref{secAlg2} we elaborate on the algorithm in \S \ref{secAlg} and give two more numerical examples, including one where the solution, starting empty,
first enters $\rS$ in $\rS^+$, and then moves from $\rS^+$ to $\AA$ and then $\rS^-$, with $\pi_{1,2}$ experiencing a discontinuity.
In \S \ref{secProofs} we give some omitted proofs.
Finally, in \S \ref{secCon} we draw conclusions and mention remaining open problems.

\section{Transient Behavior Before Hitting $\rS$} \label{secTransient}

Recall that the FQR-T control is designed to respond to unexpected overloads. We assume that the two classes
operate independently until a time at which the arrival rates change, and the system becomes overloaded.
Let $0$ be the time that the arrival rates change.
We thus think of a system in steady state at time $0$ when the arrival rates change, with
\bequ \label{initial}
q_1 (0) = q_2 (0) = z_{1,2}(0) = z_{2,1}(0) = 0.
\eeq
In particular, $q_1(0) \le \kappa$, and no sharing is taking place.
A well-operated system tends to have a critically loaded fluid limit, yielding steady-state values
$z_{1,1}(0) = m_1$ and $z_{2,2}(0) = m_2$, but we could also have an underloaded steady state, with
$z_{1,1}(0) < m_1$ and/or $z_{2,2}(0) < m_2$ as well.

The ODE in \eqref{ode}-\eqref{odeDetails} can be regarded as the fluid limit of a sequence of overloaded queueing models.
Class $1$ was assumed to be overloaded due to the arrival rate being larger than the total service rate of
service pool $1$, while class $2$ was overloaded either because its arrival rate was also too large (but less so than class $1$),
or because
pool $2$ was helping class-$1$ customers.
For the ODE, the system overload assumption translates into having $z_{1,1}(t) = m_1$ and $z_{1,2}(t) + z_{2,2}(t) = m_2$ for all $t$,
so that the state space for the fluid limit was taken to be $\rS$.
(The space $\rS$ was defined in \S \ref{secMain} and \S \ref{secUnique}.)
However, if either $z_{1,1}(0) < m_1$ or $z_{2,2}(0) < m_2$, then the initial state is not in $\rS$,
so we cannot use the ODE \eqref{ode} to describe the
system. There is a transient period $[0, t_{\rS})$ during which the two service pools fill up, but the system
is not yet overloaded.

If sharing is eventually going to take place (i.e., if $x^*$ is in either $\AA$ or $\rS^+$),
then with initial conditions as in \eqref{initial}, we should certainly hit $\rS^b$.
Sharing will begin only at a time $T$ such that $q_1(T) - rq_2(T) = \kappa$.
In this section we show that, if indeed $x^* \in \AA \cup \rS^+$, then $T < \infty$, where
\bequ \label{T}
T \equiv \inf \{ t \ge 0 : \mbox{$x(t) \in \rS^b$}\}.
\eeq

The transient period of the fluid system can be divided into two distinct periods:
The first transient period, on the interval $[0, T)$, lasts until the fluid limit hits $\rS^b$.
The second transient period is the one starting at the hitting time $T$, and is
described by the ODE \eqref{odeDetails}. This period was analyzed in the previous sections.
The first transient period is described by different ODE's, depending on the state of the system.
These ODE's, for the initial condition in \eqref{initial}, are given in the proof of Theorem \ref{thT} below.

We shall prove that $T < \infty$ under the extra assumption that at no time during $[0, T)$ is $z_{2,1} > 0$.
The assumption can be verified directly by solving the fluid model of the first transient period.
We discuss this condition after the proof of Theorem \ref{thT}.

\begin{theorem} \label{thT}
If $x^* \in \AA \cup \rS^+$, if \eqref{initial} holds and if $z_{2,1}(t) \equiv 0$ for all $t \ge 0$,
then $T < \infty$, for $T$ in \eqref{T}.
\end{theorem}

\begin{proof}
We start by developing the ODE to describe the system before hitting $\rS$.
As before, we do not consider the original queueing model and prove convergence to the appropriate fluid limit,
but instead we develop the ODE directly.
We first consider the case in $s^a_2 > 0$ (so that $q^a_2 = 0$), i.e., class $2$
experiences no overload by itself (before pool $2$ starts serving class-$1$ fluid).
First, there is an initial period in which the pools are being filled with fluid.
It is easy to see that as long as neither pool is full, the pool-content functions $z_{i,i} (t)$
behave as the fluid approximations for the number in system at time $t$ in an $M/M/\infty$ queueing model
with arrival rate $\lm_i$ and service rate $\mu_{i,i}$, $i = 1,2$; e.g., see \cite{PTW07} (where it assumed that $\lm = \mu$,
so that $\lm / \mu = 1$). Therefore, the system evolution is described by the pair of ODE's
\bes
\bsplit
\dot{z}_{1,1}(t) & = \lm_1 - \mu_{1,1}z_{1,1}(t), \quad z_{1,1}(0) = \zeta_1 \\
\dot{z}_{2,2}(t) & = \lm_2 - \mu_{2,2}z_{2,2}(t), \quad z_{2,2}(0) = \zeta_2,
\end{split}
\ees
and the unique solution to each ODE is
\bes
z_{i,i} (t) = \frac{\lm_i}{\mu_{i,i}} + \left( \zeta_i - \frac{\lm_i}{\mu_{i,i}} \right) e^{-\mu_{i,i}t}, \quad t \ge 0, \quad i = 1,2.
\ees
These ODE's describe the dynamics of the two classes until one of the pools is full, i.e.,
until the time
\bequ \label{t1}
t_1 \equiv \min_{i = 1,2}\inf\{t \ge 0 : z_{i,i}(t) = m_i \}.
\eeq
Since we assume that $s^a_2 > 0$, $t_1$ is the time at which $z_{1,1}(t) = m_1$,
and at this time we need to start considering $q_1$.
Clearly, $q_1$ evolves independently of class $2$ until $q_1(t) = \kappa$ (when sharing is initialized). Let
\bequ \label{t2}
t_2 \equiv \inf \{ t \ge t_1 : q_1(t) = \kappa \}.
\eeq
Recall that $\kappa$ may be equal to $0$, in which case $t_1 = t_2$. If $t_2 > t_1$, then
$q_1(t)$, $t \in [t_1, t_2)$, evolves as the fluid approximation for the queue-length process in an Erlang-A model
operating in the ED MS-HT regime, as in \cite{W04}.
The ODE describing the evolution of $q_1$ is
\bequ \label{qErlang}
\bsplit
\dot{q}_1(t)  = \lm_1 - \mu_{1,1}m_1 - \theta_1 q_1(t), \quad t_1 \le t < t_2, \quad \mbox{with} \quad
q_1(t_1)  = 0,
\end{split}
\eeq
and its unique solution is
\bes
q_1(t) = \frac{\lm_1 - \mu_{1,1}m_1}{\theta_1} \left( 1 - e^{-\theta_1 (t - t_1)} \right), \quad t_1 \le t < t_2.
\ees
Now, since $q_1(t_2) = \kappa$ and $q_2(t_2) = 0$, class-$1$ fluid starts flowing to service pool $2$, so that $z_{1,2}$ starts increasing.
There is a time $t_3$ such that, for $t \in [t_2, t_3)$, $q_1(t) = \kappa$ , $q_2(t) = 0$ and all the excess class-$1$ fluid, that is not
lost due to abandonment, is flowing to pool $2$. Hence, $z_{1,2}$ satisfies the ODE
\bes
\bsplit
\dot{z}_{1,2}(t)  = (\lm_1 - \mu_{1,1} m_1 - \theta_1 \kappa) - \mu_{1,2}z_{1,2}(t), \quad t_2 \le t < t_3, \quad \mbox{with} \quad
z_{1,2}(t_2)  = 0,
\end{split}
\ees
whose unique solution is
\bes
z_{1,2}(t) = \frac{\lm_1 - \mu_{1,1} m_1 - \theta_1 \kappa}{\mu_{1,2}} \left( 1 - e^{-\mu_{1,2}(t - t_2)} \right), \quad t_2 \le t < t_3.
\ees
Hence, $t_3 \equiv \inf \{ t \ge t_2 : z_{1,2}(t) + z_{2,2}(t) = m_2 \}$, so that at time $t_3$ both service pools are full, with
$q_1(t_3) = \kappa$, $q_2(t_3) = 0$ and $q_1(t_3) - rq_2(t_3) = \kappa$.
It follows that $t_3$ is the time at which the fluid model hits the space $\rS^b$, and the first transient period is over,
i.e., $t_3 = T$ for $T$ in \eqref{T}.

Now we consider the second case in which $q^a_2 > 0$. In this case there are different scenarios:
In the first scenario, pool $2$ can be filled before pool $1$,
so that $t_1 = \inf \{ t \ge 0 : z_{2,2} = m_2 \}$, for $t_1$ in \eqref{t1}. In that case $q_2$ begins to increase at time $t_1$,
evolving according to the ODE of the overloaded Erlang-A model
\bes
\dot{q}_2(t) = \lm_2 - \mu_{2,2}m_2 - \theta_2 q_2(t).
\ees
However, by the assumption of the theorem, we have ruled out the case in which $q_1(t) - r_{2,1} q_2(t) = \kappa_{2,1}$,
 so that no class-$2$ fluid will flow
to pool $1$.  Hence, from the beginning (time $0$), $z_{1,1}$ increases until time $t_1' \ge t_1$ at which $z_{1,1} = m_1$.
Then $q_1$ increases, satisfying \eqref{qErlang} with $q_1 (t_1') = 0$.
By the assumption on $x^*$, and following Corollary \ref{corRegions}, there exists a time $T < \infty$ such that
$q_1(T) - rq_2(T) = \kappa$. This is because $r q_2(t) \le r q_2^a < q^a_1 - \kappa$ for all $t \le T$. On the other hand,
it follows trivially from the solution to \eqref{qErlang}, that $q^a_1$ is the globally asymptotically stable point of \eqref{qErlang}.
Hence, for every $\ep > 0$, there exists $t_{\ep}$ such that $q_1(t) > q^a_1 - \ep$ for all $t \ge t_{\ep}$. (This is because, by the
initial conditions, $q_1(t) \le q_1^a$ for all $t$).
Thus, we can find $\ep > 0$ such that
\bequ \label{ineq3}
r q^a_2 < q^a_1 - \ep - \kappa < q_1(t) - \kappa \ \mbox{for all $t \ge t_{\ep}$}.
\eeq

The second scenario of the second case has pool $1$ filled first at time $t_1$, so that $q_1$ starts increasing according to \eqref{qErlang}.
If $q_1$ reaches
$\kappa$ before $q_2$ starts increasing, then we have the same behavior as when $s^a_2 > 0$.
However, if at time $t_2$ in \eqref{t2} $q_2 > 0$, then the two queues will continue increasing independently until time $T$.
Once again, \eqref{ineq3} can be shown to hold, so that $T < \infty$.
\end{proof}

We can easily calculate the exact value of $x(T)$ and use it to calculate the QBD drift rates $\delta_{+}(x(T))$ and $\delta_{-}(x(T))$ to find
whether the positive-recurrence condition \eqref{posrec} holds at $T$, so that $x(T) \in \AA$.

\begin{remark}{$($sharing in the wrong direction$)$}\label{rmWrong}
{\em
In Theorem \ref{thT} we assumed that we never have $z_{2,1} > 0$. The reason is that, if $z_{2,1}$ ever does become positive, then the
fluid $x$ never hits the region $\rS$.
To see that this is so, suppose that for some time $t_4$ sharing is initialized, with class-$2$ fluid flowing to service pool $1$.
Then $z_{2,1}$ is increasing until a time $t_5$ at which $q_1(t_5) - r q_2(t_5) = \kappa$, and the AP begins to operate.
At that time, $z_{2,1}$ will start decreasing according to the ODE
\bes
\dot{z}_{2,1}(t) = - \mu_{2,1}z_{2,1}(t), \quad  t \ge t_5,
\ees
whose unique solution is
\bequ \label{z21}
z_{2,1}(t) = z_{2,1}(t_5) e^{- \mu_{2,1}(t - t_5)}, \quad t \ge t_5.
\eeq
Hence $z_{2,1}$ remains strictly positive for all $t \ge t_5$, and $\rS$ is never hit.

Of course, the fluid state should be approaching a state in $\rS$ as $t$ increases.  However, if there is such a limit point,
then that limit point itself typically will {\em not} be a stationary point, because if $x$ did start at that limit point, then it
will have to continue to move toward the final stationary point $x^{*}$.

More generally, the failure of $z_{2,1}$ to actually reach $0$ in finite time has practical implications for the FQR-T control
in the original queueing system.  It suggests that it should be beneficial to
relax the one-way sharing rule, by introducing lower positive thresholds for
$z_{1,2}$ and $z_{2,1}$. For example, if $z_{2,1} (t) > 0$ at some time $t \ge 0$, and at the same time
sharing should be done in the other direction (because of a new overload incident, with class $1$ being more overloaded
and needing to get help), then we will allow pool $2$ to start helping class $1$, provided that $z_{2,1}$ is smaller
than some threshold $s_{2,1} > 0$.
In that case, if $z_{2,1} (t) > s_{2,1}$, then $z_{2,1}$ will cross the threshold $s_{2,1}$ in finite time,
as can be seen from \eqref{z21}.
It remains to examine the system performance in response to such more complex transient behavior.
}
\end{remark}


For the cases covered by Theorem \ref{thT},
the system evolution over the entire halfline $[0, \infty)$ is a continuous ``soldering'' of the different ODE's, but at the soldering points $t_i$,
the functions under consideration are typically not differentiable.  Hence, there is no single ODE that captures the full dynamics
of the system.  To see why, consider the case
in which $s^a_2 > 0$ and $\kappa > 0$. Then, for $t < t_1$, $q_1(t) = 0$ and $\dot{q}_1 = 0$, but for
$t_1 \le t < t_2$, $q_1(t)$ evolves according to \eqref{qErlang}, which typically has a strictly positive
derivative at $t_1$.  Thus the left and right derivatives at $t_1$ are not equal.
Similar arguments hold for all the other soldering points.

\section{The Algorithm And More Examples}\label{secAlg2}

\subsection{More on the Algorithm}

Let $\{t_m : m = 0,1,2, \dots, n \}$ be the Euler steps, with $t_{m+1}-t_m = h$. In our experiments we found
$h = 0.01$ to be a good candidate for the step size
since it is small enough to minimize numerical errors, while the number of iterations needed for the ODE to reach its stationary point,
is just a few thousands. Hence the algorithm takes only a few seconds to terminate.

Let $\bar{D}(t) \equiv q_1(t) - r q_2(t)$, denote the weighted difference between the two fluid queues.
The discretization of the ODE in the numerical algorithm means that if, at step $k-1$, $\bar{D}(t_{k-1}) \notin \rS^b$ but is close to it,
then $\bar{D}(t_{k})$ may miss the boundary, even though the (continuous) ODE is at the boundary at time $t_{k}$.
For that reason, if $\kappa - h < \bar{D}(t_k) < \kappa + h$, then we force $x(t_k)$ to be in $\rS^b$,
by taking $\bar{D}(t_k) = \kappa$. Once we have $\bar{D}(t_k) = \kappa$ we decide whether to keep staying on the boundary
for the next Euler step, by checking whether \eqref{posrec} holds. According to the relation between the QBD drift rates
at time $t_k$, we decide whether we should apply the AP, in order to find $\pi_{1,2}(t_k)$, or rather set
$\pi_{1,2}(t_k)$ to zero or one.

At any step in the algorithm, we must also decide which ODE to use.
That depends on the state of the system at each time, as described in \S \ref{secTransient}.
If the fluid state is not in $\rS$, as in the initial period of the example in \S \ref{secAlg} and the example below, then
we use the appropriate fluid model, as given in the proof of Theorem \ref{thT}.


\subsection{An Example with $x^* \in \rS^+$}

We now consider the same example as in \S \ref{secExample}, except now we increase the arrival rate
for class $1$ substantially, so that $x^* \in \rS^+$.   In particular, we let $\lambda_1 = 3.0$ instead of
$1.3$.  Once again,
the system is initialized empty. That means that the fluid solution in $\rS$ is moving between the two regions $\rS^b$
and $\rS^+$. In particular, the solution first hits $\rS^b$, as was proved in Theorem \ref{thT},
but it stays there for a short amount of time, and then crosses to $\rS^+$.

As before, we show the results multiplied by $n = 1000$ in the figures below.
We see how $z_{2,2}$ starts increasing up to the time $T$ in which $z_{1,2}(T) + z_{2,2}(T) = m_2$.
At this time $z_{2,2}(T)$ starts decreasing, and is replaced by class-$1$ fluid.
Since no class-$2$ fluid is flowing to either of the service pool,
all the class-$2$ fluid output is due to abandonment.
We can also observe that $z_{2,2}$ eventually hits $0$, even though $z_{2,2}$ satisfies the equation \eqref{z21}.
This is due to the numerical errors, as described in \S \ref{secTransient}.

In steady-state we have
$q^*_2 = \lm_2 / \theta_2 = 900/0.3 = 3000$ and $q^*_1 = (\lm_1 - m_1\mu_{1,1} - m_2\mu_{1,2})/\theta_2 = 4000$,
as in Corollary \ref{corStat} $(ii)$.

\begin{figure}[h!]
  \hfill
  \begin{minipage}[t]{.45\textwidth}
    \begin{center}
      \epsfig{file=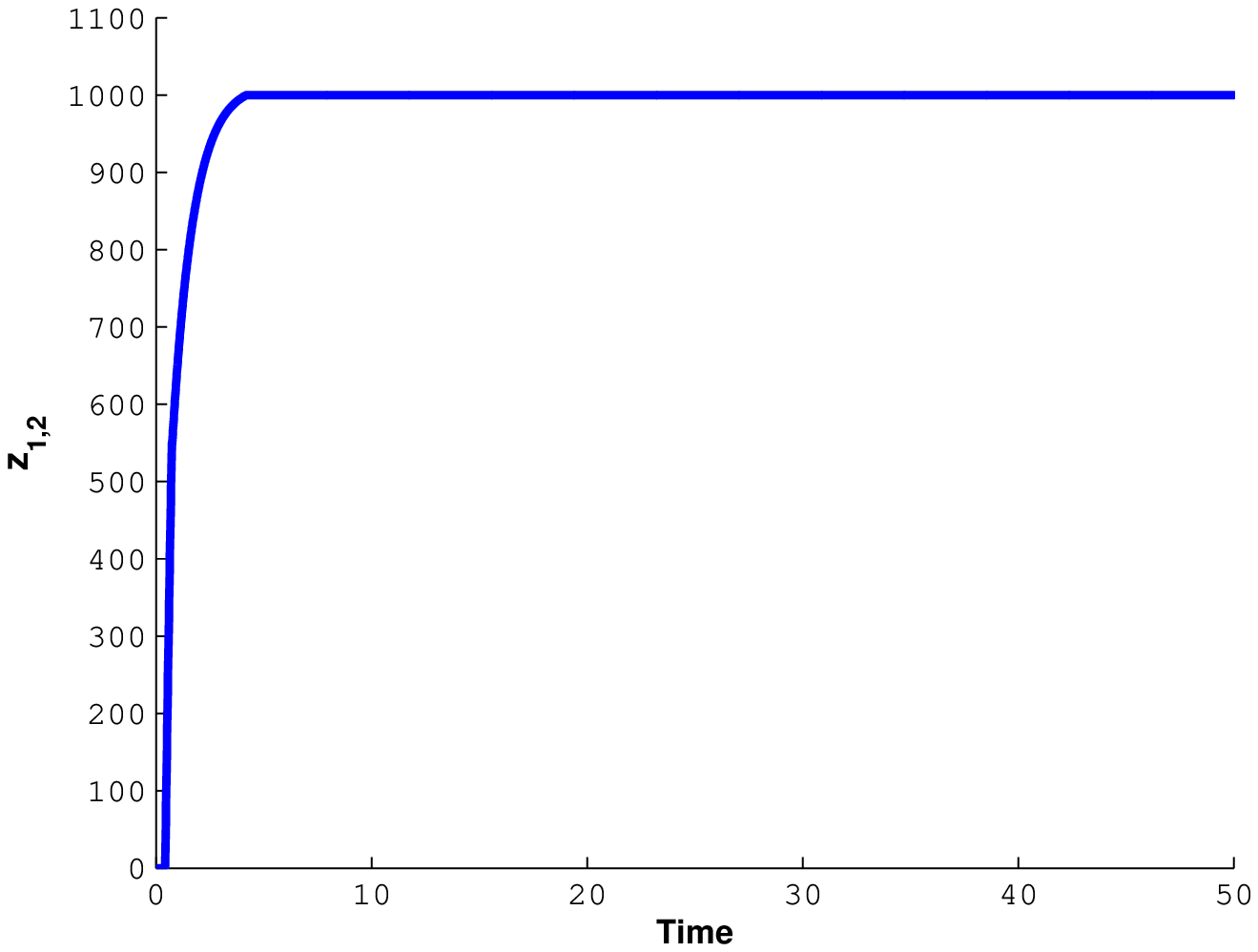, scale=.5}
      \caption{$z_{1,2}$ when $\lm_1$ exceeds the system's capacity.}
      \label{figQBD2Z12}
    \end{center}
  \end{minipage}
  \hfill
  \begin{minipage}[t]{.45\textwidth}
    \begin{center}
      \epsfig{file=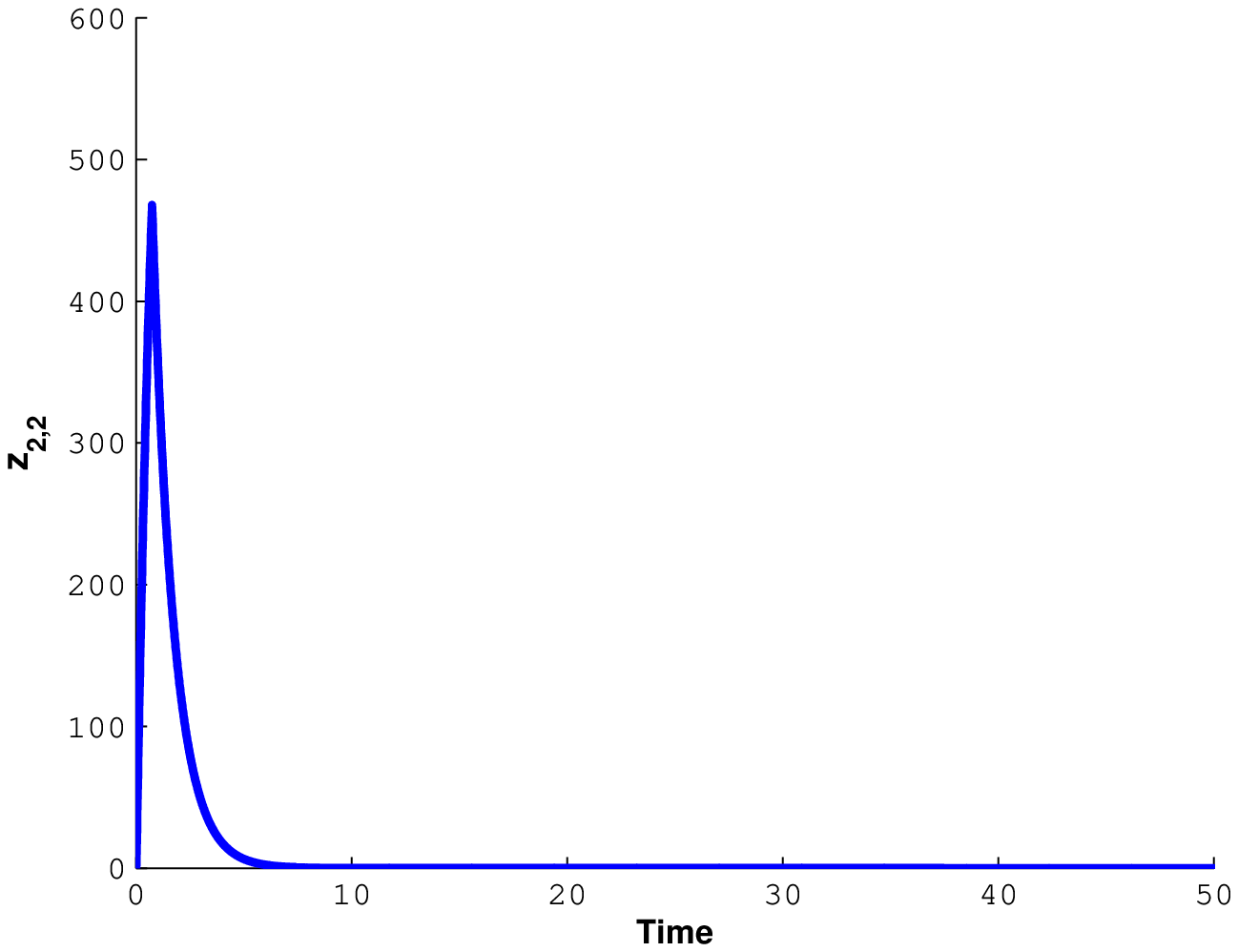, scale=.5}
      \caption{$z_{2,2}$ when $\lm_1$ exceeds the system's capacity.}
      \label{figQBD2pi}
    \end{center}
  \end{minipage}
  \hfill
\end{figure}

\begin{figure}[h!]
  \hfill
  \begin{minipage}[t]{.45\textwidth}
    \begin{center}
      \epsfig{file=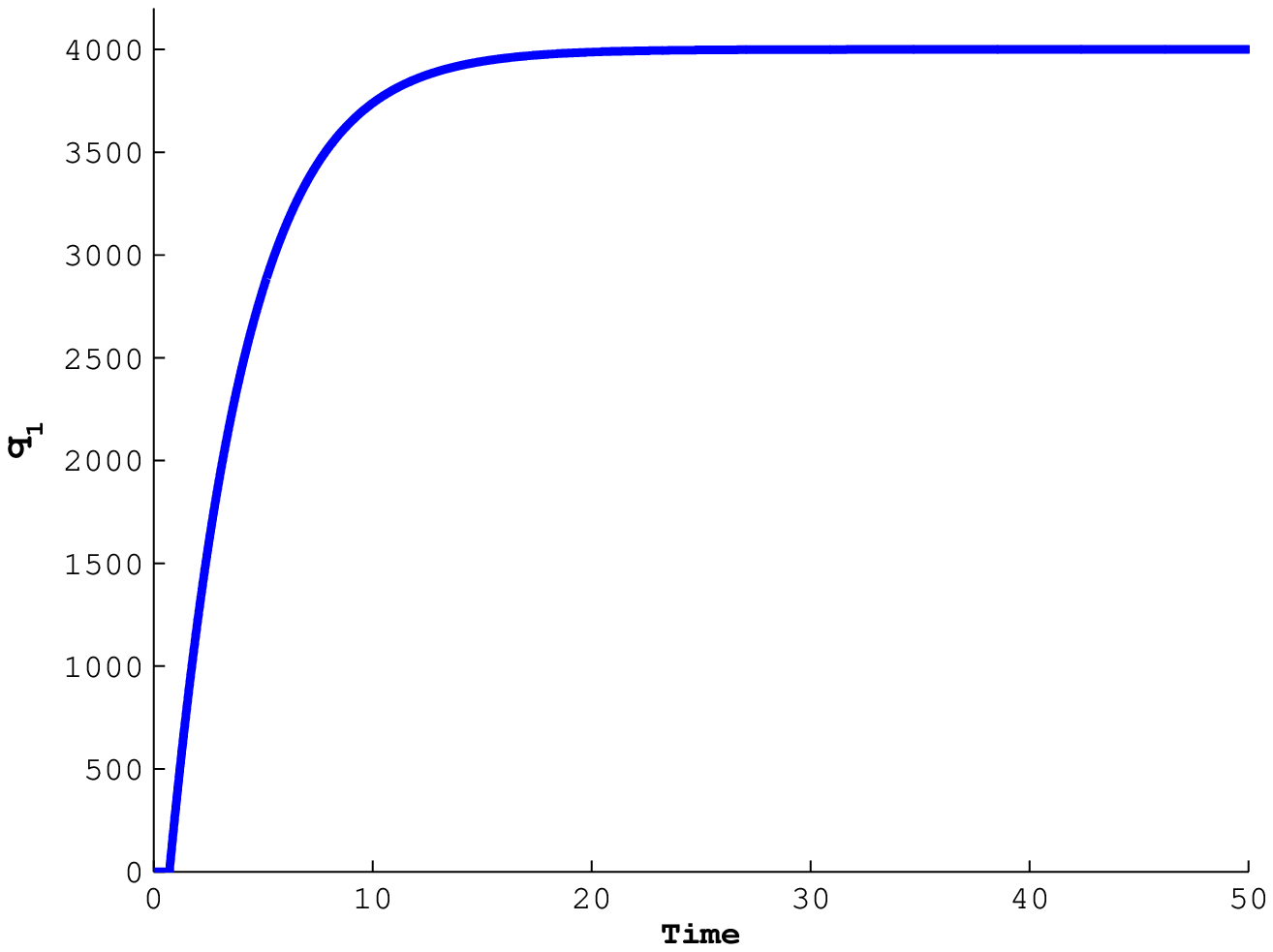, scale=.5}
      \caption{$q_1$ when $\lm_1$ exceeds the system's capacity.}
      \label{figQBD2Q1}
    \end{center}
  \end{minipage}
  \hfill
  \begin{minipage}[t]{.45\textwidth}
    \begin{center}
      \epsfig{file=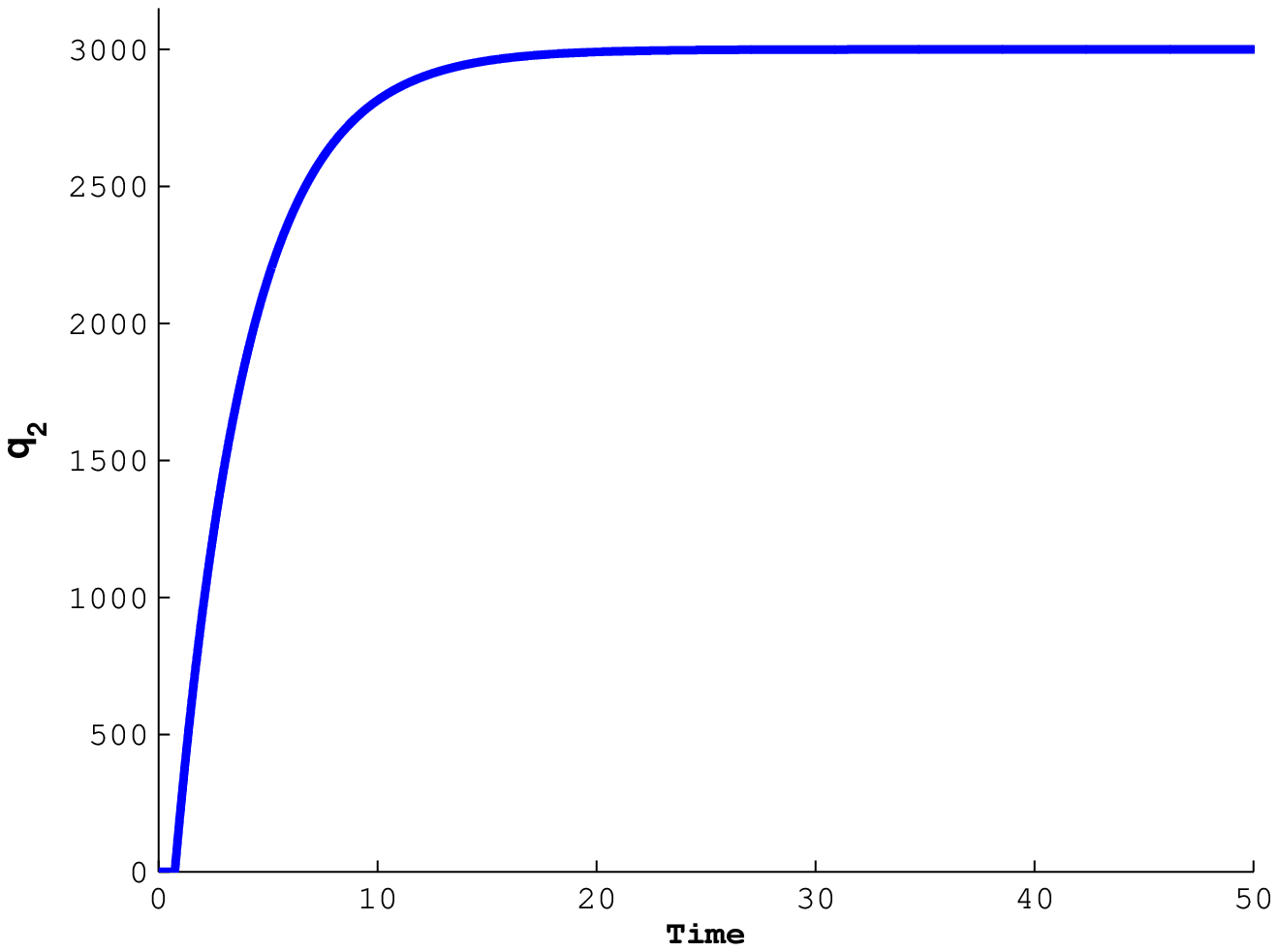, scale=.5}
      \caption{$q_2$ when $\lm_1$ exceeds the system's capacity.}
      \label{figQBD2Q2}
    \end{center}
  \end{minipage}
  \hfill
\end{figure}

\subsection{An Example With $x^* \in \rS^-$ Moving Through $\rS^+$ And $\rS^b$}

The purpose of this example is to illustrate more complex dynamics.
We make class $2$ more overloaded than class $1$, i.e., $q^a_1 < q_2^a$, but we make the rates
faster for class $1$.
Specifically, we considered the following model parameters:
$\lm_1 = 13.0$, $\lambda_2 = 1.5$, $\mu_{1,1} = 10.0$, $\mu_{1,2} = 0.8$, $\mu_{2,2} = 1$
$\theta_1 = 2$, $\theta_2 = 0.2$, $r = 0.8$ and $\kappa = 0$.
Note that the arrival, service and abandonment rates are all substantially greater for class $1$ than for class $2$.
Nevertheless, class $2$ is more overloaded than class $1$:  $q^a_1 = 1.5 < 2.5 = q^a_2$.
For this example, $x^* = (1.5, 2.5, 0) \in \rS^-$.

We applied our algorithm to this example, letting
the system start empty, i.e., $x(0) = 0$.
The results are shown in the remaining figures, where here the results are scale up by multiplying by $n = 100$.
Since the class-$1$ arrival rate is so large,
$q_1$ starts filling up rapidly, and becomes full first; see Figures \ref{fig_z12} and \ref{fig_z22}.
Since $\kappa = 0$, pool $2$ starts helping class $1$ as soon as pool $1$ becomes full.
At first, pool $2$ has spare capacity.
However, soon the spare capacity in pool $2$ is exhausted.
At that time, the solution hits $\rS$.  Even at the time pool $2$ becomes full, we have $q_1 > r q_2$, so that the solution enters $\rS$ via $\rS^+$,
Thus pool $2$ continues to help class $1$ even after it is fully occupied, causing a dip in $z_{2,2}$; see Figure \ref{fig_z22}.
However, the ratio of the queue lengths $q_1/q_2$ decreases from its peak of about $1.4$
until it reaches the target ratio $r = 0.8$,
producing the desired relation $q_1 = r q_2 + \kappa$; see Figure \ref{fig_ratio}
  At that time (about $t = 1.15$, the solution
that was in $\rS^+$ hits the set $\AA$.  At that time, $\pi_{1,2} (x)$
jumps from $1$ down to a value about equal to $0.6$; see Figure \ref{fig_pi}.
For an interval of time, the solution remains in $\AA$ with the queue ratio fixed at the target $r = 0.8$.
However, the load imbalance cause the solution to move within $\AA$, causing $\pi_{1,2} (x)$ to decrease until it reaches $0$ in the set $\AA^-$,
at about time $t = 2.5$.
From $\AA^-$, the solution moves immediately into $\rS^-$, where it rapidly converges to its stationary point.
Of course, the stationary point $x^*$ is not actually reached in finite time.
Indeed, after $S^-$ is reached, $z_{1,2}$ decreases exponentially to $0$, but $z_{1,2} (t) > 0 = z^*$ for all $t$,
consistent with Remark \ref{rmWrong}.

From the figures, it is evident that the numerical solution is not identical to the real solution.
Because of the discrete step sizes in the Euler steps, the numerical solution misses $\AA$ initially.
In fact, we have to design the algorithm such that it ``discovers'' when $\rS^b$ is missed, and then force
it to hit $\rS^b$. That is easy to do since, if $x (t_k) \in \rS^+$ and $x (t_{k+1}) \in \rS^-$,
where $t_k$ is the time of the $k^{\rm th}$ Euler step, $k \ge 1$, then $\rS^b$ must have been missed.
We can then compute $q_1 (t_{k+1})$ and take $r q_2 (t_{k+1}) = q_1 (t_{k+1}) - \kappa$.

This discreteness of the numerical solution explains the erratic behavior of $\pi_{1,2}$ at the hitting time of $\AA$,
shown in Figure \ref{fig_pi}.
The thick vertical line just after time $1$, exactly when $r = 0.8$ for the first time as can be seen from Figure \ref{fig_ratio},
appears because $\pi_{1,2}$ jumps between $0$ and $1$ at each Euler step.
These jumps are caused the solution missing $\rS^b$ at first.
\begin{figure}[h!]
  \hfill
  \begin{minipage}[t]{.45\textwidth}
    \begin{center}
      \epsfig{file=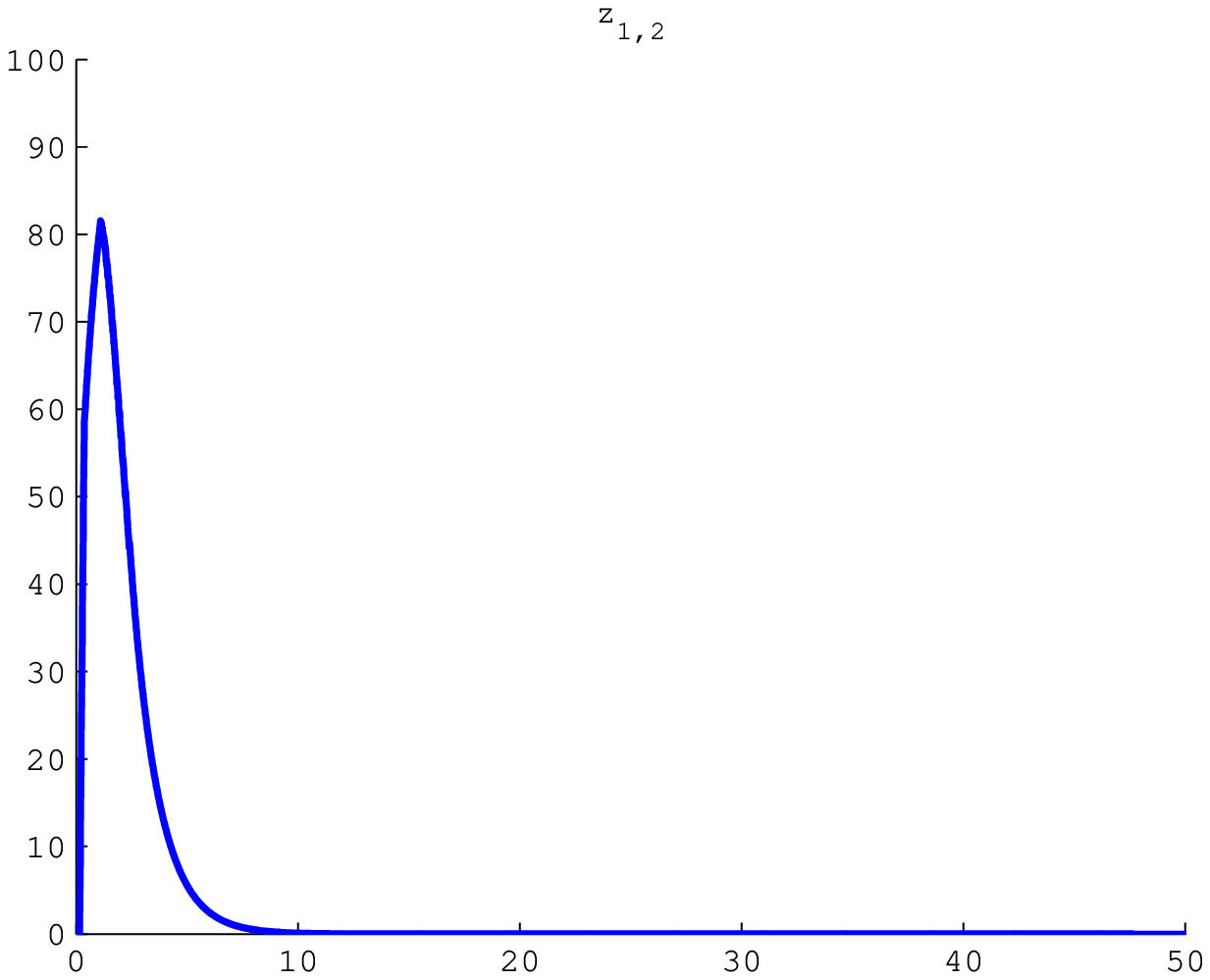, scale=.5}
      \caption{$z_{1,2}$ when $x^* \in \rS^-$.}
      \label{fig_z12}
    \end{center}
  \end{minipage}
  \hfill
  \begin{minipage}[t]{.45\textwidth}
    \begin{center}
      \epsfig{file=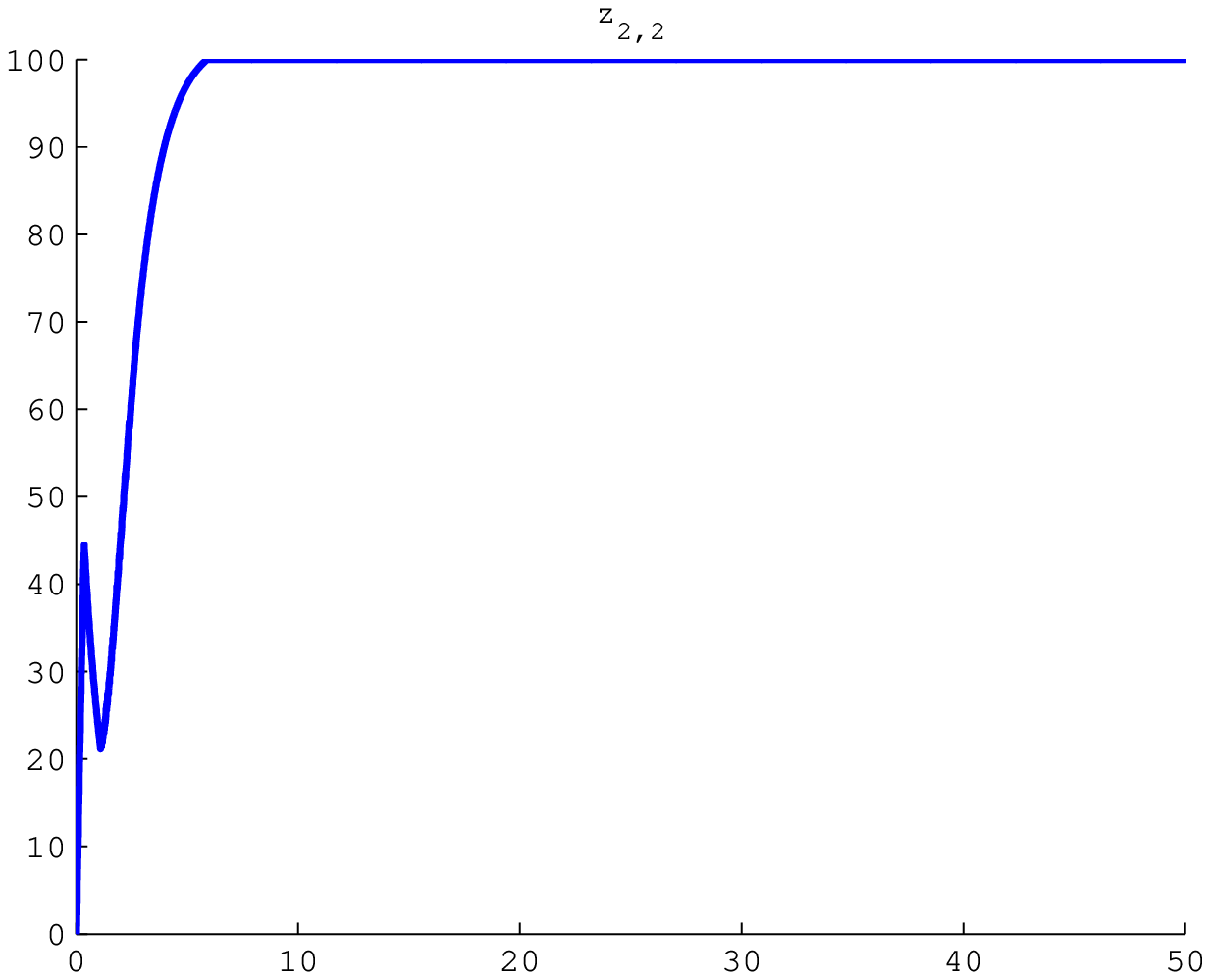, scale=.5}
      \caption{$z_{2,2}$ when $x^* \in \rS^-$.}
      \label{fig_z22}
    \end{center}
  \end{minipage}
  \hfill
\end{figure}

\begin{figure}[h!]
  \hfill
  \begin{minipage}[t]{.45\textwidth}
    \begin{center}
      \epsfig{file=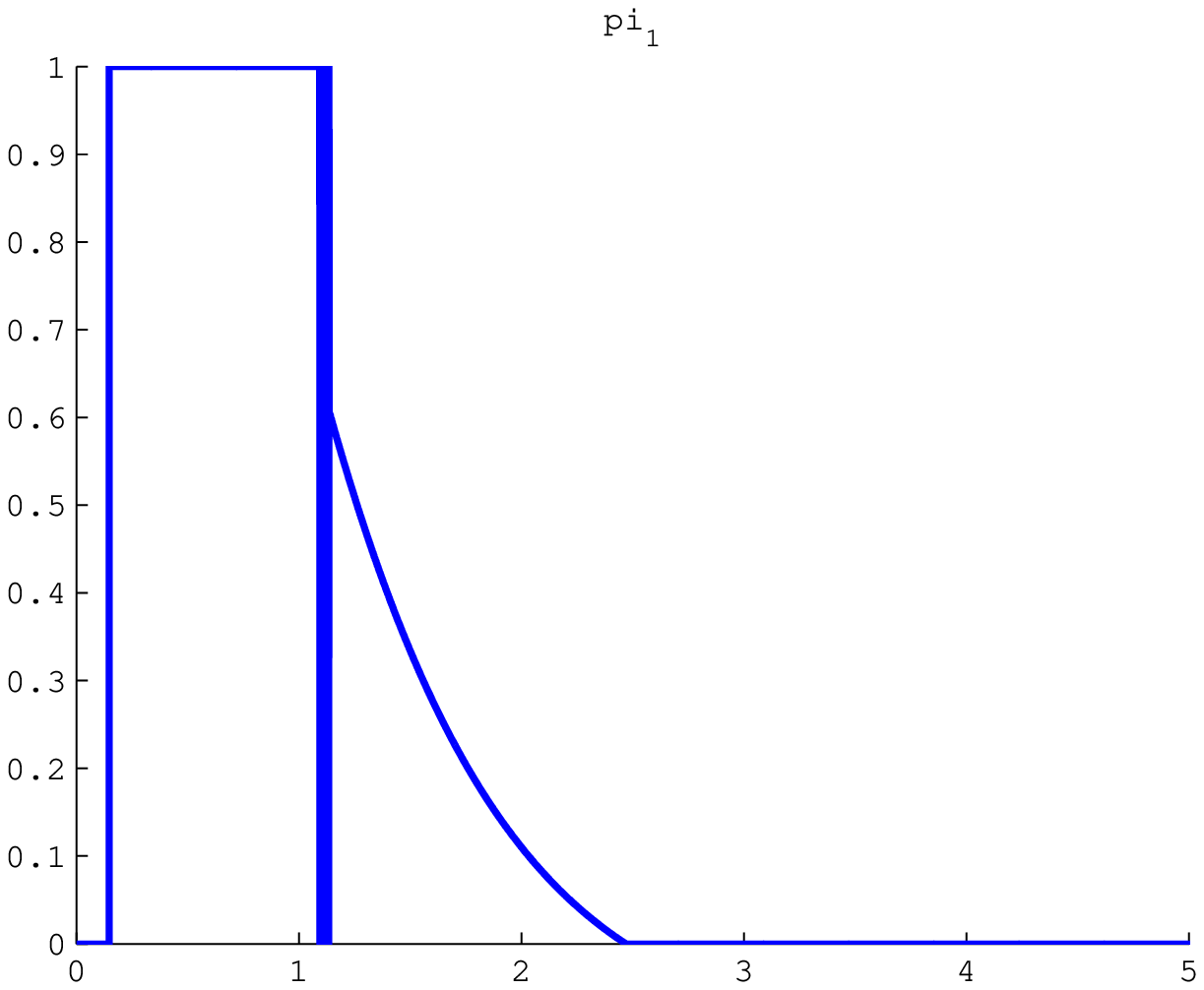, scale=.5}
      \caption{$\pi_{1,2}$ when $x^* \in \rS^-$ over short initial interval.}
      \label{fig_pi}
    \end{center}
  \end{minipage}
  \hfill
  \begin{minipage}[t]{.45\textwidth}
    \begin{center}
      \epsfig{file=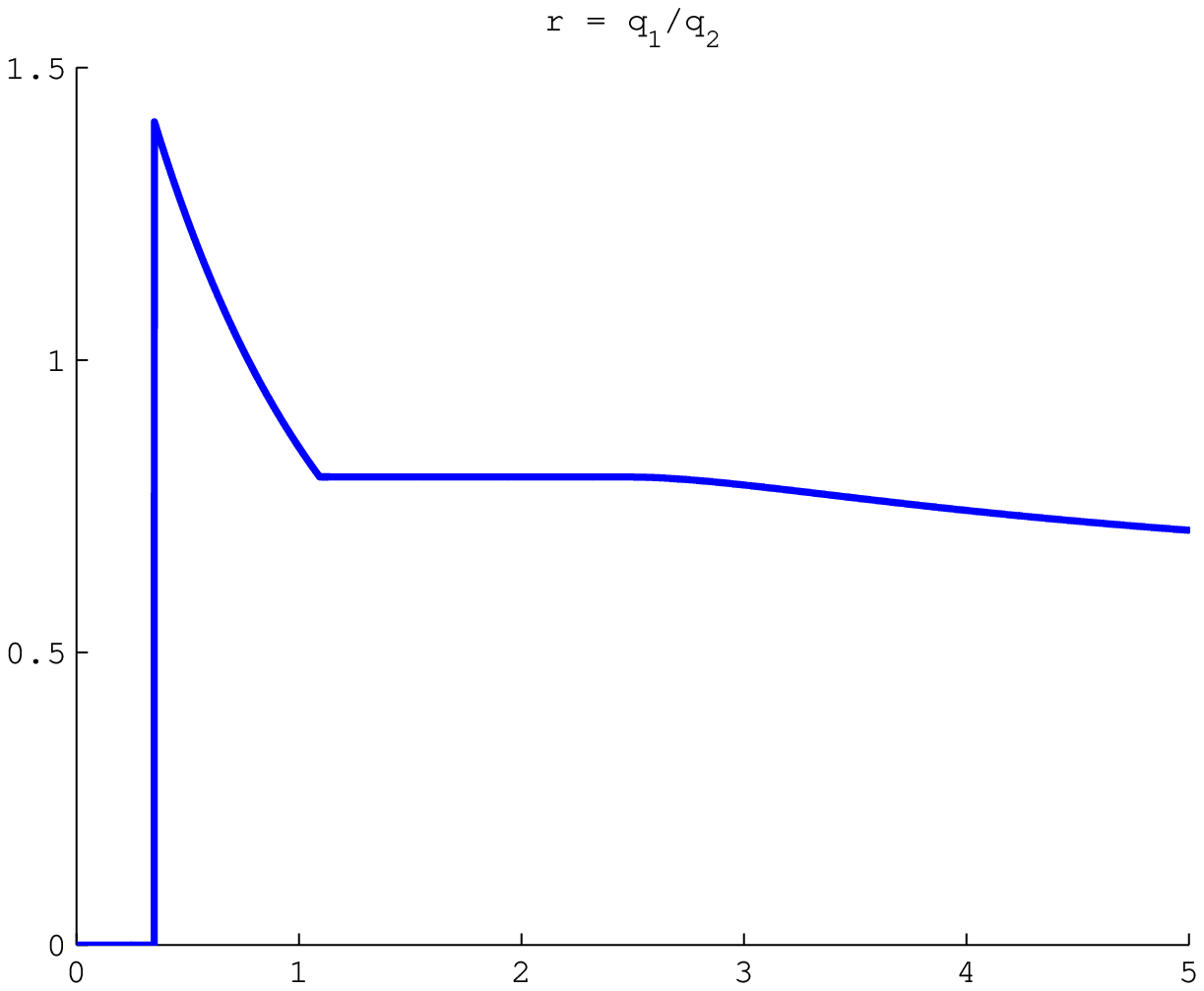, scale=.5}
      \caption{ratio between the queues when $x^* \in \rS^-$ over short initial interval.}
      \label{fig_ratio}
    \end{center}
  \end{minipage}
  \hfill
\end{figure}

\begin{figure}[h!]
  \hfill
  \begin{minipage}[t]{.45\textwidth}
    \begin{center}
      \epsfig{file=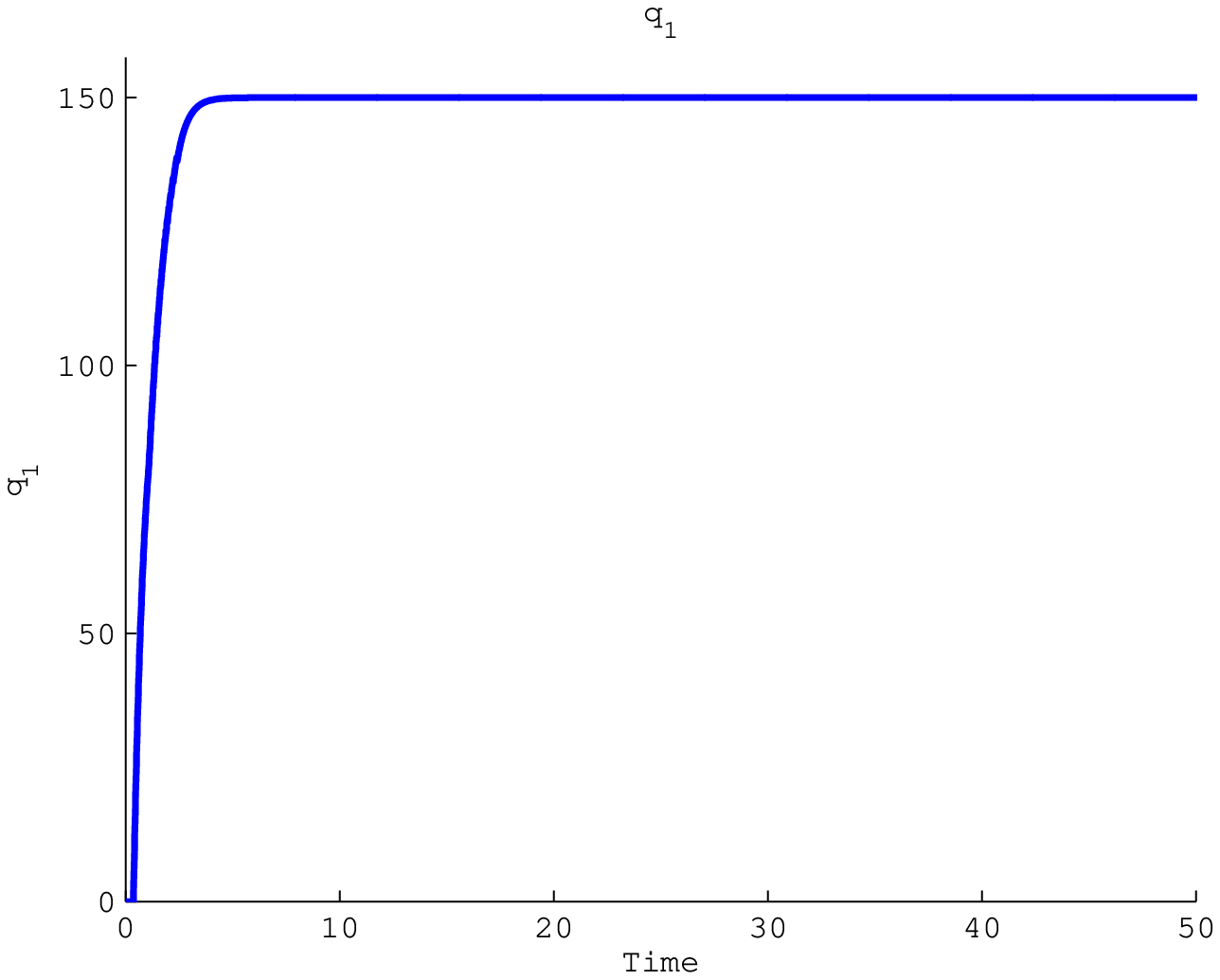, scale=.5}
      \caption{$q_1$ when $x^* \in \rS^-$.}
      \label{fig_q1}
    \end{center}
  \end{minipage}
  \hfill
  \begin{minipage}[t]{.45\textwidth}
    \begin{center}
      \epsfig{file=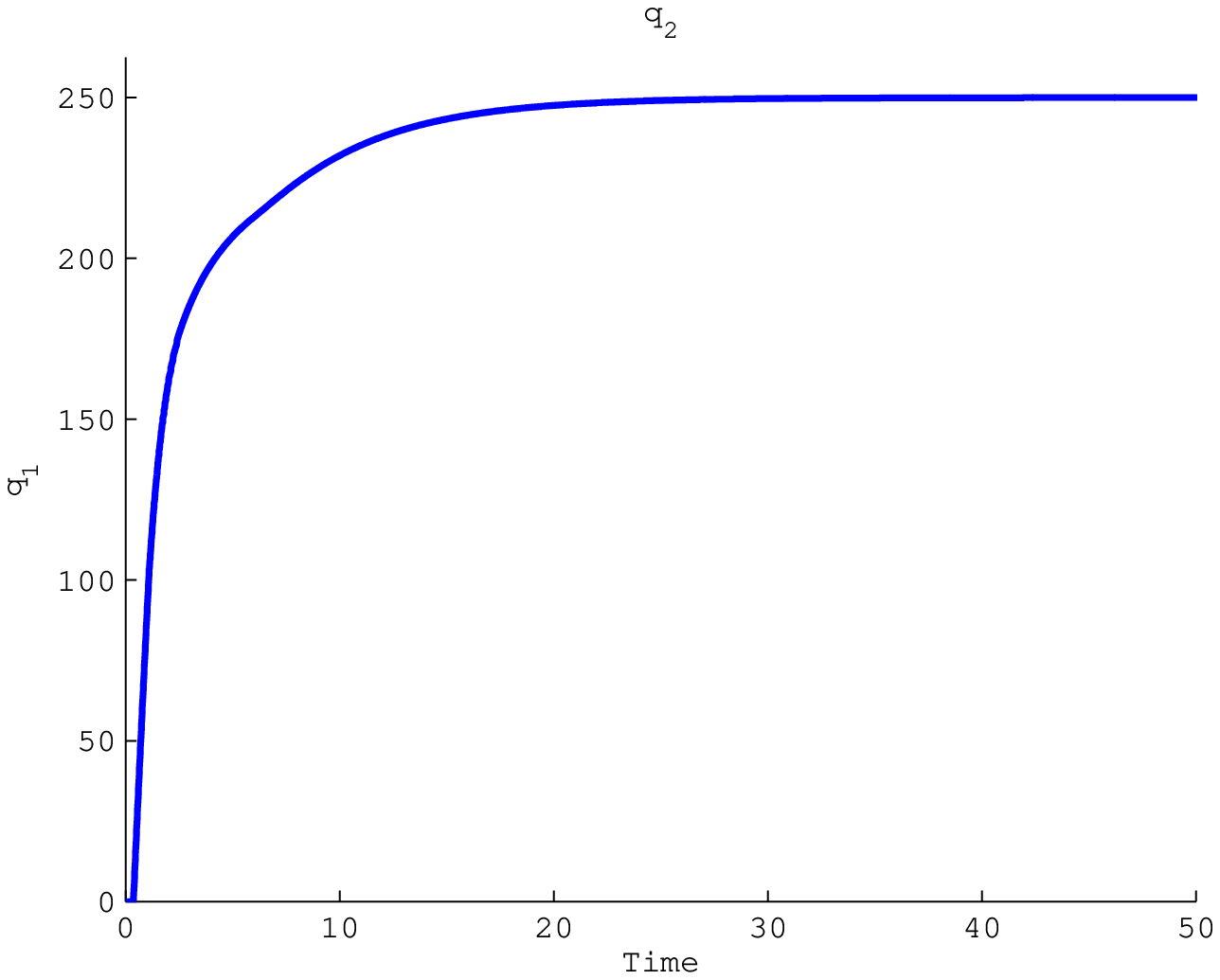, scale=.5}
      \caption{$q_2$ when $x^* \in \rS^-$.}
      \label{fig_q2}
    \end{center}
  \end{minipage}
  \hfill
\end{figure}

\section{Missing Proofs}\label{secProofs}

\paragraph{\sc Proof of Theorem \ref{thBound}}
Since $0 \le z_{1,2} \le m_2$ and $q_i \ge 0$ in $\rS$,
 we only need to prove the upper bounds \eqref{bound1}.
For $i = 1,2$, let $u_i(t)$ be the function describing the queue-length process (of queue $i$) in a modified system with no service processes
(so that all the fluid output is due to abandonment).  The queue-length process in the modified system evolves according
to the ODE
\bes
\dot{u}_i(t) = \lm_i - \theta_i u_i(t), \quad t \ge 0,
\ees
whose solution is
\bes
u_i(t) = \frac{\lm_i}{\theta_i} + \left( u_i(0) - \frac{\lm_i}{\theta_i} \right) e^{-\theta_1 t}, t \ge 0.
\ees
It follows that $u_i(t) \le u_i(0) \vee \lm_i / \theta_i$ and, when $u_i(0) = q_i(0)$, the the right-hand side in \eqref{bound1}
is an upper bound for $u_i(t)$.
We now show that this is also a bound for $q_i (t)$.  For that purpose,
define the auxiliary function $f_i(t) \equiv q_i(t) - u_i(t)$, $t \ge 0$, and observe
that $f_i(0) = 0$ and $\dot{f}_i(0) < 0$.
Hence, $f$ is decreasing at $0$ with $f(t) < f(0)$ for all $t \in [0, \delta)$ for some $\delta > 0$.
This implies that $q_i(t) < u_i(t)$ for all $t \in [0, \delta)$.

We now want to show that $q_i(t) \le u_i(t)$ for all $t \ge 0$.  For a proof by contradiction,
assume that there exists some $t_0 > 0$ such that $q_i(t_0) > u_i(t_0)$, and let
\bes
t_1 \equiv \sup\{ t < t_0 : q_i(t) = u_i(t) \}, \qquad
t_2 \equiv \inf\{ t > t_0 : q_i(t) = u_i(t) \}.
\ees
By the contradictory assumption and the continuity of $q$ and $u$, we have $0 < t_1 < t_0 < t_2$. ($t_2$ may be infinite.)
Then
\bequ \label{q-u}
q_i(t) > u_i(t)  \qforallq t_1 < t < t_2.
\eeq
It follows from the mean-value theorem that there exists some $t_3 \in (t_1, t_0)$ such that
\bes
\dot{f}_i(t_3) = \frac{f(t_0) - f(t_1)}{t_0 - t_1} = \frac{f(t_0)}{t_0 - t_1} > 0.
\ees
Hence, $\dot{q}_i(t_3) > \dot{u}_i(t_3)$. For $i = 1$, this translates to
\bes
\lambda_1  - \mu_{1,1} m_1 - \pi_{1,2} (x(t_3))\left[z_{1,2} (t_3) \mu_{1,2} + z_{2,2} (t_3) \mu_{2,2}\right] - \theta_1 q_1 (t_3)
> \lm_1 - \theta_1 u_1(t_3).
\ees
Thus,
\bes
\theta_1 (q_1(t_3) - u_1(t_3)) < - \mu_{1,1} m_1 - \pi_{1,2} (x(t_3))\left[z_{1,2} (t_3) \mu_{1,2} + z_{2,2} (t_3) \mu_{2,2}\right]
< 0,
\ees
so that $q_1(t_3) < u_1(t_3)$, contradicting \eqref{q-u}.
A similar argument holds for $q_2$.
\qed

\paragraph{\sc Proof of Corollary \ref{corStat}}
If $x^* \in \rS^b$, then the solution to \eqref{EqBalance2} will have $0 \le z \le m_2$, where the exact value of $x^*$ is
readily seen to be the one in $(i)$.
If $x^* \in  \rS^+$, then $q^*_1 - r q^*_2 > \kappa$, so that $\pi^*_{1,2} = 1$.
Plugging $\pi^*_{1,2} = 1$ in the ODE for $z_{1,2}(t)$ in \eqref{odeDetails}, we get $\dot{z}_{1,2}(t) = z_{2,2}(t)\mu_{2,2}$.
Since at stationarity $\dot{z}_{1,2}(t) = 0$, it follows that
$z^*_{2,2} = 0$, which implies that $z^*_{1,2} = m_2$. Plugging this value of $z^*_{1,2}$, together with $\pi^*_{1,2} = 1$ when
$\dot{q}_i(t) = 0$, $i = 1,2$, we get the values of $q^*_1$ and $q^*_2$ as in $(ii)$.

Finally, if $x^* \in \rS^-$, i.e., if $q^*_1 - r q^*_2 < \kappa$, then $\pi^*_{1,2} = 0$, so that, by plugging this value of
$\pi^*_{1,2}$ in the ODE for $z_{1,2}(t)$ in \eqref{odeDetails}, we see that $\dot{z}_{1,2}(t) = \mu_{1,2}z_{1,2}(t)$. Equating to zero,
to get the value at stationarity, we see that $z^*_{1,2} = 0$. Plugging $\pi^*_{1,2} = 0$ and $z^*_{1,2} = 0$ in the ODE for
$q_1(t)$ and $q_2(t)$, and equating these to zero, we get the values in $(iii)$.
\qed

\paragraph{\sc Proof of Corollary \ref{corRegions}}
We prove $(i)$ only. The proofs for $(ii)$ and $(iii)$ are similar.
First assume that $x^* \in \rS^b$. Since $z^*_{1,2} \ge 0$, It follows from the expression for $z^*_{1,2}$ in $(i)$
of Corollary \ref{corStat} that if $q^a_2 \ge 0$ then $q^a_1 - \kappa \ge r q^a_2$. If $s^a_2 > 0$ then
$q^a_1 - \kappa \ge \mu_{1,2} s^a_2 / \theta_1$ by Assumption A. For the other inequality we use the fact that
\bes
z^*_{1,2} = \frac{\theta_1 \theta_2 (q^a_1 - \kappa) - r \theta_1 (\lm_2 - \mu_{2,2} m_2)}{r \theta_1\mu_{2,2} + \theta_2\mu_{1,2}} \le m_2,
\ees
which implies the right-hand inequality in \eqref{RbSta}.

Now Assume that \eqref{RbSta} holds. It follows from the right-hand-side (RHS) inequality and the expression of $z$ in \eqref{z} that
\bes
\bsplit
z & \equiv
\frac{\theta_1 \theta_2 (q^a_1 - \kappa) - r \theta_1 (\lm_2 - \mu_{2,2} m_2)}{r \theta_1\mu_{2,2} + \theta_2\mu_{1,2}} \\
& \le \frac{\theta_1 \theta_2 (r \lm_2 / \theta_2 + \mu_{1,2} m_2 / \theta_1) - r \theta_1 (\lm_2 - \mu_{2,2} m_2)}
{r \theta_1 \mu_{2,2} + \theta_2\mu_{1,2}}
 = m_2.
\end{split}
\ees
From the left-hand inequality in \eqref{RbSta}, we see that, if $s^a_2 = 0$ (and necessarily $q^a_2 \ge 0 = s^a_2$), then
\bes
z \ge \frac{\theta_1 \theta_2 r q^a_2 - r \theta_1 (\lm_2 - \mu_{2,2} m_2)}{r \theta_1\mu_{2,2} + \theta_2\mu_{1,2}} = 0.
\ees
If $s^a_2 > 0$ (and $q^a_2 = 0$), then
\bes
z \ge \frac{\theta_2 \mu_{1,2} s^a_2 - r \theta_1 (\lm_2 - \mu_{2,2} \lm_2)}{r \theta_1\mu_{2,2} + \theta_2\mu_{1,2}} =
\frac{\theta_2 \mu_{1,2} s^a_2 + r \theta_1 \mu_{2,2} s^a_2}{r \theta_1\mu_{2,2} + \theta_2\mu_{1,2}} = s^a_2.
\ees
Thus, if \eqref{RbSta} holds, then $s^a_2 \le z \le m_2$. This was shown to to imply that $x^* \in \rS^b$ in the proof of
Theorem \ref{thODEsteady}. (In fact, we have a stronger result, since we have $z \ge s^a_2$. This is due to the requirement that
$q^a_1 - \kappa \ge \mu_{1,2}s^a_2 / \theta_1$, which is exactly Condition $(I)$ in Assumption A.)

We can show that the inequalities in \eqref{RbSta} are strict if and only if $x^* \in \AA$ by
first observing that the inequalities are strict if and only if $0 < z^* < m_2$, and then directly calculate the QBD drift rates
at the point $x^*$. This is done in \S \ref{secSufficient}; see \eqref{SSrates}. It then follows that \eqref{posrec} holds at $x^*$ if and only if
$0 < z^* < m_2$.
\qed

\section{Conclusions and Further Research} \label{secCon}

In this paper we analyzed the deterministic ODE \eqref{ode}-\eqref{odeDetails},
arising as the MS-HT fluid limit of the overloaded X call-center model operating under the FQR-T control.
In addition to being an interesting mathematical object in its own right, the ODE analyzed in
this paper is a vital part of the FWLLN and FCLT in \cite{PeW10a,PeW10b}.
We prove that the stationary point point $x^*$, which was developed heuristically in \cite{PeW09}
using flow-balance arguments,
is indeed the unique stationary point for the ODE.
Moreover, we provided mild conditions under which the solution $x(t)$ converges to $x^*$ as $t \ra \infty$. We also showed that the
convergence to $x^*$ is exponentially fast, further justifying the steady-state analysis in \cite{PeW09}.

We showed that the existence of a unique solution to the IVP \eqref{IVP} depends heavily on the characterization of
the function $\Psi$ in \eqref{ode} and its topological properties. These properties, in turn, depend
on the state space of $\Psi$, and the regions of the state space in which $\Psi$ is continuous. These regions are further
characterized by the probabilistic properties of the family of FTSP's $\{D_t : t \ge 0 \}$.

To further relate to the model considered in our previous paper \cite{PeW09},
in \S \ref{secTransient} we considered the system at the time when the arrival rates first change.
At that time, the system will typically be underloaded, so that the state space should not be $\rS$.
After the change, we assume that
the arrival rates are larger than the total service rate of the two pools.
Specifically, we assumed Assumption A in \S \ref{secStat}.
We then considered the first transient
period $[0, T)$, where $T$ is the time at which $\rS^b$ is hit.
Using alternative fluid models (ODE's), we showed that $T < \infty$, under the conditions of Theorem \ref{thT}.
The solutions to the fluid models during the first transient period are all exponential functions, so that this
period also passes exponentially fast.

Finally, we developed an efficient algorithm to solve the IVP \eqref{IVP}, based on the matrix geometric method.
This algorithm solves the different fluid models described in \S \ref{secTransient}, and combines these solutions with
the solution to \eqref{odeDetails} once the set $\rA$ is hit, where the AP takes place.

It remains to quantify or at least bound the number of times the fluid solution moves from one of the regions $\rS^+$, $\rS^-$ or $\rS^b$ to
one of the others.
Of course, the complexity of a solution is constrained by the fact that the solution path cannot cross over itself.
It also remains to consider more complicated dynamics than provided by a single change in the arrival rates.  The
numerical algorithm applies more generally, but it remains to establish mathematical results and examine the performance.
For example, it remains to consider a
second overload incident happening before the system has recovered from the first one.
Finally, it remains to establish analogs of the results here for more complex models, e.g., with more than two classes and/or more than two service pools.
\end{appendix}

\end{document}